\documentclass[12pt]{article}
\usepackage{array}
\usepackage{amsthm,amsmath,amssymb,color}
\usepackage{fullpage}
\usepackage{graphicx}
\usepackage{float}
\usepackage{tikz}
\usetikzlibrary{shapes.geometric,calc}

\newtheorem{thm}{Theorem}[section]

\newtheorem{lem}[thm]{Lemma}
\newtheorem{prop}[thm]{Proposition}
\newtheorem{obs}[thm]{Observation}
\newtheorem{cor}[thm]{Corollary}

\newtheorem{cons}[thm]{Construction}

\theoremstyle{remark}
\newtheorem{rem}[thm]{Remark}

\theoremstyle{plain}

\newcommand{\HH}{\mathcal H}

\newcommand{\CC}{\mathcal C}

\newcommand{\R}{\mathbb R}
\newcommand{\shad}{\partial_2}

\begin{document}
	
	\title{The maximum spectral radius of uniform hypergraphs\\ whose shadow excludes a complete\\ or complete bipartite minor}
	\author{Pei Liu\thanks{Department of Mathematics, Sungkyunkwan University, Suwon, 16419, Republic of Korea. liupei2023@g.skku.edu. Research supported by the China Scholarship Council.} \and Suil O\thanks{Department of Applied Mathematics and Statistics, The State University of New York, Korea, Incheon, 21985, suil.o@sunykorea.ac.kr. Corresponding author. Research supported by the National Research Foundation of Korea (NRF) grant funded by the Korea government(MSIT) No. RS-2025-23523950.}}
	\date{\today}
	\maketitle
	
	\begin{abstract}
		For a $k$-uniform hypergraph $\HH$, the \emph{shadow} of $\HH$ is the graph whose edges are the pairs covered by a hyperedge. In this paper, for all sufficiently large $n$, we determine the $n$-vertex $k$-uniform hypergraphs of maximum adjacency-tensor spectral radius whose shadow has no $K_t$ minor, for every $t\ge k+1$, and those whose shadow has no $K_{s,t}$ minor, for every $2\le s\le t$ with $s+t\ge k+1$ and every residue of $n-s+1$ modulo $t$; outside these ranges the problems are trivial. In each case the extremal hypergraph is unique, and it is the $k$-clique hypergraph of the join of a clique with a graph that we call the light part. For $K_{s,t}$ the answer depends on $j=k-s+1$. When $j\le1$, the maximum has order $n^{(k-1)/k}$, and the light part is the one found by Zhai and Lin for the adjacency matrix, including its exceptional components. When $j\ge2$, a regime that does not occur for graphs, the maximum has order $n^{(s-1)/k}$ and $t$ enters its leading constant. The light part then consists of copies of $K_t$ and one smaller clique, with a single exception: for $(k,s,t)=(9,8,8)$ and $n-s+1\equiv2\pmod 8$, the complement of the Petersen graph appears. When the smaller clique has between $1$ and $j-1$ vertices, the extremal graph is not unique. In particular, for $t=8$, $4\le s\le7$, $k=s+1$ and $n-s+1\equiv2\pmod 8$, the clique hypergraph of the extremal graph of Zhai and Lin is not extremal. For $j\ge2$ the light part is determined by a weighted clique inequality, which for $j\ge3$ follows from a weighted form of the closed-neighborhood counting of Chao and Dong.
	\end{abstract}
	
	\bigskip
	\noindent
	\textbf{Keywords:} $k$-uniform hypergraph, spectral radius, adjacency tensor, clique tensor, graph minor, spectral extremal problem\\
	\textbf{AMS subject classification 2020:} 05C65, 05C50, 05C83
	
	\section{Introduction}\label{sec:intro}
	
	Extremal problems for the spectral radius of a graph have a long history; see~\cite{W} for terminology. A typical problem fixes a family $\mathcal G$ of $n$-vertex graphs and asks which member maximizes the largest eigenvalue $\lambda_1$ of the adjacency matrix. For graphs with an excluded minor the relevant constructions are classical. Mader~\cite{M} proved that a graph with no fixed minor has at most linearly many edges, and the natural $K_r$-minor-free and $K_{s,t}$-minor-free graphs of many edges are, respectively, the join of $K_{r-2}$ with an independent set and the join of $K_{s-1}$ with a disjoint union of copies of $K_t$; these are edge-extremal for small $r$ and $s$ but not in general~\cite{K,Th,KP}.
	
	For the spectral radius the picture is cleaner. Boots and Royle~\cite{BR}, and independently Cao and Vince~\cite{CV}, conjectured that the planar graph on $n\ge9$ vertices of maximum spectral radius is $K_2\vee P_{n-2}$, and Cvetkovi\'c and Rowlinson conjectured that the outerplanar graph of maximum spectral radius is $K_1\vee P_{n-1}$. Tait and Tobin~\cite{TT} proved both conjectures for all sufficiently large $n$. Tait~\cite{T} then showed, for all sufficiently large $n$, that the two natural minor-free constructions above are spectrally extremal for every $r$ and every $s\le t$, with the extra hypothesis $t\mid n-s+1$ in the complete bipartite case, and that the analogous statement holds for graphs of bounded Colin de Verdi\`ere parameter~\cite{CdV}. The residues $t\nmid n-s+1$ were settled by Zhai and Lin~\cite{ZL}, whose answer shows that the complete-block construction conjectured by Tait is not always extremal. Earlier results in this direction are due to Hong~\cite{H} and Nikiforov~\cite{N2}. The outerplanar and planar conjectures have since been resolved completely by Lin and Ning~\cite{LN} and by Liu, Ning and Wang~\cite{LNW}, respectively.
	
	For hypergraphs the spectral radius is taken from the adjacency tensor of Cooper and Dutle~\cite{CD}, whose Perron--Frobenius theory is developed in~\cite{FGH,N,Q,Lim}. Ellingham, Lu and Wang~\cite{ELW} initiated the study of topologically restricted uniform hypergraphs by determining the outerplanar $3$-uniform hypergraph of maximum spectral radius, and the outerplanar and planar problems have been settled for every $k$ in~\cite{LO}. A parallel line of work starts from the $k$-clique tensor of a graph $G$ introduced by Liu and Bu~\cite{LB}, whose entry indexed by $(i_1,\dots,i_k)$ is $1/(k-1)!$ when $\{i_1,\dots,i_k\}$ is a $k$-clique of $G$ and $0$ otherwise; Liu and Bu proved a clique-tensor version of Mantel's theorem, and Liu, Zhou and Bu~\cite{LZB} studied spectral extremal problems for this tensor with applications to Hadwiger's conjecture. The $k$-clique tensor of $G$ is exactly the Cooper--Dutle adjacency tensor of the hypergraph $\CC_k(G)$ defined below, and by Lemma~\ref{lem:reduce} the problem considered here is the same as an extremal problem for the $k$-clique tensor of a minor-free graph, so the two lines meet.
	
	All graphs are simple. Unless otherwise stated, $k\ge3$, all asymptotic statements are for fixed parameters as $n\to\infty$, and uniqueness is up to isomorphism. For a $k$-uniform hypergraph $\HH$ on $n$ vertices we use the adjacency tensor of Cooper and Dutle, whose spectral radius has the variational description
	\begin{equation}\label{eq:var}
		\rho(\HH)=\max\Bigl\{\,k\sum_{F\in E(\HH)}\ \prod_{v\in F}x_v\ :\ x\in\R^{V(\HH)}_{\ge0},\ \sum_v x_v^{k}=1\,\Bigr\},
	\end{equation}
	attained at a normalized \emph{Perron vector}, which is positive when $\HH$ is connected. The \emph{shadow} of $\HH$ is the graph $\shad\HH$ on $V(\HH)$ with $E(\shad\HH)=\{\,\{u,v\}:\{u,v\}\subseteq F\text{ for some }F\in E(\HH)\,\}$, and $\vee$ denotes the join.
	
	We work with the following classes. For a graph $M$ let
	\[
	\HH_M^{(k)}(n)=\{\,\HH:\ |V(\HH)|=n,\ \HH\ k\text{-uniform},\ \shad\HH\ \text{is }M\text{-minor-free}\,\},
	\]
	and for a graph $G$ let $\CC_k(G)$ be the $k$-uniform hypergraph whose edges are the $k$-cliques of $G$.
	
	Since a hyperedge spans a copy of $K_k$ in the shadow, and $K_k$ contains $K_t$ when $t\le k$ and $K_{s,t}$ when $s+t\le k$, the classes above are trivial unless $t\ge k+1$ in the complete case and $s+t\ge k+1$ in the complete bipartite case. Within those ranges our results are as follows. Throughout, $j:=k-s+1$.
	
	\begin{thm}\label{thm:introKr}
		Let $k\ge3$, $t\ge k+1$ and $a=t-2$, and assume that $n\ge a+1$. Let $B^{(k)}_{n,t}=\CC_k\bigl(K_a\vee\overline{K_{n-a}}\bigr)$. Then $B^{(k)}_{n,t}\in\HH^{(k)}_{K_t}(n)$, its spectral radius is the unique root in $\bigl(\binom{a-1}{k-1},\infty\bigr)$ of
		\[
		\lambda\Bigl(\lambda-\binom{a-1}{k-1}\Bigr)^{k-1}=\binom{a-1}{k-2}^{k-1}\binom{a}{k-1}(n-a)^{k-1},
		\]
		and $\rho\bigl(B^{(k)}_{n,t}\bigr)=\bigl(c_{k,t}+o(1)\bigr)n^{(k-1)/k}$ with $c_{k,t}=\binom{t-3}{k-2}^{(k-1)/k}\binom{t-2}{k-1}^{1/k}$. For $t=k+1$ the equation degenerates to $\lambda^{k}=(n-k+1)^{k-1}$. Moreover, for all sufficiently large $n$, $B^{(k)}_{n,t}$ is the unique hypergraph of maximum spectral radius in $\HH^{(k)}_{K_t}(n)$.
	\end{thm}
	
	The equation for the spectral radius is Proposition~\ref{prop:Kreq}, and the extremality statement is Theorem~\ref{thm:Kr}. The extremal graph
	$K_{t-2}\vee\overline{K_{n-t+2}}$ is the same for every $3\le k\le t-1$
	and agrees with the extremal graph in Tait's theorem for $k=2$. The reason is structural: the part of the graph outside the dominating clique is an independent set, so every hyperedge has at most one vertex there and $k$ never enters the interaction among those vertices.
	
	For complete bipartite minors the answer depends on $j$. We first record the order of the maximum and its leading constant.
	
	\begin{thm}\label{thm:introbip}
		Let $2\le s\le t$ and $t\ge j$, and put $h=n-s+1$ and $b=\lfloor h/t\rfloor$, so that $K_{s-1}\vee bK_t$ (Figure~\ref{fig:cand}) has $n-(h-bt)$ vertices. Then
		\[
		\rho\bigl(\CC_k(K_{s-1}\vee bK_t)\bigr)=\bigl(c+o(1)\bigr)\,n^{\frac{\min(s-1,\,k-1)}{k}},\qquad
		c=\begin{cases}
			\displaystyle\binom{s-1}{k-1}\Bigl(\frac{k-1}{s-1}\Bigr)^{\frac{k-1}{k}}, & j\le1,\\[2ex]
			\displaystyle\frac1t\binom tj\,j^{\,j/k}, & j\ge2,
		\end{cases}
		\]
		and this is the maximum of $\rho(\HH)$ over $\HH\in\HH^{(k)}_{K_{s,t}}(n)$ up to a factor $1+o(1)$. The mass carried by the $s-1$ dominating vertices in the Perron vector tends to $\frac{\min(s-1,k-1)}{k}$.
	\end{thm}
	
	When $j\ge2$ the order of the maximum drops from $n^{(k-1)/k}$ to $n^{(s-1)/k}$, because a $k$-clique can use at most $s-1$ dominating vertices and must therefore take at least $j$ vertices from a single block. The constant $\frac1t\binom tj$ is, by Lemma~\ref{lem:count}, the largest number of $j$-cliques per vertex in a graph of maximum degree $t-1$. When $j\le1$ the leading term comes from cliques with a single vertex outside the dominating clique, so $t$ does not enter it; for graphs this is why $t$ appears only in lower-order terms of the bound of Tait. Since $k=2$ forces $s\ge2=k$, the regime $j\ge2$ does not occur for graphs.
	
	For the extremal hypergraphs, write $h=n-s+1=bt+p$ with $0\le p<t$, and put $q=\lfloor\frac{t+1}{s+1}\rfloor$. For $q\ge2$ let $F_{s,t}$ be the complement of $(q-1)K_{1,s}\cup K_{1,t-(q-1)(s+1)}$, a graph on $t+1$ vertices; for $q=2$ let $D_{s,t}$ be obtained from $F_{s,t}$ by subdividing the edge joining the centers of the two stars; and let $\overline P$ be the complement of the Petersen graph.
	
	\begin{thm}\label{thm:introresidue}
		Let $3\le k\le s\le t$. For all sufficiently large $n$ the unique hypergraph of maximum spectral radius in $\HH^{(k)}_{K_{s,t}}(n)$ is $\CC_k(K_{s-1}\vee H^{*})$, and $K_{s-1}\vee H^*$ is the unique extremal graph, where
		\[
		H^{*}=\begin{cases}
			(b-1)K_8\cup\overline P, & q=1,\ t=8,\ p=2,\\
			(b-1)K_t\cup D_{s,t}, & q=2,\ p=2,\\
			(b-p)K_t\cup pF_{s,t}, & 1\le p\le2q-2,\ (q,p)\ne(2,2),\\
			bK_t\cup K_p, & \text{otherwise.}
		\end{cases}
		\]
	\end{thm}
	
	\begin{thm}\label{thm:introkgs}
		Let $2\le s\le\min(k-1,t)$ and $t\ge j$. For all sufficiently large $n$ the unique hypergraph of maximum spectral radius in $\HH^{(k)}_{K_{s,t}}(n)$ is $\CC_k(K_{s-1}\vee H^{*})$, where
		\[
		H^{*}=\begin{cases}
			(b-1)K_8\cup\overline P, & (k,s,t,p)=(9,8,8,2),\\
			bK_t\cup K_p, & \text{otherwise.}
		\end{cases}
		\]
		If $p=0$ or $p\ge j$, then $K_{s-1}\vee H^{*}$ is the unique extremal graph. If $0<p<j$, then in every extremal graph the $p$ vertices that lie neither in the dominating clique nor in a copy of $K_t$ belong to no $k$-clique, and the extremal graph is not unique.
	\end{thm}
	
	The graph $H^*$ of Theorem~\ref{thm:introresidue} is exactly the light part in the adjacency-matrix classification of Zhai and Lin~\cite[Theorem~1.2]{ZL}, so for $j\le1$ the extremal structure does not depend on $k$. Theorem~\ref{thm:introkgs} is different: when $j\ge2$ the exceptional components $F_{s,t}$ and $D_{s,t}$ never occur, and $\overline P$ occurs only for $(k,s,t)=(9,8,8)$. In particular, for $t=8$, $4\le s\le7$, $k=s+1$ and $p=2$, the clique hypergraph of the extremal graph of Zhai and Lin, $K_{s-1}\vee((b-1)K_8\cup\overline P)$, is not extremal.
	
	The two theorems rest on different comparisons of light parts. When $j\le1$, every $k$-clique of the join may use a single light vertex, the light coordinates of the Perron vector are asymptotically equal, and light parts of the same order are compared first by their numbers of edges and then by $D_2+\theta c_3$, a combination of the degree-square sum and the triangle count with a weight $\theta>0$ depending on $k$; this replaces the degree-sequence majorization used by Zhai and Lin, and the component analysis follows the outline of~\cite[Sections~3 and~4]{ZL}. When $j\ge2$, the light coordinates are no longer equal, and light parts are compared through a capacity $\mathcal E_j$, the suitably normalized maximum of the weighted $j$-clique sum. For $j\ge3$ we prove the extremal inequality for $\mathcal E_j$ under the maximum-degree condition alone, using a weighted form of the closed-neighborhood counting of Chao and Dong~\cite{ChD}; for $j=2$ the corresponding comparison fails, and the minor conditions and a classification of small connected light components are needed. Both inequalities have a gap that does not depend on the number of complete blocks, and this uniformity makes an exact spectral comparison possible for every residue.
	
	The structural input for $j\ge2$ is a stability theorem. The Perron vector of an extremal graph concentrates on $s-1$ dominating vertices, and all but boundedly many vertices are adjacent to all of them; Proposition~\ref{prop:Rempty} removes the remaining vertices in the divisible case by a comparison with the candidate at the Perron vector, and Proposition~\ref{prop:Rbdd} bounds their number for every residue. For graphs, Tait~\cite{T} identified the extremal graph using interlacing and the edge-density theorem of Chudnovsky, Reed and Seymour~\cite{CRS}; for $k\ge3$ the clique terms of every size up to $k$ make these tools unnecessary.
	
	Section~\ref{sec:tools} collects the tools. Section~\ref{sec:Kr} treats complete minors and the leading asymptotics when $j\le1$, and Section~\ref{sec:bip} evaluates the bipartite candidate. Section~\ref{sec:stab} contains the stability argument, the bound on the exceptional set, and the proof of Theorem~\ref{thm:introresidue}. Section~\ref{sec:weighted} proves the weighted clique inequalities, and Section~\ref{sec:kgs} proves Theorem~\ref{thm:introkgs}.

	\section{Definitions and tools}\label{sec:tools}

		Let $M$ and $G$ be graphs. An \emph{$M$-model} in $G$ is a family
		$\{B_v:v\in V(M)\}$ of pairwise disjoint nonempty subsets of $V(G)$
		such that each $G[B_v]$ is connected and, for every $uv\in E(M)$,
		there is an edge of $G$ with one endpoint in $B_u$ and the other
		in $B_v$. The sets $B_v$ are called the \emph{branch sets} of the
		model. Such a model exists if and only if $M$ is a \emph{minor}
		of $G$, meaning that $M$ can be obtained from $G$ by deleting
		vertices, deleting edges, and contracting edges. Edges between
		$B_u$ and $B_v$ are allowed even when $uv\notin E(M)$.

		We call the vertices of a designated \emph{core} set $A$
		the \emph{hubs}. In a join $K_a\vee H$, the core is
		$A=V(K_a)$, so every hub is adjacent to every other vertex
		of the join. In the stability arguments, the core is selected
		from the vertices with the largest coordinates of the vector
		under consideration. Calling these vertices hubs does not
		by itself assert that they form a clique or are adjacent
		to all remaining vertices.

	\begin{lem}\label{lem:reduce}
		Let \(n\ge1\) and \(k\ge2\), and let \(M\) be a graph for which
		there exists an \(n\)-vertex \(M\)-minor-free graph. Then
		\[
		\max_{\HH\in\HH^{(k)}_M(n)}\rho(\HH)\ =\ \max\{\,\rho(\CC_k(G))\ :\ G\ \text{an $n$-vertex $M$-minor-free graph}\,\},
		\]
		and some extremal hypergraph is of the form $\CC_k(G)$. Every extremal hypergraph $\HH$ has a completion $\CC_k(\shad\HH)$ with the same spectral radius; if the vertices covered by the edges of this completion span a connected subhypergraph, then $\HH=\CC_k(\shad\HH)$.
	\end{lem}
	
	\begin{proof}
		Let $\HH\in\HH^{(k)}_M(n)$ and put $G=\shad\HH$. Every hyperedge of $\HH$ spans a $k$-clique of $G$, so $\HH\subseteq\CC_k(G)$; moreover $\shad\CC_k(G)\subseteq G$, so $\CC_k(G)\in\HH^{(k)}_M(n)$. Adding hyperedges does not decrease the right side of~\eqref{eq:var}, so $\rho(\CC_k(G))\ge\rho(\HH)$. Conversely, every $M$-minor-free graph $G$ gives an admissible hypergraph $\CC_k(G)$, proving equality of the maxima. If $\HH$ is extremal then $\rho(\HH)=\rho(\CC_k(G))$, and under the stated connectivity condition $\HH=\CC_k(G)$ by Observation~\ref{obs:strict} below.
	\end{proof}
	
	The following observation is used repeatedly, and is the reason that comparisons of two hypergraphs on the same vertex set may be made at a single vector even when that vector has zero coordinates.
	
	\begin{obs}\label{obs:strict}
		Let $\HH\subsetneq\HH'$ be $k$-uniform hypergraphs on the same vertex set, and suppose that the vertices covered by the edges of $\HH'$ span a connected subhypergraph. Then $\rho(\HH')>\rho(\HH)$.
	\end{obs}
	
	\begin{proof}
		A vertex lying in no edge of $\HH'$ contributes to the normalization in~\eqref{eq:var} but not to the form, for $\HH'$ and for $\HH$ alike, so deleting all such vertices changes neither spectral radius. We may therefore assume that $\HH'$ is connected. Let $x$ be a Perron vector of $\HH$, so that $\rho(\HH')\ge kP_{\HH'}(x)\ge kP_{\HH}(x)=\rho(\HH)$, where $P_{\HH}(x)=\sum_{F\in E(\HH)}\prod_{v\in F}x_v$. If equality held then $x$ would attain the maximum in~\eqref{eq:var} for $\HH'$, hence be a Perron vector of $\HH'$ and therefore positive; but then $\prod_{v\in F}x_v>0$ for $F\in E(\HH')\setminus E(\HH)$, a contradiction.
	\end{proof}

	\begin{lem}\label{lem:common}
		Let $G$ be $K_{s,t}$-minor-free. Then any $s$ vertices of $G$ have at most $t-1$ common neighbors. In particular, if $A_1\ne A_2$ are $(s-1)$-subsets of $V(G)$, then $|N(A_1)\cap N(A_2)|\le t-1$, where $N(A)$ denotes the set of common neighbors of $A$.
	\end{lem}
	
	\begin{proof}
		If $s$ vertices had $t$ common neighbors, those $s+t$ vertices would span a $K_{s,t}$ subgraph. For the second statement, $|A_1\cup A_2|\ge s$; choosing $B\subseteq A_1\cup A_2$ with $|B|=s$ gives $N(A_1)\cap N(A_2)\subseteq N(B)$.
	\end{proof}
	
	The second statement is what forces the set of hubs to be essentially unique: two different candidate hub sets cannot both have large common neighborhoods overlapping in more than a bounded set. The first statement, applied to a hub set enlarged by one of the vertices it dominates, bounds the degrees inside a common neighborhood.
	
	\begin{cor}\label{cor:deg}
		Let $G$ be $K_{s,t}$-minor-free and let $A$ be a set of $s-1$ vertices with common neighborhood $N(A)$. Then the graph $G[N(A)]$ has maximum degree at most $t-1$.
	\end{cor}
	
	\begin{proof}
		Let $v\in N(A)$. Its neighbors inside $N(A)$ are common neighbors of the $s$ vertices of $A\cup\{v\}$, so there are at most $t-1$ of them by Lemma~\ref{lem:common}.
	\end{proof}
	
	Corollary~\ref{cor:deg} is what makes the whole bipartite analysis go through: it supplies, for free, the degree hypothesis under which Lemmas~\ref{lem:count} and~\ref{lem:sigma} are sharp.

	Tait uses the implication that the part of the graph outside the hubs has no $K_{1,t}$ minor. That condition is necessary but not sufficient, and for the purpose of describing the extremal family we need the exact criterion. This criterion appears in Zhai and Lin~\cite{ZL}, following their Lemma~2.2; we include a self-contained proof.
	
	\begin{lem}\label{lem:join}
		Let $2\le s\le t$ and let $H$ be a graph. Then $K_{s-1}\vee H$ has a $K_{s,t}$ minor if and only if $H$ has a $K_{r+1,\,t-r}$ minor for some $r$ with $0\le r\le s-1$.
	\end{lem}
	
	\begin{proof}
		Every vertex of the join $K_{s-1}$ is adjacent to all other vertices, so a branch set meeting $K_{s-1}$ is adjacent to every other branch set; replacing such a branch set by a single one of its hub vertices preserves connectedness and all adjacencies. We may therefore assume that the branch sets meeting the hubs are singletons, say $q\le s-1$ of them.
		
		Suppose a $K_{s,t}$ minor is given, with $r$ hub singletons on the side of size $t$ and $q-r$ on the side of size $s$. The remaining branch sets lie in $H$: there are $s-q+r$ of them on the first side and $t-r$ on the second, and all adjacencies between the two groups are edges of $H$. Hence $H$ has a $K_{s-q+r,\,t-r}$ minor, and $s-q+r\ge r+1$ because $q\le s-1$, so $H$ has a $K_{r+1,t-r}$ minor.
		
		Conversely, given a $K_{r+1,t-r}$ minor of $H$ with $0\le r\le s-1$, add $s-1-r$ hubs as branch sets on the first side and $r$ hubs on the second. The first side then has $(s-1-r)+(r+1)=s$ branch sets and the second has $r+(t-r)=t$, and all required adjacencies hold.
	\end{proof}
	
	Taking $r=0$ recovers the condition used by Tait~\cite{T}; the other cases are not implied by it, since the cycle $C_5$ has no $K_{1,3}$ minor while contracting one edge produces $K_{2,2}$, so $K_2\vee C_5$ does have a $K_{3,3}$ minor. An equivalent criterion is stated by Zhai and Lin~\cite{ZL} just after their Lemma~2.2, under the name of the $(s,t)$-property; we include a proof because we use the formulation with $K_{r+1,t-r}$ throughout.
	
	\begin{cor}\label{cor:small}
		If every component of $H$ has at most $t$ vertices, then $K_{s-1}\vee H$ is $K_{s,t}$-minor-free.
	\end{cor}
	
	\begin{proof}
		A $K_{r+1,t-r}$ minor with $r+1\ge1$ and $t-r\ge1$ needs $t+1$ branch sets, pairwise joined across the two sides, hence all inside one component.
	\end{proof}

	\begin{lem}\label{lem:cliquesum}
		Let $M$ be a $q$-connected graph. If $G_1,G_2$ are $M$-minor-free and their intersection is a clique $S$ of order less than $q$, with no edges between $V(G_1)\setminus S$ and $V(G_2)\setminus S$, then their union is $M$-minor-free.
	\end{lem}
	
	\begin{proof}
		Suppose a model of $M$ exists in the union. At most $|S|$ branch sets meet $S$. Delete the corresponding vertices of $M$; the remaining graph is connected because $|S|<q$. Every remaining branch set lies in one side, and no edge joins different sides, so all of them lie in the same side, say $G_1-S$. For each branch set that meets $S$, keep its intersection with $G_1$. This intersection is connected using the edges of the clique $S$, since every component of the intersection meets $S$. Adjacencies to branch sets in $G_1-S$ are retained, and any two retained branch sets meeting $S$ are adjacent through the clique. This gives an $M$-model in $G_1$, a contradiction.
	\end{proof}

	\begin{lem}\label{lem:Kst}
		Let $G$ be $K_{s,t}$-minor-free and let $B\subseteq V(G)$ induce $K_{s+t-1}$. Then every component of $G-B$ has at most $s-1$ neighbors in $B$.
	\end{lem}

	\begin{proof}
		If a component $D$ of $G-B$ had $s$ neighbors in $B$, these $s$ vertices as singletons on one side, and $D$ together with the other $t-1$ vertices of $B$ as singletons on the other, would form a $K_{s,t}$-model.
	\end{proof}

	\begin{lem}\label{lem:count}
		Let $H$ satisfy $\Delta(H)\le t-1$, which holds in particular when $H$ is $K_{1,t}$-minor-free, and let $j\ge2$. Then
		\[
		|E(\CC_j(H))|\ \le\ \frac1t\binom tj\,|V(H)| .
		\]
		For $3\le j\le t$ equality holds if and only if every component of $H$ is $K_t$. For $j=2$ equality holds if and only if $H$ is $(t-1)$-regular. For $j>t$ both sides are zero.
	\end{lem}
	
	\begin{proof}
		A vertex of degree at least $t$ together with $t$ of its neighbors is a $K_{1,t}$ subgraph, so the second hypothesis implies the first. If $j>t$ both sides vanish, so assume $j\le t$. The number of $j$-cliques containing a fixed $v$ is the number of $(j-1)$-cliques of $H[N(v)]$, which is at most $\binom{d(v)}{j-1}\le\binom{t-1}{j-1}$, with equality if and only if every $(j-1)$-subset of $N(v)$ is a clique and $d(v)=t-1$. Counting incidences, $j|E(\CC_j(H))|\le|V(H)|\binom{t-1}{j-1}$, and $\frac1j\binom{t-1}{j-1}=\frac1t\binom tj$. For $j\ge3$ the equality condition says that $H[N(v)]$ is complete for every $v$, so each component is $K_t$; for $j=2$ it says only that $d(v)=t-1$.
	\end{proof}
	
	\begin{lem}\label{lem:sigma}
		Let $F$ be a graph with nonnegative vertex weights $x$, let
		$W=\sum_v x_v^k$, and let $1\le j\le k-1$.
		For $v\in V(F)$ put
		$m_v=|\{Q\in E(\CC_j(F)):v\in Q\}|$,
		the number of $j$-cliques of $F$ containing $v$. Then
		\[
		\sum_{Q\in E(\CC_j(F))}
		\prod_{v\in Q}x_v
		\le
		\frac1j\,W^{j/k}
		\Bigl(\sum_v m_v^{\frac{k}{k-j}}\Bigr)^{\frac{k-j}{k}}.
		\]
		If moreover $\Delta(F)\le t-1$, then
		$m_v\le\binom{t-1}{j-1}$ and the right side is at most
		\[
		\frac1t\binom tj\,W^{j/k}\,
		|V(F)|^{\frac{k-j}{k}}.
		\]
	\end{lem}
	
	\begin{proof}
		By the arithmetic-geometric mean inequality,
		$\prod_{v\in Q}x_v\le\frac1j\sum_{v\in Q}x_v^j$,
		so the left side is at most $\frac1j\sum_v m_vx_v^j$.
		Apply H\"older with exponents $k/j$ and $k/(k-j)$.
		For $j=1$, $m_v=1$ for every $v$, and
		the statement reads
		$\sum_vx_v\le W^{1/k}|V(F)|^{(k-1)/k}$,
		which is H\"older.
	\end{proof}
	
	The next proposition compares an arbitrary light part with a union of copies of $K_t$ term by term, at the same vector; it is the basic comparison in the stability argument of Section~\ref{sec:stab}.
	
	\begin{prop}\label{prop:termwise}
		Let $2\le s\le t$, and let
		$H$ be a graph on $h>0$ vertices with
		$\Delta(H)\le t-1$. Let $y$ be a nonnegative vector on
		$V(K_{s-1}\vee H)$, and let $A=V(K_{s-1})$.
			Define $u\ge0$ and $W$ by
		\[
		u^k=\frac1{s-1}\sum_{a\in A}y_a^k,
		\qquad
		W=\sum_{v\in V(H)}y_v^k.
		\]
		Then
		\begin{equation}\label{eq:termwise}
			\begin{aligned}
				\sum_{F\in E(\CC_k(K_{s-1}\vee H))}
				\prod_{v\in F}y_v
				\le{}&
				\binom{s-1}{k}u^k+\sum_{r=1}^{k}\binom{s-1}{k-r}u^{k-r}
				\frac1t\binom tr\,W^{r/k}h^{(k-r)/k}.
			\end{aligned}
		\end{equation}
		If $t\mid h$, the right side is the value of the left side
		for $H'=\frac htK_t$ and the vector $y'$ equal to $u$ on
		the dominating vertices and to $(W/h)^{1/k}$ on $V(H')$,
		and $y'$ has the same $k$-norm as $y$.
		If $k\ge3$, $t\ge k-s+1$, $u>0$, and $W>0$,
		equality holds in~\eqref{eq:termwise} only if $y$ is constant
		on the dominating vertices, every component of $H$ is a copy
		of $K_t$, and $y$ is constant on $V(H)$.
	\end{prop}
	
	\begin{proof}
		A $k$-clique of $K_{s-1}\vee H$ consists of $k-r$ dominating
		vertices and an $r$-clique of $H$, for some $0\le r\le k$.
		Separating the cliques lying entirely in $A$,
			the left side equals
		\[
			e_k(y_A)+\sum_{r=1}^{k}e_{k-r}(y_A)
			\sum_{Q\in E(\CC_r(H))}\prod_{v\in Q}y_v,
		\]
		where $e_m(y_A)$ is the elementary symmetric polynomial of
		degree $m$ in the coordinates on the dominating vertices,
		with $e_0(y_A)=1$ and $e_m(y_A)=0$ for
			$m>s-1$.
		For $1\le m\le s-1$, Maclaurin's inequality gives
		\[
		e_m(y_A)\le
		\binom{s-1}{m}
		\left(\frac{e_1(y_A)}{s-1}\right)^m,
		\]
		and the power mean inequality gives $e_1(y_A)/(s-1)\le u$.
		Hence
		\[
		e_{k-r}(y_A)\le\binom{s-1}{k-r}u^{k-r},
		\]
		with equality for $1\le k-r\le s-1$ only if $y$ is constant
		on the dominating vertices.
		
		For $1\le r\le k-1$, Lemma~\ref{lem:sigma} bounds the inner
		sum by
		\[
		\frac1t\binom tr W^{r/k}h^{(k-r)/k},
		\]
		since $\Delta(H)\le t-1$. For $r=k$, the
		arithmetic-geometric mean step of that lemma alone gives
		\[
		\begin{aligned}
			\sum_{Q\in E(\CC_k(H))}
			\prod_{v\in Q}y_v
			&\le \frac1k\sum_{v\in V(H)}m_vy_v^k\\
			&\le \frac1k\binom{t-1}{k-1}W
			=\frac1t\binom tkW,
		\end{aligned}
		\]
		where here $m_v$ denotes the number of
			$k$-cliques of $H$ containing $v$.
		These estimates prove~\eqref{eq:termwise}.
		
		For $H'=\frac htK_t$ and the vector $y'$, one has
		for every $1\le r\le k$
		\[
		|E(\CC_r(H'))|
		=\frac ht\binom tr,
		\qquad
			\sum_{Q\in E(\CC_r(H'))}\prod_{v\in Q}y'_v
		=\frac ht\binom tr(W/h)^{r/k}.
		\]
		Together with the contribution $\binom{s-1}{k}u^k$ from the
		core, this gives the stated value. Moreover,
		\[
			\sum_{v\in V(K_{s-1}\vee H')}(y'_v)^k
			=(s-1)u^k+W
			=\sum_{v\in V(K_{s-1}\vee H)}y_v^k.
		\]
		
		For equality, assume $k\ge3$,
			$t\ge k-s+1$, and $u,W>0$. Put $j=k-s+1$.
			If $j\ge2$, consider the term $r=j$,
		whose coefficient is $\binom{s-1}{s-1}=1$ and for which
		$1\le k-j=s-1$. Equality there forces $y$ to be constant on
		the dominating vertices. In Lemma~\ref{lem:sigma}, applied
		with this $j\le k-1$, write
			$m_v=|\{Q\in E(\CC_j(H)):v\in Q\}|$.
			Equality in
		\[
			\sum_{v\in V(H)}m_v^{k/(k-j)}
			\le h\binom{t-1}{j-1}^{k/(k-j)}
		\]
		forces $m_v=\binom{t-1}{j-1}$ for every $v$.
		Equality in the H\"older step then gives that $y_v^k$
		is proportional to $m_v^{k/(k-j)}$, hence constant on
		$V(H)$.
		
		If $j\ge3$, or $j=2$ and $t=2$, Lemma~\ref{lem:count}
		now makes every component of $H$ a copy of $K_t$.
		If $j=2$ and $t\ge3$, so that $s=k-1$, the lemma gives
		only that $H$ is $(t-1)$-regular. But then the term $r=3$
		is present with coefficient
		$\binom{s-1}{k-3}=k-2>0$, and since $y$ is constant on
		$V(H)$, its inner sum equals
		$|E(\CC_3(H))|(W/h)^{3/k}$.
		Thus equality in~\eqref{eq:termwise} requires
		\[
		|E(\CC_3(H))|
		=\frac ht\binom t3,
		\]
		and Lemma~\ref{lem:count} with $j=3$ identifies the
		components as copies of $K_t$.
		
		If $j\le1$, so that $s\ge k$, consider instead the term
		$r=1$, whose coefficient $\binom{s-1}{k-1}$ is positive
		and for which $2\le k-1\le s-1$. Equality there forces
		$y$ to be constant on the dominating vertices, and
		Lemma~\ref{lem:sigma} with $j=1$ is H\"older's inequality,
		whose equality case gives that $y$ is constant on $V(H)$.
		The terms $r=2$ and $r=3$ are then present with coefficients
		$\binom{s-1}{k-2}$ and $\binom{s-1}{k-3}$, both positive
		since $k\ge3$ and $s\ge k$. The same argument applied first
		to $r=2$ and then to $r=3$ makes $H$ $(t-1)$-regular and
		then a disjoint union of copies of $K_t$.
	\end{proof}
	
	\begin{rem}\label{rem:sharp}
		Both inequalities in Lemma~\ref{lem:sigma} are equalities when
		$F$ is a disjoint union of copies of $K_t$ and $x$ is constant
		on $V(F)$. If $j\le t$, the arithmetic-geometric mean step is
		an equality because $x$ is constant on every clique, and the
		H\"older step is an equality because
		$m_v=\binom{t-1}{j-1}>0$ is constant; the second inequality is
		an equality for the same reason. If $j>t$, both sides of both
		inequalities are zero. In the complete bipartite problem,
		applying the lemma with clique size $\max\{1,k-s+1\}$ gives
		the exponent $\min(s-1,k-1)/k$ appearing in
		Theorem~\ref{thm:bip}.
	\end{rem}
	
	\section{Complete minors, and the regime $j\le1$}\label{sec:Kr}
	
	Fix $t\ge k+1$ and put $a=t-2$. Since a hyperedge spans $K_k$ in the shadow and $K_k$ contains $K_t$ whenever $t\le k$, the class $\HH^{(k)}_{K_t}(n)$ is trivial unless $t\ge k+1$; equivalently $a\ge k-1$.
	
	\begin{cons}\label{cons:Kr}
		Assume that $n\ge a+1$. Let $A$ be a set of $a=t-2$ vertices and $B$ a set of $n-a$ vertices, and let $B^{(k)}_{n,t}$ be the $k$-uniform hypergraph whose edges are the $k$-subsets of $A\cup B$ meeting $A$ in at least $k-1$ vertices. Then $B^{(k)}_{n,t}=\CC_k(K_a\vee\overline{K_{n-a}})$, and its shadow $K_a\vee\overline{K_{n-a}}$ is $K_t$-minor-free: at most $a$ branch sets of a minor model meet $A$, and every other branch set is a single vertex of the independent set $B$, so at most one of them can occur in a $K_t$ model, which therefore has at most $a+1=t-1$ branch sets.
	\end{cons}
	
	\begin{prop}\label{prop:Kreq}
		For $n\ge a+1$, $\lambda=\rho(B^{(k)}_{n,t})$ is the unique root in $\bigl(\binom{a-1}{k-1},\infty\bigr)$ of
		\begin{equation}\label{eq:Kr}
			\lambda\Bigl(\lambda-\binom{a-1}{k-1}\Bigr)^{k-1}=\binom{a-1}{k-2}^{k-1}\binom{a}{k-1}(n-a)^{k-1} .
		\end{equation}
		Consequently
		\[
		\rho(B^{(k)}_{n,t})=\bigl(c_{k,t}+o(1)\bigr)n^{\frac{k-1}{k}},\qquad
		c_{k,t}=\binom{t-3}{k-2}^{\frac{k-1}{k}}\binom{t-2}{k-1}^{\frac1k},
		\]
		and the mass carried by $A$ tends to $\frac{k-1}{k}$.
	\end{prop}
	
	\begin{proof}
		The hypergraph is connected and vertex-transitive on $A$ and on $B$, so the Perron vector has a common value $u$ on $A$ and $v$ on $B$. A vertex of $A$ lies in $\binom{a-1}{k-1}$ edges inside $A$ and in $\binom{a-1}{k-2}(n-a)$ edges meeting $B$; a vertex of $B$ lies in $\binom{a}{k-1}$ edges. The eigenequations are
		\[
		\lambda u^{k-1}=\binom{a-1}{k-1}u^{k-1}+\binom{a-1}{k-2}(n-a)u^{k-2}v,\qquad
		\lambda v^{k-1}=\binom{a}{k-1}u^{k-1}.
		\]
		Eliminating $v/u$ gives~\eqref{eq:Kr}. The left side is strictly increasing from $0$ to $\infty$ on the stated interval, which gives uniqueness, and dividing by $\lambda^{k}$ gives the asymptotics.
		
		The same eigenequations give
		\[
		(n-a)(v/u)^k=\frac{\binom a{k-1}}{\binom{a-1}{k-2}}
		\left(1-\frac{\binom{a-1}{k-1}}{\lambda}\right)
		\longrightarrow\frac a{k-1}.
		\]
		Since $au^k+(n-a)v^k=1$, it follows that
		$au^k=a/[a+(n-a)(v/u)^k]\to(k-1)/k$.
		
	\end{proof}
	
	For $k=3$ equation~\eqref{eq:Kr} reads $\lambda(\lambda-\binom{a-1}{2})^{2}=\frac{a(a-1)^{3}}{2}(n-a)^{2}$, which is the equation of the $3$-uniform case. At the boundary $t=k+1$ one has $a=k-1$, both binomial coefficients on the right equal $1$, and the constant term vanishes, so
	\begin{equation}\label{eq:boundary}
		\rho\bigl(B^{(k)}_{n,k+1}\bigr)^{k}=(n-k+1)^{k-1}\quad\text{exactly.}
	\end{equation}
	
	We now prove that this construction is extremal. Throughout this section $G$ is $K_t$-minor-free, $\HH=\CC_k(G)$, and $x$ is a Perron vector of $\HH$ normalized by $\sum_vx_v^{k}=1$. Since $K_{t-1,t-1}$ contains a $K_t$ minor, obtained by contracting a matching of size $t-2$ and taking the two unmatched vertices as singletons, we record the following two consequences of minor-freeness.
	
	\begin{lem}\label{lem:Krcommon}
		Any $t-1$ vertices of $G$ have at most $t-2$ common neighbors.
	\end{lem}
	
	\begin{proof}
		Otherwise $t-1$ vertices and $t-1$ common neighbors span a $K_{t-1,t-1}$ subgraph.
	\end{proof}
	
	\begin{lem}\label{lem:Krstruct}
		Let $A\subseteq V(G)$ with $|A|=a=t-2$, let $T$ be the set of common neighbors of $A$, and suppose $|T|\ge a+1$. Then $T$ is independent, and every vertex outside $A\cup T$ has at most one neighbor in $T$.
	\end{lem}
	
	\begin{proof}
		Write $A=\{u_1,\dots,u_a\}$. If $v,w\in T$ were adjacent, choose $z_1,\dots,z_{a-1}$ in $T\setminus\{v,w\}$, which is possible since $|T|\ge a+1$. The sets $\{u_1,z_1\},\dots,\{u_{a-1},z_{a-1}\},\{u_a\},\{v\},\{w\}$ are disjoint and connected, and any two of them are joined by an edge because every vertex of $T$ is adjacent to every vertex of $A$ and $vw\in E(G)$. These $a+2=t$ sets form a $K_t$ minor. If instead $r\notin A\cup T$ had two neighbors $v,w\in T$, replace the last two singletons by $\{v,r\}$ and $\{w\}$, which are joined by the edge $rw$.
	\end{proof}
	
	\subsection*{Stability for complete minors}
	
	Fix $\varepsilon\in(0,1)$, let $U=\{v:x_v>\varepsilon\}$ and $S=V(G)\setminus U$, call the vertices of $S$ \emph{light}, and put $\beta=\sum_{u\in U}x_u^{k}$, so that $|U|\le\varepsilon^{-k}$. Write $P=\rho(\HH)/k$ and let $P_m$ be the part of $P$ coming from the $k$-cliques with exactly $m$ light vertices, so that $P=\sum_{m=0}^{k}P_m$ and $P_0\le\binom{|U|}{k}=O_\varepsilon(1)$. By the theorem of Mader~\cite{M} there is a constant
	$d=d(t)$ such that every $K_t$-minor-free graph is
	$d$-degenerate.
	Fix an ordering of $V(G)$ in which
		every vertex has at most $d$ neighbors appearing later,
		and let $N^{+}(w)$ denote the set of those later
		neighbors of $w$. We use the same notation for each
		degeneracy ordering below.
	Every clique of $G$ is then contained in
	$\{w\}\cup N^{+}(w)$ for its earliest vertex $w$,
	and at most $\binom{d}{k-1}$ of the $k$-cliques
	have the same earliest vertex. The cliques with two or more light vertices are then negligible, for a reason that has nothing to do with the excluded minor.
	
	\begin{lem}\label{lem:tail}
		Let $G$ be $d$-degenerate. Then $\displaystyle\sum_{m\ge2}P_m\ \le\ \binom{d}{k-1}\,\varepsilon\,n^{\frac{k-1}{k}}$.
	\end{lem}
	
	\begin{proof}
		Let $F$ be a $k$-clique with at least two light vertices and let $w$ be its earliest vertex. Then $F\setminus\{w\}$ still contains a light vertex, and the coordinates on the other $k-2$ vertices are at most $1$, so $\prod_{v\in F}x_v\le\varepsilon\,x_w$; and $F\setminus\{w\}\subseteq N^{+}(w)$, so at most $\binom{d}{k-1}$ such cliques have the same $w$. Summing over $w$ and using $\sum_wx_w\le n^{(k-1)/k}$, which is H\"older's inequality, gives the bound.
	\end{proof}
	
	For $v\in S$ set $q_v=\sum_{A'}\prod_{u\in A'}x_u$, the sum being over those $(k-1)$-subsets $A'$ of $U\cap N(v)$ that induce cliques, so that $P_1=\sum_{v\in S}x_vq_v$. Call $v$ \emph{ambiguous} if $|U\cap N(v)|\ge t-1$.
	
	\begin{lem}\label{lem:Kramb}
		The number of ambiguous vertices is at most $(t-2)\binom{|U|}{t-1}=O_\varepsilon(1)$, and for every other $v\in S$ we have $q_v\le Q$, where $Q=e_{k-1}(z_1,\dots,z_a)$ is the elementary symmetric polynomial of degree $k-1$ in the $a$ largest coordinates of $x$ restricted to $U$, padded with zeros if $|U|<a$.
	\end{lem}
	
	\begin{proof}
		Each $(t-1)$-subset of $U$ has at most $t-2$ common neighbors by Lemma~\ref{lem:Krcommon}, and every ambiguous vertex is a common neighbor of some such subset. If $v$ is not ambiguous then $|U\cap N(v)|\le t-2=a$, and $q_v$ is a sum of products over $(k-1)$-subsets of a set of at most $a$ coordinates, each of which is at most the corresponding coordinate among the $a$ largest; since all terms are nonnegative, $q_v\le Q$.
	\end{proof}
	
	\begin{lem}\label{lem:Krmac}
		$\displaystyle Q\le\binom{a}{k-1}a^{-\frac{k-1}{k}}\beta^{\frac{k-1}{k}}$, with equality only if $z_1=\dots=z_a$ and $\beta=\sum_iz_i^{k}$.
	\end{lem}
	
	\begin{proof}
		By Maclaurin's inequality $\bigl(e_{k-1}/\binom{a}{k-1}\bigr)^{1/(k-1)}\le e_1/a$, with equality only if all $z_i$ are equal, and by H\"older $e_1=\sum_iz_i\le a^{(k-1)/k}\bigl(\sum_iz_i^{k}\bigr)^{1/k}\le a^{(k-1)/k}\beta^{1/k}$. Hence $e_{k-1}\le\binom{a}{k-1}\bigl(a^{-1/k}\beta^{1/k}\bigr)^{k-1}$.
	\end{proof}
	
	\begin{prop}\label{prop:Krmaster}
		Put $\Psi(\beta)=\binom{a}{k-1}a^{-\frac{k-1}{k}}\beta^{\frac{k-1}{k}}(1-\beta)^{\frac1k}$. With $\varepsilon$ fixed and $n\to\infty$,
		\begin{equation}\label{eq:Krmaster}
			\frac{\rho(\HH)}{k\,n^{(k-1)/k}}\ \le\ \Psi(\beta)+\binom{d}{k-1}\varepsilon+o_\varepsilon(1).
		\end{equation}
		Consequently $\rho(\HH)\le(c_{k,t}+o(1))n^{(k-1)/k}$ for every $K_t$-minor-free $G$; and if $\rho(\HH)\ge(c_{k,t}-o(1))n^{(k-1)/k}$, one may choose $\varepsilon=\varepsilon_n\downarrow0$ sufficiently slowly so that $\beta\to\frac{k-1}{k}$ and the inequalities in the proof are asymptotically sharp.
	\end{prop}
	
	\begin{proof}
		Split $P=P_1+P_0+\sum_{m\ge2}P_m$. The middle term is $O_\varepsilon(1)$ and the last is bounded by Lemma~\ref{lem:tail}. For $P_1$, discard the ambiguous vertices, whose contribution is at most $\binom{|U|}{k-1}(t-2)\binom{|U|}{t-1}=O_\varepsilon(1)$, and apply Lemma~\ref{lem:Kramb} and then Lemma~\ref{lem:sigma} with $j=1$:
		\[
		P_1\le Q\sum_{v\in S}x_v+O_\varepsilon(1)\le Q\,(1-\beta)^{1/k}n^{(k-1)/k}+O_\varepsilon(1).
		\]
		Inserting Lemma~\ref{lem:Krmac} gives~\eqref{eq:Krmaster}. The function $\beta\mapsto\beta^{(k-1)/k}(1-\beta)^{1/k}$ attains its maximum $\bigl(\frac{k-1}{k}\bigr)^{(k-1)/k}k^{-1/k}$ only at $\beta=\frac{k-1}{k}$, and
		\[
		k\binom{a}{k-1}a^{-\frac{k-1}{k}}\Bigl(\frac{k-1}{k}\Bigr)^{\frac{k-1}{k}}k^{-\frac1k}
		=\binom{a}{k-1}\Bigl(\frac{k-1}{a}\Bigr)^{\frac{k-1}{k}}
		=\binom{a-1}{k-2}\Bigl(\frac{a}{k-1}\Bigr)^{\frac1k}=c_{k,t},
		\]
		using $\binom{a}{k-1}=\frac{a}{k-1}\binom{a-1}{k-2}$ in the middle and $\binom{a-1}{k-2}^{(k-1)/k}\binom{a}{k-1}^{1/k}=\binom{a-1}{k-2}\bigl(\frac{a}{k-1}\bigr)^{1/k}$ at the end. Hence $k\Psi\le c_{k,t}$ on $[0,1]$, and~\eqref{eq:Krmaster} gives $\rho(\HH)\le\bigl(c_{k,t}+k\binom{d}{k-1}\varepsilon+o_\varepsilon(1)\bigr)n^{(k-1)/k}$; letting $n\to\infty$ and then $\varepsilon\to0$ proves the second assertion. For a near-extremal sequence, choose $\varepsilon_n\downarrow0$ sufficiently slowly that the terms $O_{\varepsilon_n}(1)/n^{(k-1)/k}$ tend to zero. Then $\Psi(\beta)\to c_{k,t}/k$, and the unique maximum of $\Psi$ forces $\beta\to(k-1)/k$ and asymptotic equality in the preceding estimates.
	\end{proof}
	
	\begin{cor}\label{cor:Krstab}
		Let $G_n$ be $K_t$-minor-free with $\rho(\CC_k(G_n))\ge(c_{k,t}-o(1))n^{(k-1)/k}$. Then $G_n$ contains a set $A_n$ of $a=t-2$ vertices whose coordinates all tend to $(\frac{k-1}{ka})^{1/k}$, all other coordinates tend to $0$, and all but $o(n)$ vertices are common neighbors of $A_n$.
	\end{cor}

	\begin{proof}
		Let $x$ be a normalized Perron vector of $\CC_k(G_n)$, and put
		$\alpha=(k-1)/k$ and $z_0=(\alpha/a)^{1/k}$.
		Choose $\varepsilon_n\downarrow0$ as in
		Proposition~\ref{prop:Krmaster}, and write
		$U_n=\{v:x_v>\varepsilon_n\}$ and
		$\beta_n=\sum_{v\in U_n}x_v^k$.
		By asymptotic sharpness and the equality conditions in
		Lemma~\ref{lem:Krmac}, compactness implies that the $a$ largest
		coordinates tend to $z_0$ and their $k$th powers sum to
		$\beta_n-o(1)=\alpha+o(1)$.
		Let $A=A_n$ be the corresponding vertices.
		Since $z_0>0$, we have $A\subseteq U_n$ for all sufficiently
		large $n$, and hence
		\[
		\eta^k:=\max_{v\notin A}x_v^k
		\le
		\max\left\{\varepsilon_n^k,\,
		\beta_n-\sum_{u\in A}x_u^k\right\}
		=o(1).
		\]
		
		Put $T=N(A)$, $R=V(G_n)\setminus(A\cup T)$,
		$W=\sum_{v\notin A}x_v^k$, and $Q=e_{k-1}(x_A)$.
		For $v\notin A$, let $q_v$ be the total product weight of
		the $(k-1)$-cliques in $G_n[A\cap N(v)]$.
		Then $q_v\le Q$ for $v\in T$.
		Every $v\in R$ misses some $u\in A$, so, for all sufficiently
		large $n$,
		\[
		q_v
		\le Q-x_u e_{k-2}(x_{A\setminus\{u\}})
		\le Q-\delta,
		\qquad
		\delta=\frac12\binom{a-1}{k-2}z_0^{k-1}>0.
		\]
		By the proof of Lemma~\ref{lem:tail}, the total weight of
		$k$-cliques with at least two vertices outside $A$ is
		$O(\eta n^\alpha)=o(n^\alpha)$.
		The cliques entirely inside $A$ contribute $O(1)$.
		Consequently, H\"older's inequality gives
		\[
		\frac{\rho(\CC_k(G_n))}{k n^\alpha}
		\le
		W^{1/k}
		\left(
		\frac{|T|}{n}Q^{1/\alpha}
		+\frac{|R|}{n}(Q-\delta)^{1/\alpha}
		\right)^\alpha
		+o(1).
		\]
		Now $W\to1/k$, $Q\to Q_0:=\binom{a}{k-1}z_0^{k-1}$,
		and $(|T|+|R|)/n\to1$.
		By the hypothesis and Proposition~\ref{prop:Krmaster}, the
		left side tends to $c_{k,t}/k=Q_0k^{-1/k}$.
		Since $\delta>0$, any positive limit point of $|R|/n$
		would make the right side strictly smaller than this limit.
		Thus $|R|=o(n)$, proving the final assertion.
	\end{proof}

	\begin{thm}\label{thm:Kr}
		Let $k\ge3$ and $t\ge k+1$. For all sufficiently large $n$, $B^{(k)}_{n,t}$ is the unique hypergraph of maximum spectral radius in $\HH^{(k)}_{K_t}(n)$.
	\end{thm}
	
	\begin{proof}
		By Lemma~\ref{lem:reduce}, the completion of an extremal hypergraph is $\CC_k(G)$ for some $K_t$-minor-free $G$ and is also extremal, and $\rho\ge\rho(B^{(k)}_{n,t})$, so Corollary~\ref{cor:Krstab} applies. Let $A=A_n$, let $T$ be the set of common neighbors of $A$ and $R=V(G)\setminus(A\cup T)$, so $|R|=o(n)$ and $|T|\ge a+1$ for large $n$. By Lemma~\ref{lem:Krstruct}, $T$ is independent and every vertex of $R$ has at most one neighbor in $T$.
		
		Let $B$ be the copy of $B^{(k)}_{n,t}$ on $V(G)$ with core $A$ and compare $P_B(x)$ with $P_{\CC_k(G)}(x)$ at the same vector, where $P_{\HH}(x)=\sum_{F\in E(\HH)}\prod_{v\in F}x_v$, as in Observation~\ref{obs:strict}. Each $r\in R$ misses some $u\in A$, so none of the $\binom{a-1}{k-2}$ sets $\{u,r\}\cup A''$ with $A''\in\binom{A\setminus\{u\}}{k-2}$ is an edge of $\CC_k(G)$, while all of them are edges of $B$; these sets are distinct for distinct $r$. Writing $m=\min_{u\in A}x_u$ and $X_R=\sum_{r\in R}x_r$,
		\[
		\sum_{e\in E(B)\setminus E(\CC_k(G))}\ \prod_{v\in e}x_v\ \ge\ \binom{a-1}{k-2}m^{k-1}X_R .
		\]
		Conversely an edge of $\CC_k(G)$ not in $B$ is a $k$-clique meeting $R$, since $T$ is independent and any $k$-clique inside $A\cup T$ with at most one vertex of $T$ is already an edge of $B$. Such a clique has at most one vertex in $T$, by independence, and its vertices outside $A$ have coordinates at most $\eta=\max_{v\notin A}x_v$. Since it has at most $k-2$ vertices in $A$, its vertices outside $A$ form a clique $Q$ of $G-A$ with $|Q|\ge2$ and $Q\cap R\ne\emptyset$, so the product of its coordinates is at most $\eta\,x_r$ for any $r\in Q\cap R$. The graph $G-A$ is $d$-degenerate, so we may fix a degeneracy ordering of it and let $w$ be the earliest vertex of $Q$, so that $Q\subseteq\{w\}\cup N^{+}(w)$ and at most $2^{d+a}$ cliques of $\CC_k(G)$ have the same $w$. If $w\in R$, the product is at most $\eta x_w$. If $w\in T$, then $Q\setminus\{w\}\subseteq R$ and $w$ is the unique neighbor in $T$ of each $r\in Q\setminus\{w\}$, so $w$ is determined by any such $r$ and the product is at most $\eta x_r$. In either case the total weight of the $k$-cliques of $G$ that are not edges of $B$ is at most $C_0\,\eta X_R$ for a constant $C_0=C_0(k,t)$. Since $\eta\to0$ and $m$ tends to a positive limit by Corollary~\ref{cor:Krstab}, we get $P_B(x)>P_{\CC_k(G)}(x)$ as soon as $X_R>0$. If $X_R=0$, that is $x_r=0$ for every $r\in R$, then $P_B(x)\ge P_{\CC_k(G)}(x)$ by the same two estimates, and the vector $z=x+\varepsilon\mathbf 1_R$ satisfies
		\[
		P_B(z)\ \ge\ P_{\CC_k(G)}(x)+\binom{a-1}{k-2}m^{k-1}|R|\,\varepsilon,\qquad \|z\|_k^{k}=1+|R|\varepsilon^{k},
		\]
		so that $\rho(B)\ge kP_B(z)/\|z\|_k^{k}>\rho(\CC_k(G))$ for small $\varepsilon$, because $k\ge2$. Either way extremality is contradicted, so $R=\emptyset$.
		
		Now every $k$-clique of $G$ has at most one vertex in $T$, so $\CC_k(G)\subseteq B$, and if the inclusion were proper then $\rho(B)>\rho(\CC_k(G))$ by Observation~\ref{obs:strict}, since $B$ is connected. Therefore $\CC_k(G)=B\cong B^{(k)}_{n,t}$. Since this completion is connected, Lemma~\ref{lem:reduce}
		also forces the original extremal hypergraph to equal it.
	\end{proof}
	
	\subsection*{The regime $j\le1$ of the bipartite problem}
	
	The argument above is not specific to complete minors. Suppose $2\le s\le t$ and $s\ge k$, so that $j\le1$, let $G$ be $K_{s,t}$-minor-free and $\HH=\CC_k(G)$, and keep the notation $U,S,\beta,P,P_m$ of this section.
	
	\begin{thm}\label{thm:jle1stab}
		Let $G_n$ be $K_{s,t}$-minor-free with $s\ge k$ and
		\[
		\rho(\CC_k(G_n))\ \ge\ \bigl(c-o(1)\bigr)n^{\frac{k-1}{k}},\qquad c=\binom{s-1}{k-1}\Bigl(\frac{k-1}{s-1}\Bigr)^{\frac{k-1}{k}} .
		\]
		Then $\rho(\CC_k(G_n))=(c+o(1))n^{(k-1)/k}$, and $G_n$ contains a set $A_n$ of $s-1$ vertices whose coordinates all tend to $\bigl(\frac{k-1}{k(s-1)}\bigr)^{1/k}$, while every other coordinate tends to $0$, all but $o(n)$ vertices are common neighbors of $A_n$, and the vertices outside $A_n\cup N(A_n)$ carry total mass $o(1)$. Moreover the graph induced on those common neighbors has maximum degree at most $t-1$.
	\end{thm}
	
	\begin{proof}
		By the theorem of Mader~\cite{M} every $K_{s,t}$-minor-free graph is $d$-degenerate for some $d=d(s,t)$, so Lemma~\ref{lem:tail} and the bound $P_0\le\binom{|U|}{k}$ apply verbatim and again only $P_1$ matters. Call $v\in S$ ambiguous if $|U\cap N(v)|\ge s$. By Lemma~\ref{lem:common} every $s$-subset of $U$ has at most $t-1$ common neighbors, so there are at most $(t-1)\binom{|U|}{s}=O_\varepsilon(1)$ ambiguous vertices, and their contribution to $P_1$ is $O_\varepsilon(1)$. For every other $v$ we have $|U\cap N(v)|\le s-1$, and since $k-1\le s-1$ the quantity $q_v$ is a sum of products over $(k-1)$-subsets of at most $s-1$ coordinates, whence $q_v\le Q:=e_{k-1}(z_1,\dots,z_{s-1})$ for the $s-1$ largest coordinates on $U$, padded with zeros if necessary. Lemma~\ref{lem:Krmac} with $a$ replaced by $s-1$ gives $Q\le\binom{s-1}{k-1}(s-1)^{-(k-1)/k}\beta^{(k-1)/k}$, and Lemma~\ref{lem:sigma} with $j=1$ gives $\sum_{v\in S}x_v\le(1-\beta)^{1/k}n^{(k-1)/k}$. Exactly as in Proposition~\ref{prop:Krmaster},
		\[
		k\binom{s-1}{k-1}(s-1)^{-\frac{k-1}{k}}\Bigl(\frac{k-1}{k}\Bigr)^{\frac{k-1}{k}}k^{-\frac1k}=\binom{s-1}{k-1}\Bigl(\frac{k-1}{s-1}\Bigr)^{\frac{k-1}{k}}=c,
		\]
		so that
		\[
		\frac{\rho(\CC_k(G_n))}{k\,n^{(k-1)/k}}\ \le\ \binom{s-1}{k-1}(s-1)^{-\frac{k-1}{k}}\beta^{\frac{k-1}{k}}(1-\beta)^{\frac1k}+\binom{d}{k-1}\varepsilon+o_\varepsilon(1)\ \le\ \frac ck+\binom{d}{k-1}\varepsilon+o_\varepsilon(1).
		\]
		Letting $n\to\infty$ and then $\varepsilon\to0$ gives
		$\rho(\CC_k(G_n))=(c+o(1))n^{(k-1)/k}$.
		
			Choose $\varepsilon_n\downarrow0$ sufficiently slowly as in
			the proof of Corollary~\ref{cor:Krstab}.
			The same argument, applied to the preceding estimates with
			$a=s-1$, gives the stated coordinate concentration and
			$|V(G_n)\setminus(A\cup N(A))|=o(n)$, where $A$ consists
			of the vertices with the $s-1$ largest coordinates.
			For the mass, put $T=N(A)$, $R=V(G_n)\setminus(A\cup T)$, $W=\sum_{v\notin A}x_v^k$ and $W_R=\sum_{v\in R}x_v^k$, and let $q_v$ and $Q=e_{k-1}(x_A)$ be as in that proof. Applying H\"older's inequality to $T$ and to $R$ separately,
			\[
			\sum_{v\notin A}x_vq_v\ \le\ Q\Bigl((W-W_R)^{\frac1k}|T|^{\frac{k-1}{k}}+W_R^{\frac1k}|R|^{\frac{k-1}{k}}\Bigr)\ \le\ Q\,(W-W_R)^{\frac1k}n^{\frac{k-1}{k}}+o\bigl(n^{\frac{k-1}{k}}\bigr),
			\]
			since $|R|=o(n)$. The cliques with at least two vertices outside $A$, or none, have weight $o(n^{(k-1)/k})$ as in that proof. As $W\to1/k$ and $Q\to\binom{s-1}{k-1}z_0^{k-1}$, where $z_0$ is the limit of the coordinates on $A$, the lower bound on $\rho(\CC_k(G_n))$ forces $\liminf(W-W_R)\ge1/k$, that is, $W_R\to0$.
		
		The last assertion is Corollary~\ref{cor:deg}.
	\end{proof}
	
	\begin{rem}\label{rem:jle1}
		Theorem~\ref{thm:jle1stab} determines the leading
		asymptotic constant and the concentration of the
		Perron vector on $s-1$ vertices. The leading
		contribution comes from $k$-cliques with exactly
		one light vertex, which explains why the constant
		does not depend on $t$.

	\end{rem}
	
	\section{Complete bipartite minors}\label{sec:bip}
	
	Let $2\le s\le t$ and $j=k-s+1$. A hyperedge spans $K_k$, and $K_k$ contains $K_{s,t}$ when $k\ge s+t$, so $\HH^{(k)}_{K_{s,t}}(n)$ is trivial unless $s+t\ge k+1$; equivalently $t\ge j$.

	\begin{figure}[H]
		\centering
		\begin{tikzpicture}[scale=0.95,
			hub/.style={circle,fill=black,inner sep=2.2pt},
			blk/.style={circle,draw=black,fill=white,inner sep=1.5pt}]
			
			\node[hub] (a1) at (-0.9,2.1) {};
			\node[above] at (a1) {$a_1$};
			\node[hub] (a2) at (0.9,2.1) {};
			\node[above] at (a2) {$a_2$};
			\draw (a1)--(a2);
			
			\foreach \b/\x in {1/-3.6, 2/-1.2, 3/1.2, 4/3.6} {
				\node[blk] (p\b) at (\x-0.55,0) {};
				\node[blk] (q\b) at (\x+0.55,0) {};
				\node[blk] (r\b) at (\x,-0.95) {};
				\draw (p\b)--(q\b)--(r\b)--(p\b);
				\foreach \v in {p,q,r} {
					\draw[gray!70] (a1)--(\v\b);
					\draw[gray!70] (a2)--(\v\b);
				}
			}
			
			\node at (0,-1.75) {$K_{s-1}\vee bK_t$};
		\end{tikzpicture}
		
		\caption{
				The graph $K_{s-1}\vee bK_t$ of Theorem~\ref{thm:introbip},
				with $s=3$, $t=3$ and $b=4$. Triangles are the largest admissible
				complete blocks: a complete block on four vertices, together
				with the two hubs, would contain a $K_{3,3}$ subgraph.
		}
		\label{fig:cand}
	\end{figure}
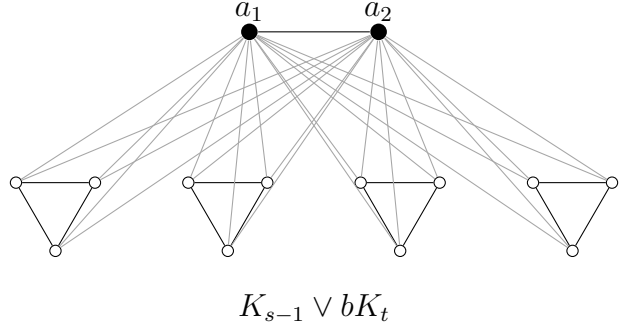
	
	For the normalized Perron vector of
	$\CC_k(K_{s-1}\vee bK_t)$, a hub carries mass of order $1$
	and a light vertex mass of order $1/n$.
	Thus the cliques meeting the light part and containing
	$i$ hubs contribute at order $n^{i/k}$ whenever this family
	is nonempty, and the dominant family is the one with
	$i^*=\min(s-1,k-1)$ hubs. The all-hub term contributes
	only $O(1)$. The proof below makes this exact.
	
	\begin{thm}\label{thm:bip}
		Let $2\le s\le t$ and $t\ge j$, put $i^{*}=\min(s-1,k-1)$, and let $h=bt=n-s+1$. Then
		\[
		\rho\bigl(\CC_k(K_{s-1}\vee bK_t)\bigr)=c\,h^{\frac{i^{*}}{k}}+O\bigl(h^{\frac{i^{*}-1}{k}}\bigr)=\bigl(c+o(1)\bigr)n^{\frac{i^{*}}{k}},
		\]
		where
		\[
		c=\begin{cases}
			\displaystyle\binom{s-1}{k-1}\Bigl(\frac{k-1}{s-1}\Bigr)^{\frac{k-1}{k}}, & s\ge k\quad(j\le1),\\[2ex]
			\displaystyle\frac1t\binom tj\,j^{\,j/k}, & 2\le s\le k-1\quad(j\ge2),
		\end{cases}
		\]
		and the mass carried by the hubs in the Perron vector tends to $\frac{i^{*}}{k}$ as $b\to\infty$.
	\end{thm}
	
	\begin{proof}
		Write $\HH=\CC_k(K_{s-1}\vee bK_t)$ and $m=h=bt$. The automorphism group of $K_{s-1}\vee bK_t$ is transitive on the hubs and on the light vertices, and the positive Perron vector of the connected hypergraph $\HH$ is unique, so it takes a common value $u$ on the hubs and a common value $v$ on the light vertices. Consequently $\rho(\HH)$ is the maximum of the function in~\eqref{eq:var} over the two-parameter family of such vectors:
		\begin{equation}\label{eq:twovar}
			\rho(\HH)=\max\Bigl\{\,\widehat P(u,v)\ :\ u,v\ge0,\ (s-1)u^{k}+mv^{k}=1\,\Bigr\},
		\end{equation}
		since the maximum over all unit vectors is attained at a vector of this family. A $k$-clique of $K_{s-1}\vee bK_t$ consists of $i$ hubs and a $(k-i)$-clique of a single block, so
		\[
		\widehat P(u,v)=k\binom{s-1}{k}u^{k}+k\sum_{i=0}^{i^{*}}\binom{s-1}{i}\,b\binom{t}{k-i}\,u^{i}v^{k-i},
		\]
		where the first term is absent unless $s-1\ge k$. Parametrize the constraint by the hub mass $\alpha=(s-1)u^{k}\in[0,1]$, so that $mv^{k}=1-\alpha$ and
		\[
		b\,u^{i}v^{k-i}=\frac mt\Bigl(\frac{\alpha}{s-1}\Bigr)^{\frac ik}\Bigl(\frac{1-\alpha}{m}\Bigr)^{\frac{k-i}{k}}
		=\frac1t\Bigl(\frac{\alpha}{s-1}\Bigr)^{\frac ik}(1-\alpha)^{\frac{k-i}{k}}\,m^{\frac ik}.
		\]
		Hence
		\begin{equation}\label{eq:Pexp}
			\widehat P=\sum_{i=0}^{i^{*}}m^{\frac ik}g_i(\alpha)+k\binom{s-1}{k}\frac{\alpha}{s-1},\qquad
			g_i(\alpha)=\frac kt\binom{s-1}{i}\binom{t}{k-i}\Bigl(\frac{\alpha}{s-1}\Bigr)^{\frac ik}(1-\alpha)^{\frac{k-i}{k}} .
		\end{equation}
		Each $g_i$ is continuous on $[0,1]$ and bounded by a constant depending only on $k,s,t$, so for every $\alpha$,
		\[
		\bigl|\widehat P-m^{\frac{i^{*}}{k}}g_{i^{*}}(\alpha)\bigr|\ \le\ C\,m^{\frac{i^{*}-1}{k}},
		\]
		where the constant term of~\eqref{eq:Pexp} is absorbed because $i^{*}\ge1$. Taking the maximum over $\alpha$ in~\eqref{eq:twovar},
		\begin{equation}\label{eq:rhoexp}
			\rho(\HH)=m^{\frac{i^{*}}{k}}\max_{0\le\alpha\le1}g_{i^{*}}(\alpha)+O\bigl(m^{\frac{i^{*}-1}{k}}\bigr).
		\end{equation}
		The function $g_{i^{*}}$ is a constant multiple of $\alpha^{i^{*}/k}(1-\alpha)^{(k-i^{*})/k}$, whose unique maximum on $[0,1]$ is at $\alpha=i^{*}/k$, where it equals $(i^{*}/k)^{i^{*}/k}\bigl((k-i^{*})/k\bigr)^{(k-i^{*})/k}$.
		
		If $j\ge2$ then $i^{*}=s-1$ and $k-i^{*}=j$, and
		\[
		\max g_{i^{*}}=\frac kt\binom tj(s-1)^{-\frac{s-1}{k}}\Bigl(\frac{s-1}{k}\Bigr)^{\frac{s-1}{k}}\Bigl(\frac jk\Bigr)^{\frac jk}
		=\frac1t\binom tj\,j^{\frac jk},
		\]
		since $k\cdot k^{-(s-1)/k-j/k}=1$. If $j\le1$ then $i^{*}=k-1$ and $k-i^{*}=1$, and
		\[
		\max g_{i^{*}}=k\binom{s-1}{k-1}(s-1)^{-\frac{k-1}{k}}\Bigl(\frac{k-1}{k}\Bigr)^{\frac{k-1}{k}}\Bigl(\frac1k\Bigr)^{\frac1k}
		=\binom{s-1}{k-1}\Bigl(\frac{k-1}{s-1}\Bigr)^{\frac{k-1}{k}} .
		\]
		In both cases $\max g_{i^{*}}=c$, and~\eqref{eq:rhoexp} is the first assertion; the second form follows from $h=n-s+1$.
		
		For the mass, let $\alpha_b$ be the hub mass of the Perron vector, which by~\eqref{eq:twovar} maximizes $\widehat P$. Dividing by $m^{i^{*}/k}$, $\alpha_b$ maximizes $g_{i^{*}}+\varepsilon_b$ where $\sup_\alpha|\varepsilon_b(\alpha)|\le Cm^{-1/k}\to0$. Since $g_{i^{*}}$ is continuous on the compact interval $[0,1]$ with the unique maximizer $i^{*}/k$, every limit point of $(\alpha_b)$ is a maximizer of $g_{i^{*}}$, so $\alpha_b\to i^{*}/k$.
	\end{proof}
	
	Although Theorem~\ref{thm:bip} is stated for \(t\mid h\), it supplies
	the same asymptotic lower bound for every residue class.
	
	\begin{cor}\label{cor:lb}
		Let $2\le s\le t$ and $t\ge j$, and put $i^{*}=\min(s-1,k-1)$. Then, as \(n\to\infty\),
		\[
		\max_{\HH\in\HH^{(k)}_{K_{s,t}}(n)}\rho(\HH)\ \ge\ c\,n^{\frac{i^{*}}{k}}+O\bigl(n^{\frac{i^{*}-1}{k}}\bigr),
		\]
		with $c$ as in Theorem~\ref{thm:bip}.
	\end{cor}
	
	\begin{proof}
		Write $h=n-s+1=bt+p$ with $0\le p<t$ and let $G=K_{s-1}\vee bK_t$ together with $p$ isolated vertices, which lies in the class by Corollary~\ref{cor:small}. By Theorem~\ref{thm:bip}, $\rho(\CC_k(G))=c(bt)^{i^{*}/k}+O\bigl((bt)^{(i^{*}-1)/k}\bigr)$, and $bt=h-p=n-s+1-p$ differs from $n$ by a constant, so $c(bt)^{i^{*}/k}=cn^{i^{*}/k}+O\bigl(n^{i^{*}/k-1}\bigr)$. Finally $\frac{i^{*}}{k}-1<\frac{i^{*}-1}{k}$, so the second error term is absorbed by the first.
	\end{proof}
	
	\begin{rem}\label{rem:trich}
		The three regimes are governed by $j$ alone:
		\begin{itemize}
			\item $j\le1$: order $n^{(k-1)/k}$, the constant does not
			involve $t$, and
			the leading term does not depend on edges
				within the light part. The exponent agrees with that in
				Section~\ref{sec:Kr} and in the graph case $k=2$.
			
			\item $j=2$: order $n^{(s-1)/k}$, the constant involves $t$,
			and equality in
			Lemma~\ref{lem:count} gives only $(t-1)$-regularity.
			For $t\ge3$, equality in the next term
				identifies the components as copies of $K_t$; when $t=2$,
				regularity already gives this conclusion.
			
			\item $j\ge3$: order $n^{(s-1)/k}$, the constant involves $t$,
			and equality in
			Lemma~\ref{lem:count} forces the components of the light
			part to be $K_t$ at once.
		\end{itemize}
	\end{rem}
	
	\section{Stability, and all residues when $k\le s$}\label{sec:stab}
	
	Fix $k\ge3$. Until the final subsection, let $G$ be $K_{s,t}$-minor-free with $2\le s\le\min(k,t)$ and $t\ge j$, on $n$ vertices, and let $d=d(s,t)$ be such that every $K_{s,t}$-minor-free graph is $d$-degenerate, which exists by the theorem of Mader~\cite{M}. Throughout this section $x$ is a nonnegative vector on $V(G)$ with $\sum_vx_v^{k}=1$; it is the Perron vector of $\CC_k(G)$ only where this is said. Write $P_G(x)=\sum_{F\in E(\CC_k(G))}\prod_{v\in F}x_v$, and abbreviate this to $P$ when $G,x$ are fixed. An extremal graph means an $n$-vertex $K_{s,t}$-minor-free graph maximizing $\rho(\CC_k(G))$. Write $x_{(1)}\ge x_{(2)}\ge\cdots\ge x_{(n)}$ for the coordinates in nonincreasing order, fix an ordering of $V(G)$ realizing it, let $A^{(\ell)}$ be the set of the first $\ell$ vertices, and put
	\[
	\beta_\ell=\sum_{u\in A^{(\ell)}}x_u^{k},\qquad A=A^{(s-1)},\qquad \pi_A=\prod_{u\in A}x_u,\qquad K=\frac1t\binom tj,\qquad c=K\,j^{\,j/k}.
	\]
	The set $A$ of the $s-1$ largest coordinates plays the role of the hubs, and everything below is stated in terms of the sorted coordinates; no threshold is needed. The case $j=1$, that is $s=k$, is included: there Theorems~\ref{thm:exact} and~\ref{thm:stab} below, which are stated for $j\ge2$, are replaced by Theorem~\ref{thm:bip} and Theorem~\ref{thm:jle1stab}, which supply everything that the final argument needs, namely $\rho=(c+o(1))n^{(s-1)/k}$, the convergence of the hub coordinates to $k^{-1/k}$, and $\max_{v\notin A}x_v\to0$.

	The first lemma is the heart of the section. It says that the $k$-cliques containing at most $s-2$ of the $\ell$ heaviest vertices carry weight of order $x_{(\ell+1)}n^{(s-1)/k}$, one factor of the next coordinate below the leading order. It holds for every vector $x$, and its proof uses only the two structural facts available in a $K_{s,t}$-minor-free graph: $s$ vertices have at most $t-1$ common neighbors, and the graph is $d$-degenerate.
	
	\begin{lem}\label{lem:low}
		Let $s-1\le\ell<n$. The $k$-cliques $F$ of $G$ with $|F\cap A^{(\ell)}|\le s-2$ satisfy
		\[
		\sum_{F}\ \prod_{v\in F}x_v\ \le\ C_1\,x_{(\ell+1)}\,n^{\frac{s-1}{k}}+C_1\,\ell\,x_{(\ell+1)}^{\,j+1},\qquad C_1=\binom{d}{s-1}\binom{t-1}{j-1}.
		\]
	\end{lem}
	
	\begin{proof}
		Let $F$ be such a clique, and let $Y$ consist of its $s$ heaviest vertices, listed as $v_1,\dots,v_s$ with $x_{v_1}\ge\cdots\ge x_{v_s}$, where vertices of $A^{(\ell)}$ are listed first; this is consistent with the ordering, since every vertex of $A^{(\ell)}$ is at least as heavy as every vertex outside it. Since $|F\cap A^{(\ell)}|\le s-2$, the set $Y$ contains $F\cap A^{(\ell)}$ and at least two further vertices, so $v_{s-1},v_s\notin A^{(\ell)}$, and in particular $x_{v_{s-1}}\le x_{(\ell+1)}$. The remaining $k-s=j-1$ vertices of $F$ are common neighbors of $Y$ and have coordinates at most $x_{v_s}$; hence
		\[
		\prod_{v\in F}x_v\ \le\ x_{v_1}\cdots x_{v_{s-1}}\,x_{v_s}^{\,j},
		\]
		and by Lemma~\ref{lem:common} at most $\binom{t-1}{j-1}$ cliques $F$ have the same $Y$. Fix a degeneracy ordering of $G$ and charge $Y$ to its earliest vertex $w$ in it; then $Y\setminus\{w\}\subseteq N^{+}(w)$, so at most $\binom{d}{s-1}$ sets $Y$ are charged to a given $w$.
		
		If $w\in A^{(\ell)}$, then $v_{s-1}$ and $v_s$ lie outside $A^{(\ell)}$, so $x_{v_{s-1}}\le x_{(\ell+1)}$ and $x_{v_s}\le x_{(\ell+1)}$, and the product $x_{v_1}\cdots x_{v_{s-1}}x_{v_s}^{\,j}$ is at most $x_{(\ell+1)}^{\,j+1}$; there are at most $\ell$ such $w$, which accounts for the term $C_1\ell x_{(\ell+1)}^{\,j+1}$. If $w\notin A^{(\ell)}$, let $w=v_m$. Every $v_i$ with $i>m$ satisfies $x_{v_i}\le x_w$, and every $v_i$ with $i<m$ satisfies $x_{v_i}\le1$. If $m\le s-1$ this gives $x_{v_1}\cdots x_{v_{s-1}}x_{v_s}^{\,j}\le x_w^{\,s-m+j}=x_w^{\,k+1-m}\le x_{(\ell+1)}\,x_w^{\,j}$, using $x_w\le x_{(\ell+1)}$ and $k-m\ge j$. If $m=s$, then $x_{v_{s-1}}\le x_{(\ell+1)}$ gives the same bound $x_{(\ell+1)}x_w^{\,j}$. Summing over $w$ and using H\"older's inequality with $\sum_wx_w^{k}=1$,
		\[
		\sum_{w\notin A^{(\ell)}}x_{(\ell+1)}\,x_w^{\,j}\ \le\ x_{(\ell+1)}\,n^{1-\frac jk}=x_{(\ell+1)}\,n^{\frac{s-1}{k}},
		\]
		and the lemma follows.
	\end{proof}

	\begin{prop}\label{prop:master}
		For every $s-1\le\ell<n$ there is a constant $C_2(\ell)$,
		depending only on $k,s,t,\ell$, such that
		\begin{equation}\label{eq:master}
			\sum_{F\in E(\CC_k(G))}\prod_{v\in F}x_v\ \le\ K\,\pi_A\,(1-\beta_\ell)^{\frac jk}\,n^{\frac{s-1}{k}}+C_1\,x_{(\ell+1)}\,n^{\frac{s-1}{k}}+C_2(\ell).
		\end{equation}
		For $\ell=s-1$ the first term can be replaced by $K\pi_A\,W_A^{j/k}\,n_A^{(s-1)/k}$, where $W_A$ and $n_A$ are the mass and the size of $N(A)$.
	\end{prop}
	
	\begin{proof}
		Classify the $k$-cliques by $i=|F\cap A^{(\ell)}|$. Those with $i\le s-2$ are covered by Lemma~\ref{lem:low}. Those with $i\ge s$ contain $s$ vertices of $A^{(\ell)}$, whose common neighborhood has at most $t-1$ vertices by Lemma~\ref{lem:common}; there are at most $\binom{\ell}{s}\binom{t-1}{j-1}$ of them and each has weight at most $1$. A clique with $i=s-1$ consists of a set $B\in\binom{A^{(\ell)}}{s-1}$ and a $j$-clique of $G[N(B)\setminus A^{(\ell)}]$. Let $Z=\{v\notin A^{(\ell)}:|N(v)\cap A^{(\ell)}|\ge s\}$. Every $v\in Z$ is a common neighbor of some $s$-subset of $A^{(\ell)}$, so $|Z|\le(t-1)\binom{\ell}{s}$ by Lemma~\ref{lem:common}; and every $v\notin A^{(\ell)}\cup Z$ lies in $N(B)$ for at most one $B$, since $|N(v)\cap A^{(\ell)}|\le s-1$. Hence the sets $N(B)\setminus(A^{(\ell)}\cup Z)$ are pairwise disjoint, with no further argument. A clique $B\cup Q$ with $Q$ meeting $Z$ is determined by $B$, a vertex $z\in Q\cap Z$ and the remaining $j-1$ vertices of $Q$, which are common neighbors of the $s$ vertices of $B\cup\{z\}$ and hence at most $t-1$ in number; its weight is at most $\pi_Ax_{(\ell+1)}^{\,j}$, because $Q$ avoids $A^{(\ell)}$. So the cliques whose light part meets $Z$ have total weight at most $(t-1)\binom{\ell}{s}\binom{\ell}{s-1}\binom{t-1}{j-1}\pi_A\,x_{(\ell+1)}^{\,j}$. For the rest, write $W'_B$ and $n'_B$ for the mass and size of $N(B)\setminus(A^{(\ell)}\cup Z)$; these sets are pairwise disjoint, so $\sum_BW'_B\le1-\beta_\ell$ and $\sum_Bn'_B\le n$. Lemma~\ref{lem:sigma} bounds the weight of the $j$-cliques of $G[N(B)\setminus(A^{(\ell)}\cup Z)]$ by $K\,W_B'^{\,j/k}n_B'^{\,(s-1)/k}$, and $\prod_{u\in B}x_u\le\pi_A$ because $A$ consists of the $s-1$ largest coordinates. Since $\frac jk+\frac{s-1}{k}=1$, H\"older's inequality gives
		\[
		\sum_{B}\prod_{u\in B}x_u\cdot K\,W_B'^{\,\frac jk}n_B'^{\,\frac{s-1}{k}}\ \le\ K\pi_A\Bigl(\sum_BW'_B\Bigr)^{\frac jk}\Bigl(\sum_Bn'_B\Bigr)^{\frac{s-1}{k}}\ \le\ K\pi_A(1-\beta_\ell)^{\frac jk}n^{\frac{s-1}{k}} .
		\]
		Collecting the terms gives~\eqref{eq:master}. For $\ell=s-1$ there is a single $B=A$, no set $Z$, and the same computation gives the refined first term.
	\end{proof}
	
	Two consequences are immediate. First, taking $\ell=s-1$ in~\eqref{eq:master} and using the arithmetic-geometric mean inequality $\pi_A\le(\beta_{s-1}/(s-1))^{(s-1)/k}$,
	\begin{equation}\label{eq:masterPhi}
		\sum_{F\in E(\CC_k(G))}\prod_{v\in F}x_v\ \le\ \Phi(\beta_{s-1})\,n^{\frac{s-1}{k}}+C_1\,x_{(s)}\,n^{\frac{s-1}{k}}+C_2(s-1),\qquad
		\Phi(\beta)=K\Bigl(\frac{\beta}{s-1}\Bigr)^{\frac{s-1}{k}}(1-\beta)^{\frac jk},
	\end{equation}
	where $\max_{[0,1]}\Phi=\Phi(\beta_0)=c/k$ with $\beta_0=\frac{s-1}{k}$, attained only there because $\log\Phi$ is strictly concave, and $k\cdot k^{-(s-1)/k-j/k}=1$. Applied to the Perron vector this is an explicit upper bound on the spectral radius,
	\begin{equation}\label{eq:explicit}
		\rho(\CC_k(G))\ \le\ \bigl(c+kC_1\,x_{(s)}\bigr)\,n^{\frac{s-1}{k}}+kC_2(s-1),
	\end{equation}
	valid for every $K_{s,t}$-minor-free $G$, where $x_{(s)}$ is the $s$-th largest Perron coordinate. Second, the terms of~\eqref{eq:master} coming from the hub mass improve as $\ell$ grows: since $\pi_A\le(\beta_{s-1}/(s-1))^{(s-1)/k}$ and $1-\beta_\ell\le1-\beta_{s-1}$,
	\begin{equation}\label{eq:masterell}
		K\pi_A(1-\beta_\ell)^{\frac jk}\ \le\ \Phi(\beta_{s-1})\Bigl(\frac{1-\beta_\ell}{1-\beta_{s-1}}\Bigr)^{\frac jk}\ \le\ \frac ck\Bigl(1-\frac jk\bigl(\beta_\ell-\beta_{s-1}\bigr)\Bigr),
	\end{equation}
	by $(1-z)^{j/k}\le1-\frac jkz$ and $1-\beta_{s-1}\le1$. If $\beta_{s-1}=1$, the left side is zero and the bound holds directly, without using the quotient.
	
	The asymptotic and stability arguments below use Proposition~\ref{prop:master} only with $\ell$ fixed independently of $n$.
	
	\begin{thm}\label{thm:exact}
		Let $2\le s\le\min(k-1,t)$ and $t\ge j$. Then
		\[
		\max_{\HH\in\HH^{(k)}_{K_{s,t}}(n)}\rho(\HH)=\Bigl(\frac1t\binom tj\,j^{\,j/k}+o(1)\Bigr)n^{\frac{s-1}{k}} .
		\]
	\end{thm}
	
	\begin{proof}
		The lower bound is Theorem~\ref{thm:bip} applied to $\CC_k(K_{s-1}\vee bK_t)$, together with the remaining isolated vertices, which lies in the class by Corollary~\ref{cor:small}. For the upper bound, Lemma~\ref{lem:reduce} allows us to choose a maximizing $K_{s,t}$-minor-free graph $G$ and put $\HH=\CC_k(G)$; let $x$ be its Perron vector and $y=\rho(\HH)/(kn^{(s-1)/k})$. By the lower bound, $y\ge c/k-o(1)$. Fix an integer $\ell\ge s$ independently of $n$. Combining~\eqref{eq:master} with~\eqref{eq:masterell},
		\[
		\frac ck-o(1)\ \le\ y\ \le\ \frac ck\Bigl(1-\frac jk\bigl(\beta_\ell-\beta_{s-1}\bigr)\Bigr)+C_1x_{(\ell+1)}+o_\ell(1),
		\]
		so that $\beta_\ell-\beta_{s-1}\le\frac{k^{2}C_1}{cj}\,x_{(\ell+1)}+o_\ell(1)$. Since $x_{(\ell+1)}^{k}\le\beta_{\ell+1}/(\ell+1)\le1/(\ell+1)$ and $x_{(s)}^{k}\le\beta_\ell-\beta_{s-1}$,
		\[
		x_{(s)}^{k}\ \le\ \frac{k^{2}C_1}{cj}\,(\ell+1)^{-1/k}+o_\ell(1)
		\]
		for every fixed $\ell\ge s$, and letting $n\to\infty$ and then $\ell\to\infty$ gives $x_{(s)}\to0$. Now~\eqref{eq:masterPhi} gives $y\le\Phi(\beta_{s-1})+C_1x_{(s)}+o(1)\le c/k+o(1)$.
	\end{proof}
	
	The proof shows slightly more, which is the starting point of the stability argument: for every $K_{s,t}$-minor-free sequence with $\rho(\CC_k(G_n))\ge(c-o(1))n^{(s-1)/k}$, the $s$-th largest Perron coordinate tends to zero, and the cliques containing at most $s-2$ of the $s-1$ heaviest vertices have weight $o(n^{(s-1)/k})$.
	
	\begin{thm}\label{thm:stab}
		Let $2\le s\le\min(k-1,t)$ and $t\ge j$, and let $G_n$ be $K_{s,t}$-minor-free with $\rho(\CC_k(G_n))\ge(c-o(1))n^{(s-1)/k}$. Let $A_n$ be the set of the $s-1$ largest coordinates of the Perron vector $x=x^{(n)}$. Then
		\[
		\sum_{u\in A_n}x_u^{k}\longrightarrow\frac{s-1}{k},\qquad x_u\longrightarrow k^{-1/k}\ (u\in A_n),\qquad \max_{v\notin A_n}x_v\longrightarrow0,
		\]
		all but $o(n)$ vertices lie in $N(A_n)$, and all but $o(n)$ vertices of $N(A_n)$ have degree exactly $t-1$ in $G[N(A_n)]$. If $j\ge3$, all but $o(n)$ vertices of $N(A_n)$ lie in components of $G[N(A_n)]$ isomorphic to $K_t$.
	\end{thm}
	
	\begin{proof}
		Write $A=A_n$, $\beta=\beta_{s-1}$, $y=\rho(\CC_k(G_n))/(kn^{(s-1)/k})$, and let $W_A$, $n_A$ be the mass and size of $N(A)$. By the proof of Theorem~\ref{thm:exact}, $x_{(s)}=\max_{v\notin A}x_v\to0$ and $y\to c/k$.
		
		\emph{Step 1: the hubs.} By Proposition~\ref{prop:master} with $\ell=s-1$, in its refined form,
		\begin{equation}\label{eq:stab1}
			\frac ck-o(1)\ \le\ y\ \le\ K\pi_AW_A^{\frac jk}\Bigl(\frac{n_A}{n}\Bigr)^{\frac{s-1}{k}}+o(1)\ \le\ K\pi_A(1-\beta)^{\frac jk}+o(1)\ \le\ \Phi(\beta)+o(1)\ \le\ \frac ck+o(1).
		\end{equation}
		Every inequality is therefore tight up to $o(1)$. Since $\Phi$ is continuous with the unique maximizer $\beta_0$, we get $\beta\to\beta_0$. The fourth inequality is the arithmetic-geometric mean inequality for the $s-1$ numbers $x_u^{k}$, $u\in A$, whose sum tends to $\beta_0>0$; by its equality case and compactness, tightness forces $x_u^{k}-\beta/(s-1)\to0$ for every $u\in A$, hence $x_u\to k^{-1/k}$.
		
		\emph{Step 2: concentration on $N(A)$.} Since $K\pi_A(1-\beta)^{j/k}\le c/k$, the second inequality in~\eqref{eq:stab1} gives $(W_A/(1-\beta))^{j/k}(n_A/n)^{(s-1)/k}\ge1-o(1)$, hence $W_A\ge(1-o(1))(1-\beta)$ and $n_A\ge(1-o(1))n$: all but $o(n)$ vertices lie in $N(A)$, and the mass outside $A\cup N(A)$ is $o(1)$.
		
		\emph{Step 3: the structure of $G[N(A)]$.} Let $H=G[N(A)]$, and for $v\in V(H)$ let $m_v$ be the number of $j$-cliques of $H$ containing $v$; put $M=\binom{t-1}{j-1}$ and $\sigma=k/(s-1)$, so that $M/j=K$. By Step~2 and~\eqref{eq:stab1}, the $j$-cliques of $H$ have weight at least $(1-o(1))KW_A^{j/k}n_A^{(s-1)/k}$, while the two inequalities of Lemma~\ref{lem:sigma} bound it by $\frac1jW_A^{j/k}(\sum_vm_v^{\sigma})^{1/\sigma}\le\frac1jW_A^{j/k}Mn_A^{1/\sigma}$. Hence $\sum_vm_v^{\sigma}\ge(1-o(1))M^{\sigma}n_A$. Since every $m_v$ is an integer with $m_v\le M$, each vertex with $m_v<M$ contributes at most $(M-1)^{\sigma}$, so the number of such vertices is $o(n)$. If $m_v=M$ then $v$ has at least $t-1$ neighbors in $H$, while its degree in $G[N(A)]$ is at most $t-1$ by Corollary~\ref{cor:deg}; so $v$ has degree exactly $t-1$ in $G[N(A)]$, all its neighbors lie in $H$, and every $(j-1)$-subset of its neighborhood is a clique. Call the other vertices of $N(A)$ \emph{defective}; there are $o(n)$ of them, which proves the degree statement. If $j\ge3$, the neighborhood of a nondefective vertex $v$ is a clique of size $t-1$; if moreover no neighbor of $v$ is defective, each neighbor has degree $t-1$ and all its neighbors lie in $N[v]$, so $N[v]$ is a component of $G[N(A)]$ isomorphic to $K_t$. At most $t\cdot o(n)$ vertices have a defective vertex in their closed neighborhood, which proves the last statement.
	\end{proof}

	Let $A$ be a set of $s-1$ vertices, let $T=N(A)$ be its common neighborhood and $R=V(G)\setminus(A\cup T)$, and let
	\[
	\Xi_R=\sum_{\substack{F\in E(\CC_k(G))\\ F\cap R\ne\emptyset}}\ \prod_{v\in F}x_v
	\]
	be the weight of the $k$-cliques meeting $R$.
	
	We first show that $R=\emptyset$ for an extremal graph in the divisible case. Recall that $j=k-s+1$, and let $\eta=\max_{v\notin A}x_v=x_{(s)}$, $W_R=\sum_{r\in R}x_r^{k}$ and $\mu=|R|$. The next proposition describes the only way $R\ne\emptyset$ can happen: the average mass per vertex of $R$ is then at least a constant multiple of $n^{-1}$, and some vertex outside $A$ is much heavier than a typical light vertex and has a large neighborhood in $R$.
	
	\begin{prop}\label{prop:Rlarge}
		Let $2\le s\le\min(k,t)$ and $t\ge j$, let $t$ divide $h=n-s+1$, and let $G=G_n$ be an $n$-vertex $K_{s,t}$-minor-free graph maximizing $\rho(\CC_k(G))$, with Perron vector $x$. Let $A$ be the set of the $s-1$ largest coordinates, $T=N(A)$ and $R=V(G)\setminus(A\cup T)$. There is a constant $c_{*}=c_{*}(k,s,t)>0$ such that, for all sufficiently large $n$, if $R\ne\emptyset$ then
		\[
		W_R\ \ge\ c_{*}\,\frac{\mu}{n},\qquad
		\eta\ \ge\ c_{*}\,n^{-\frac{j-1}{jk}},\qquad
		\deg_R(v)\ \ge\ c_{*}\,n^{\frac{k-1}{jk}}\ \text{ for some vertex } v\notin A,
		\]
		where $\deg_R(v)$ is the number of neighbors of $v$ in $R$.
	\end{prop}
	
	\begin{proof}
		Throughout, $C_1,C_2,\dots$ and $c_1,c_2,\dots$ are positive constants depending only on $k,s,t$. Put $u^{k}=\frac1{s-1}\sum_{a\in A}x_a^{k}$, $W_0=1-(s-1)u^{k}$ and $W_T=\sum_{v\in T}x_v^{k}$, so that $W_T+W_R=W_0$, and write $\Xi(h',W)$ for the right side of~\eqref{eq:termwise} with this $u$; let $\Xi_0=u^{s-1}KW_0^{j/k}h^{(s-1)/k}$ be its term $r=j$ at $(h,W_0)$. By Theorem~\ref{thm:stab}, or Theorem~\ref{thm:jle1stab} when $j=1$, $u^{k}\to\frac1k$, $W_0\to\frac jk$, $W_R\to0$ and $\eta\to0$, so that $\Xi_0=(1+o(1))\frac ckn^{(s-1)/k}$, and by Theorem~\ref{thm:exact}, or Theorem~\ref{thm:jle1stab} when $j=1$, $\rho(\CC_k(G))=kP\le(1+o(1))k\,\Xi_0$. We use three facts about $R$. Since $r\in R$ misses a vertex of $A$, every $k$-clique through $r$ contains at most $s-2$ vertices of $A$. Since $A\cup\{r\}$ has $s$ vertices, $r$ has at most $t-1$ neighbors in $T$ by Lemma~\ref{lem:common}. And $G-A$ is $K_{s,t}$-minor-free, hence $d$-degenerate; we fix a degeneracy ordering of $G-A$, so that every clique $Q$ of $G-A$ satisfies $Q\subseteq\{w\}\cup N^{+}(w)$ for its earliest vertex $w$, and at most $2^{d}$ cliques share the same earliest vertex.
		
		\emph{Step 1: comparison with the candidate.} The $k$-cliques inside $A\cup T$ consist of $k-r$ vertices of $A$ and an $r$-clique of $G[T]$, and $\Delta(G[T])\le t-1$ by Corollary~\ref{cor:deg}, so Proposition~\ref{prop:termwise} with $H=G[T]$ bounds their weight by $\Xi(h-\mu,W_T)$; the remaining cliques meet $R$, and we write $\Xi_R$ for their weight. On the other hand $\Xi(h,W_0)$ is the value of $\CC_k(K_{s-1}\vee\frac htK_t)$ at a unit vector, so by extremality $\Xi(h,W_0)\le\rho(\CC_k(K_{s-1}\vee\frac htK_t))/k\le P$. Hence
		\[
		\Xi_R\ \ge\ \Xi(h,W_0)-\Xi(h-\mu,W_T)\ \ge\ \frac{\partial\Xi}{\partial W}(h,W_0)\,W_R+\frac{\partial\Xi}{\partial h'}(h,W_T)\,\mu ,
		\]
		since $\Xi$ is increasing and concave in each variable separately. Keeping only the term $r=j$ in each derivative,
		\begin{equation}\label{eq:Rgain}
			\Xi_R\ \ge\ \frac{j}{kW_0}\,\Xi_0\,W_R+\frac{s-1}{kh}\Bigl(\frac{W_T}{W_0}\Bigr)^{\frac jk}\Xi_0\,\mu\ \ge\ (1-o(1))\,\Xi_0\,W_R+(1-o(1))\frac{s-1}{k}\,\Xi_0\,\frac{\mu}{h},
		\end{equation}
		using $W_0\to\frac jk$ and $W_T/W_0\to1$.
		
		\emph{Step 2: the eigenequation on $R$.} Summing the eigenequation $\rho x_r^{k}=x_r\sum_{F\ni r}\prod_{v\in F\setminus\{r\}}x_v$ over $r\in R$ counts every clique meeting $R$ at least once, so
		\begin{equation}\label{eq:XiR}
			\Xi_R\ \le\ \rho\,W_R\ \le\ (1+o(1))\,k\,\Xi_0\,W_R .
		\end{equation}
		Combining~\eqref{eq:Rgain} and~\eqref{eq:XiR} and dividing by $\Xi_0$,
		\[
		(k-1+o(1))\,W_R\ \ge\ (1-o(1))\frac{s-1}{k}\,\frac{\mu}{h},
		\]
		which gives $W_R\ge c_1\mu/n$ with $c_1=\frac{s-1}{2k(k-1)}$ for large $n$. This is the first assertion, and we may take $c_{*}\le c_1$.
		
		\emph{Step 3: an upper bound on the mass of $R$.} Let $r\in R$ and let $F\ni r$ be a $k$-clique. Its vertices outside $A\cup\{r\}$ form a clique $Q'$ of $G-A$ with $|Q'|\ge j$, contained in $N(r)\setminus A$, and $\prod_{v\in F\setminus\{r\}}x_v\le x_w\eta^{|Q'|-1}\le x_w\eta^{j-1}$ for the earliest vertex $w$ of $Q'$, the coordinates on $A$ being at most $1$. At most $2^{d+s-1}$ cliques $F\ni r$ share the same $w$, and $N(r)\setminus A$ consists of at most $t-1$ vertices of $T$, each of coordinate at most $\eta$, and the $R$-neighbors of $r$. Hence
		\[
		\rho\,x_r^{k-1}\ \le\ 2^{d+s-1}\eta^{j-1}\Bigl((t-1)\eta+\sum_{w\in N(r)\cap R}x_w\Bigr).
		\]
		Multiplying by $x_r$ and summing over $R$, the last sum contributes $2\sum_{rw\in E(G[R])}x_rx_w\le2\sum_rx_r\sum_{w\in N^{+}(r)\cap R}x_w\le2d\eta\sum_{r\in R}x_r$ by degeneracy, so that
		\[
		\rho\,W_R\ \le\ C_2\,\eta^{j}\sum_{r\in R}x_r\ \le\ C_2\,\eta^{j}\,\mu^{\frac{k-1}{k}}W_R^{\frac1k}
		\]
		by H\"older's inequality, that is, $W_R\le(C_2\eta^{j}/\rho)^{k/(k-1)}\mu$. Together with Step~2 this gives $(C_2\eta^{j}/\rho)^{k/(k-1)}\ge c_1/n$. Since $\rho\ge c_3n^{(s-1)/k}$ by Theorem~\ref{thm:bip},
		\[
		\eta^{j}\ \ge\ c_4\,\rho\,n^{-\frac{k-1}{k}}\ \ge\ c_5\,n^{\frac{s-1}{k}-\frac{k-1}{k}}=c_5\,n^{-\frac{j-1}{k}},
		\]
		which is the second assertion.
		
		\emph{Step 4: the eigenequation at a heaviest vertex outside $A$.} Let $v\notin A$ satisfy $x_v=\eta$. A clique $F\ni v$ containing all of $A$ has its remaining $j-1$ vertices in $N(A)\cap N(v)$, a set of at most $t-1$ vertices by Lemma~\ref{lem:common}, so these cliques contribute at most $\binom{t-1}{j-1}\eta^{j-1}$ to $\rho\eta^{k-1}$. The cliques with at most $s-2$ vertices of $A$ contribute, exactly as in Step~3, at most $2^{d+s-1}\eta^{j-1}\bigl((t-1)\eta+\eta\deg_R(v)\bigr)$. Dividing by $\eta^{j-1}$,
		\[
		\rho\,\eta^{s-1}\ \le\ C_6\bigl(1+\eta\deg_R(v)\bigr).
		\]
		By Step~3, $\rho\eta^{s-1}\ge c_3c_5^{(s-1)/j}n^{\frac{s-1}{k}-\frac{(j-1)(s-1)}{jk}}=c_7\,n^{\frac{s-1}{jk}}\to\infty$, so for large $n$ the term $1$ is negligible and $\deg_R(v)\ge\rho\eta^{s-2}/(2C_6)$. Using $\rho\ge c_3n^{(s-1)/k}$ and $\eta\ge c_5^{1/j}n^{-(j-1)/(jk)}$ once more,
		\[
		\deg_R(v)\ \ge\ c_8\,n^{\frac{s-1}{k}-\frac{(j-1)(s-2)}{jk}}=c_8\,n^{\frac{j(s-1)-(j-1)(s-2)}{jk}}=c_8\,n^{\frac{k-1}{jk}},
		\]
		using $j+s-2=k-1$. This is the third assertion.
	\end{proof}
	
	Proposition~\ref{prop:Rlarge} says that if an extremal graph has $R\ne\emptyset$, then some vertex $v$ outside $A$ carries a coordinate far above the typical light value $n^{-1/k}$ and has a neighborhood of polynomial size in $R$. We exclude this possibility by controlling the mass on $R$ and its neighbors.
	
	Fix $\lambda\ge1$ and call a vertex $q\notin A$ \emph{semi-heavy} if $x_q>\lambda n^{-1/k}$ and \emph{light} otherwise; let $Q^{*}$ be the set of semi-heavy vertices. In the candidate every vertex outside the hubs has coordinate $(1+o(1))(j/(kh))^{1/k}$, so for $\lambda$ large there are none. The following is the last ingredient.
	
	\begin{prop}\label{prop:Rempty}
		Let $2\le s\le\min(k,t)$ and $t\ge j$, let $t$ divide $h=n-s+1$, and let $G$ be an extremal $K_{s,t}$-minor-free graph with $A,T,R$ as in Proposition~\ref{prop:Rlarge}. Then $R=\emptyset$ for all sufficiently large $n$.
	\end{prop}
	
	\begin{proof}
		If $j=1$, Proposition~\ref{prop:Rlarge} gives $\eta\ge c_*>0$ whenever $R\ne\emptyset$, contradicting Theorem~\ref{thm:jle1stab}. We may therefore assume $j\ge2$.
		Suppose $R\ne\emptyset$, and let
		\[
		S=R\cup\bigcup_{r\in R}\bigl(N_G(r)\cap T\bigr).
		\]  A vertex of $R$ has at most $t-1$ neighbors in $T$ by Lemma~\ref{lem:common}, so $|S|\le t\mu$. Let $\lambda$ be a large constant to be fixed, let $Q^{*}$ be the set of semi-heavy vertices, and put $W(X)=\sum_{v\in X}x_v^{k}$. We bound $\Xi_R$ from above and below.
		
		\emph{Upper bound.} Let $F$ be a $k$-clique meeting $R$; it contains at most $s-2$ hubs. Let $v_1,\dots,v_s$ be its $s$ heaviest vertices in nonincreasing order, hubs first, so that $v_{s-1},v_s\notin A$, and let $w$ be the earliest vertex of $\{v_1,\dots,v_s\}\setminus A$ in a fixed degeneracy ordering of $G-A$, which is $d$-degenerate. As in the proof of Lemma~\ref{lem:low}, with $x_{(\ell+1)}$ replaced by $\eta$, the weight of $F$ is at most $\eta\,x_w^{j}$; and $F$ is determined by $w$, the hubs in $F$, the other non-hub vertices among $v_1,\dots,v_s$, which form a clique of $G-A$ lying in $N^{+}(w)$, and the remaining $j-1$ vertices, which are common neighbors of the $s$ vertices $v_1,\dots,v_s$. Hence at most $2^{s-1}2^{d}\binom{t-1}{j-1}$ cliques are charged to a given $w$. Since $F$ contains a vertex of $R$ and $w\in F\setminus A$, we have $w\in S$. By H\"older's inequality,
		\[
		\Xi_R\ \le\ C\,\eta\sum_{w\in S}x_w^{j}\ \le\ C\,\eta\,|S|^{\frac{s-1}{k}}\,W(S)^{\frac jk}\ \le\ C\,\eta\,(t\mu)^{\frac{s-1}{k}}\,W(S)^{\frac jk}.
		\]
		
		\emph{The mass of $S$.} We claim that $W(S)\le C_\lambda W_R$. The light vertices of $S\setminus R$ contribute at most $t\mu\lambda^{k}n^{-1}\le(t\lambda^{k}/c_*)W_R$, by the first assertion of Proposition~\ref{prop:Rlarge}. For the semi-heavy vertices, let $g\in Q^{*}$ and write $\rho x_g^{k-1}=a_g+b^{T}_g+b^{R}_g$, where $a_g$ comes from the cliques through $g$ containing all of $A$, $b^{T}_g$ from those with at most $s-2$ hubs contained in $A\cup T$, and $b^{R}_g$ from those meeting $R$; one of the three is at least $\frac13\rho x_g^{k-1}$. If $b^{R}_g\ge\frac13\rho x_g^{k-1}$, then $\rho x_g^{k}\le3x_gb^{R}_g$, and $\sum_gx_gb^{R}_g$ counts each clique meeting $R$ at most $k$ times, so the total mass of these $g$ is at most $3k\,\Xi_R/\rho\le3kW_R$. If $a_g\ge\frac13\rho x_g^{k-1}$, then $g\in T$ and $a_g>0$, so $N(A)\cap N(g)\ne\emptyset$. Put $\theta_g=\max\{x_{g'}:g'\in N(A)\cap N(g)\}$, a maximum over at most $t-1$ vertices by Lemma~\ref{lem:common}, so that $a_g\le\pi_A\binom{t-1}{j-1}\theta_g^{j-1}$. If $\theta_g\le x_g$ this gives $x_g^{s-1}\le C/\rho\le C'n^{-(s-1)/k}$ and hence $x_g\le C''n^{-1/k}$, excluded for $\lambda$ large; so $\theta_g=x_{g'}$ for a semi-heavy $g'\in N(A)\cap N(g)$, and $x_g^{k}\le\frac C\rho x_{g'}^{j}\le\frac{C}{\lambda^{s-1}}x_{g'}^{k}$, using $x_{g'}^{j}=x_{g'}^{k}x_{g'}^{-(s-1)}$, $x_{g'}>\lambda n^{-1/k}$ and $\rho\ge c_3n^{(s-1)/k}$. If $b^{T}_g\ge\frac13\rho x_g^{k-1}$, then $g\in T$, and the cliques in question are $\{g\}\cup I\cup Q'$ with $I\subseteq A$, $|I|\le s-2$, and $Q'$ a clique of $G[N(g)\cap T]$ with $|Q'|\ge j$; since $|N(g)\cap T|\le t-1$ by Corollary~\ref{cor:deg}, there are boundedly many of them, and with $\theta_g=\max_{g'\in N(g)\cap T}x_{g'}$ we get $\rho x_g^{k-1}\le C\theta_g^{j}$. If $\theta_g\le\lambda n^{-1/k}$ this gives $x_g^{k-1}\le C\lambda^{j}n^{-j/k}/\rho\le C'\lambda^{j}n^{-1}$, using $j+s-1=k$, hence $x_g\le C''\lambda^{j/(k-1)}n^{-1/(k-1)}<\lambda n^{-1/k}$ for large $n$, a contradiction; so $\theta_g=x_{g'}$ for a semi-heavy $g'\in N(g)\cap T$, and $x_g^{k}\le\eta\,\frac C\rho x_{g'}^{j}\le\frac{C\eta}{\lambda^{s-1}}x_{g'}^{k}$ as before. In both of the last two cases $g,g'\in T$, and a vertex $g'$ arises from at most $2(t-1)$ vertices $g$, since $g\in N(A)\cap N(g')$ or $g\in N(g')\cap T$. Summing,
		\[
		W(Q^{*})\ \le\ 3kW_R+\frac{C''}{\lambda^{s-1}}\,W(Q^{*}),
		\]
		so $W(Q^{*})\le6kW_R$ once $\lambda^{s-1}\ge2C''$, which fixes $\lambda$. Hence $W(S)\le W_R+(t\lambda^{k}/c_*)W_R+6kW_R=C_\lambda W_R$.
		
		\emph{Lower bound.} By Step~1 of the proof of Proposition~\ref{prop:Rlarge}, which uses only extremality and Proposition~\ref{prop:termwise}, $\Xi_R\ge(1-o(1))\Xi_0W_R\ge(1-o(1))\frac\rho kW_R$.
		
		\emph{Conclusion.} Combining the three estimates,
		\[
		\frac{\rho}{2k}\,W_R\ \le\ C\,\eta\,(t\mu)^{\frac{s-1}{k}}\bigl(C_\lambda W_R\bigr)^{\frac jk},
		\qquad\text{so}\qquad
		W_R\ \le\ \Bigl(\frac{C'\eta}{\rho}\Bigr)^{\frac{k}{s-1}}\mu\ \le\ C''\,\eta^{\frac{k}{s-1}}\,\frac{\mu}{n},
		\]
		using $\rho^{k}\ge c\,n^{s-1}$. Since $W_R\ge c_*\mu/n$ by Proposition~\ref{prop:Rlarge} and $\eta\to0$ by Theorem~\ref{thm:stab}, or by Theorem~\ref{thm:jle1stab} when $j=1$, this is impossible for large $n$.
	\end{proof}
	
	The proof uses only Proposition~\ref{prop:termwise}, Lemma~\ref{lem:common}, and the first assertion and Step~1 of Proposition~\ref{prop:Rlarge}. The hypothesis $t\mid h$ enters at one point only, in Step~1, where extremality is compared with $\CC_k\bigl(K_{s-1}\vee\frac htK_t\bigr)$. For a general residue that graph does not exist and the comparison must be made with $\CC_k\bigl(K_{s-1}\vee(bK_t\cup K_p)\bigr)$, which is smaller; the loss is a constant multiple of one unit of $\mu$, and this costs exactly a constant in the conclusion.
	
	\begin{prop}\label{prop:Rbdd}
		Let $2\le s\le\min(k,t)$ and $t\ge j$, write
		$h=n-s+1=bt+p$ with $0\le p<t$, and put
		\[
		\kappa_p=p-\frac{t\binom pj}{\binom tj}\in[0,p],
		\]
		which equals $p$ exactly when $p<j$.
		If $G$ is an $n$-vertex $K_{s,t}$-minor-free
			graph maximizing $\rho(\CC_k(G))$, and $A,T,R$ are as in
			Proposition~\ref{prop:Rlarge}, then, for all sufficiently
		large $n$,
		\[
		|R|\le\frac{k}{s-1}\kappa_p<\frac{kt}{s-1}.
		\]
		Moreover, for a normalized Perron vector \(x\),
		\[
		W_R=\sum_{r\in R}x_r^k=O(n^{-1}),
		\]
		where the implicit constant depends only on \(k,s,t\). In particular $R=\emptyset$ when $t\mid h$, which is
		Proposition~\ref{prop:Rempty}.
	\end{prop}
	
	\begin{proof}
		Use $u,W_0,W_T$ as in Step~1 of
			Proposition~\ref{prop:Rlarge}, and put $\mu=|R|$.
		Write $\Xi(h',W)$ for the right side of~\eqref{eq:termwise}
		and $\Xi_j$ for its term of index $r=j$,
		evaluated at $(h,W_0)$.
		Evaluating the clique sum for $bK_t\cup K_p$ at the uniform
		light vector gives $\Xi(h,W_0)-\delta_p$, where
		\[
		\delta_p
		=\sum_{r\ge j}\binom{s-1}{k-r}u^{k-r}
		\left(\frac pt\binom tr-\binom pr\right)
		\left(\frac{W_0}h\right)^{r/k}
		=\bigl(\kappa_p+o(1)\bigr)\frac{\Xi_j}{h}.
		\]
		Indeed, the term $r=j$ gives $\kappa_p\Xi_j/h$, and the
		remaining terms are of smaller order.

		By stability,
		$W_0\to j/k$, $W_T/W_0\to1$, $\eta\to0$ and
		$\rho=(1+o(1))k\Xi_j$.
		The comparison and concavity calculation in Step~1 of
		Proposition~\ref{prop:Rlarge} therefore give, with
		$\alpha=(s-1)/k$,
		\[
		\Xi_R
		\ge \Xi(h,W_0)-\delta_p-\Xi(h-\mu,W_T)
		\ge \frac{j}{kW_0}\Xi_jW_R
		+\left[
		\alpha\left(\frac{W_T}{W_0}\right)^{j/k}\mu
		-\kappa_p-o(1)
		\right]\frac{\Xi_j}{h}.
		\]
		Suppose that $\mu>\kappa_p/\alpha$. Since $\mu$ is an integer,
		\[
		\mu\ge m_p:=\left\lfloor\frac{\kappa_p}{\alpha}\right\rfloor+1,
		\qquad
		\alpha\mu-\kappa_p\ge\gamma_p\mu,
		\qquad
		\gamma_p:=\alpha-\frac{\kappa_p}{m_p}>0.
		\]
		Thus, for all sufficiently large $n$, the errors in the
		comparison can be absorbed to yield
		\[
		\Xi_R\ge\frac12\Xi_jW_R
		+\frac{\gamma_p}{2}\frac{\Xi_j\mu}{h}.
		\]
		The eigenequations on $R$ give $\Xi_R\le\rho W_R$, so
		\[
		W_R\ge c_1\frac{\mu}{n},
		\qquad
		\Xi_R\ge c_2\rho W_R
		\]
		for fixed $c_1,c_2>0$.
		The upper bound and mass estimate in the proof of
		Proposition~\ref{prop:Rempty} for $j\ge2$, and Step~3 of
		Proposition~\ref{prop:Rlarge} for $j=1$, now give
		\[
		W_R\le C\eta^{k/(s-1)}\frac{\mu}{n}.
		\]
		These estimates do not require divisibility once the two
		lower bounds above are available. Since $\eta\to0$, this
		contradicts $W_R\ge c_1\mu/n$. Hence
		$\mu\le k\kappa_p/(s-1)$.

		Finally, $m\mapsto\binom mj$ is convex on the nonnegative
		integers and vanishes at $m=0$, so
		$\binom pj\le\frac pt\binom tj$ for $0\le p\le t$.
		Hence $\kappa_p\ge0$, with $\kappa_0=0$, while
		$\kappa_p\le p\le t-1<t$ gives the second inequality
		of the statement.
		
		It remains to bound \(W_R\).
		The bound just proved gives \(|R|\le M\) for a constant
		\(M\) depending only on \(k,s,t\).
		Every vertex of \(T\) has at most \(t-1\) neighbors in \(T\)
		by Corollary~\ref{cor:deg}, and at most \(M\) neighbors in \(R\).
		Every vertex of \(R\) has at most \(t-1\) neighbors in \(T\)
		by Lemma~\ref{lem:common}, and at most \(M-1\) neighbors in \(R\).
		Consequently,
		\[
		\Delta(G-A)\le t-1+M.
		\]
		
		Put \(\eta=\max_{v\notin A}x_v\).
		If \(\eta=0\), then \(W_R=0\).
		Otherwise, choose \(v\notin A\) with \(x_v=\eta\).
		The degree of \(v\) in \(G\) is bounded in terms of \(k,s,t\),
		so only boundedly many \(k\)-cliques contain \(v\).
		Every such clique has at least \(j-1\) further vertices
		outside \(A\). Since the coordinates outside \(A\) are at most
		\(\eta\), and all coordinates are at most \(1\), the
		eigenequation at \(v\) gives
		\[
		\rho\,\eta^{k-1}\le C\eta^{j-1}.
		\]
		Here \(\rho=\rho(C_k(G))\), and \(C\) depends only on \(k,s,t\).
		Using \(k-j=s-1\) and
		\(\rho=\Theta(n^{(s-1)/k})\), we obtain
		\[
		\eta=O(n^{-1/k}),
		\qquad
		W_R\le |R|\eta^k=O(n^{-1}),
		\]
		as required.
	\end{proof}

	\begin{lem}\label{lem:Acliquelow}
		Let $2\le s\le\min(k-1,t)$ and $t\ge j$, and let $G$ be an extremal $K_{s,t}$-minor-free graph with $A$ as in Proposition~\ref{prop:Rlarge}. For all sufficiently large $n$, the set $A$ induces a clique.
	\end{lem}
	
	\begin{proof}
		Otherwise every $k$-clique contains at most $s-2$ vertices of $A$, and Lemma~\ref{lem:low} with $\ell=s-1$ gives $\rho(\CC_k(G))\le kC_1x_{(s)}n^{(s-1)/k}+kC_1(s-1)$. By Theorem~\ref{thm:stab}, $x_{(s)}\to0$, so $\rho(\CC_k(G))=o(n^{(s-1)/k})$, contradicting Corollary~\ref{cor:lb}.
	\end{proof}
	
	By Proposition~\ref{prop:Rbdd} and Lemma~\ref{lem:Acliquelow}, for all large $n$ and every residue, an extremal graph $G$ with $2\le s\le k-1$ satisfies $G-R=K_{s-1}\vee G[T]$ with $|R|<kt/(s-1)$. The light part is determined in Section~\ref{sec:kgs}.
	
	\subsection*{The regime $s\ge k$}
	
	For $s\ge k$ the parameter $j=k-s+1$ is at most $1$ and, when $s>k$, the argument in Lemma~\ref{lem:low} sorting the $s$ heaviest vertices of a $k$-clique is unavailable. On the other hand the extremal structure is easier to reach, because a vertex outside the common neighborhood of the dominating set loses a constant factor rather than a lower-order term. Throughout this subsection $k\ge3$, $k\le s\le t$, $G$ is an extremal $K_{s,t}$-minor-free graph, $x$ is its Perron vector, $A$ is the set of the $s-1$ largest coordinates, $T=N(A)$, $R=V(G)\setminus(A\cup T)$ and $\eta=\max_{v\notin A}x_v$. By Theorem~\ref{thm:jle1stab} the coordinates on $A$ tend to $\bigl(\frac{k-1}{k(s-1)}\bigr)^{1/k}$, $\eta\to0$ and $|R|=o(n)$.
	
	\begin{lem}\label{lem:Aclique}
		For all sufficiently large $n$, $A$ induces a clique.
	\end{lem}

	\begin{proof}
		Suppose, along an infinite sequence of orders, that $a,a'\in A$
		are not adjacent.  Put $\alpha=(k-1)/k$,
		$Q=e_{k-1}(x_A)$, and $W=\sum_{v\notin A}x_v^k$.
		By Theorem~\ref{thm:jle1stab}, all coordinates on $A$ tend to
		$z_0=((k-1)/(k(s-1)))^{1/k}>0$, and
		$\eta=\max_{v\notin A}x_v\to0$.
		The coefficient $q_v$ for a clique with its unique light vertex at
		$v\notin A$ is the sum over $(k-1)$-cliques of $G[A\cap N(v)]$.
		No such clique contains both $a$ and $a'$, and therefore, for every
		$v\notin A$,
		\[
		q_v\le Q-\omega_n,
		\qquad
		\omega_n=x_a x_{a'}e_{k-3}(x_{A\setminus\{a,a'\}})
		\longrightarrow \binom{s-3}{k-3}z_0^{k-1}>0.
		\]
		Here $e_0=1$, covering $k=3$.
		The cliques inside $A$ have bounded total weight, and
		Lemma~\ref{lem:tail}, with threshold $\eta$, makes the total weight
		of cliques with at least two vertices outside $A$ equal to $o(n^\alpha)$.
		Consequently H\"older's inequality gives
		\[
		\frac{\rho(\CC_k(G))}{k n^\alpha}
		\le (Q-\omega_n)W^{1/k}+o(1).
		\]
		But $W\to1/k$ and
		$Q\to\binom{s-1}{k-1}z_0^{k-1}$, so
		$QW^{1/k}\to c/k$, where $c$ is the sharp leading constant in
		Theorem~\ref{thm:jle1stab}.
		The positive limit of $\omega_n$ contradicts the lower bound
		$\rho(\CC_k(G))\ge(c-o(1))n^\alpha$ for an extremal graph.
		Thus $A$ induces a clique for all sufficiently large $n$.
	\end{proof}

	\begin{prop}\label{prop:Rempty2}
		$R=\emptyset$ for all sufficiently large $n$.
	\end{prop}
	
	\begin{proof}
		Suppose $R\ne\emptyset$. Since $G$ is $d$-degenerate, $G[R]$ has a vertex $r$ with $|N(r)\cap R|\le d$; and $|N(r)\cap T|\le t-1$ by Lemma~\ref{lem:common} applied to the $s$ vertices of $A\cup\{r\}$. Let $G'$ be obtained from $G$ by deleting all edges at $r$ and joining $r$ to every vertex of $A$. By Lemma~\ref{lem:Aclique} the graph $G'$ is the clique-sum of $G-r$ and $K_s$ along the clique $A$, whose order $s-1$ is smaller than the connectivity $s$ of $K_{s,t}$. Both parts are $K_{s,t}$-minor-free, so $G'$ is $K_{s,t}$-minor-free by Lemma~\ref{lem:cliquesum}.
		
		The $k$-cliques avoiding $r$ are the same in $G$ and $G'$. Those containing $r$ have total weight $x_r\Sigma_r$ in $G$ and $x_r\,e_{k-1}(x_A)$ in $G'$, where $\Sigma_r=\sum_{F\ni r}\prod_{u\in F\setminus\{r\}}x_u$. Now
		\[
		\Sigma_r
		\le e_{k-1}\bigl(x_{A\cap N(r)}\bigr)
		+\binom{s+t+d}{k-1}\eta.
		\]
		since a $k$-clique through $r$ that is not contained in $A\cup\{r\}$ has a vertex of $N(r)\setminus A$, and $|N(r)\setminus A|\le t-1+d$. As $r\notin T$ it misses some $a_0\in A$, so
		\[
		e_{k-1}(x_A)-e_{k-1}\bigl(x_{A\cap N(r)}\bigr)\ \ge\ x_{a_0}e_{k-2}\bigl(x_{A\setminus\{a_0\}}\bigr)\ \ge\ c_0>0,
		\]
		because $\binom{s-2}{k-2}>0$ and the coordinates on $A$ tend to a positive limit. Hence $e_{k-1}(x_A)-\Sigma_r\ge c_0-C\eta>0$ for large $n$.
		
		If $x_r>0$ this gives $P_{G'}(x)>P_{G}(x)$ and therefore $\rho(\CC_k(G'))>\rho(\CC_k(G))$. If $x_r=0$, put $z=x+\varepsilon\mathbf 1_{\{r\}}$; then $P_{G'}(z)\ge P_G(x)+\varepsilon\,e_{k-1}(x_A)$ while $\|z\|_k^{k}=1+\varepsilon^{k}$, so that $\rho(\CC_k(G'))\ge k\bigl(P_G(x)+\varepsilon c_0\bigr)/(1+\varepsilon^{k})>\rho(\CC_k(G))$ for small $\varepsilon$, since $k\ge2$. Either way extremality is contradicted.
	\end{proof}
	
	For the remainder of this subsection we determine the extremal hypergraph for \emph{every} residue of $h$ modulo $t$. Write $\mathcal A_{s,t}$ for the class of admissible graphs, that is, of graphs $H$ for which $K_{s-1}\vee H$ is $K_{s,t}$-minor-free. By Lemma~\ref{lem:join}, $\mathcal A_{s,t}$ consists of the graphs with no $K_{a,t+1-a}$ minor for $1\le a\le s$; it is closed under minors and disjoint unions, and by the case $a=1$ every member has maximum degree at most $t-1$. Put $q=\lfloor\frac{t+1}{s+1}\rfloor$ and $E=\binom t2$.

	At order $t+1$ admissibility is completely explicit. The next two statements are Lemmas~3.1 and~3.2 of~\cite{ZL}, where the criterion is stated with $\gamma=\min\{s,\lfloor\frac{t+1}{2}\rfloor\}$ in place of $s$; the two forms are equivalent because $K_{a,b}\cong K_{b,a}$. We include the short proofs in the form used here.
	
	\begin{lem}[Zhai--Lin~\cite{ZL}]\label{lem:tplus1}
		Let $|F|=t+1$. Then $F\in\mathcal A_{s,t}$ if and only if every component of the complement $\overline F$ has at least $s+1$ vertices.
	\end{lem}
	
	\begin{proof}
		Each forbidden minor $K_{a,t+1-a}$ has exactly $t+1$ vertices, so a model of it in $F$ uses every vertex in a singleton branch set; that is, $F$ has a $K_{a,t+1-a}$ minor if and only if it has one as a spanning subgraph. A spanning $K_{a,t+1-a}\subseteq F$ is the same as $\overline F\subseteq K_a\cup K_{t+1-a}$, that is, a partition of $V(F)$ into parts of sizes $a$ and $t+1-a$ with no edge of $\overline F$ between them; such a partition exists if and only if some union of components of $\overline F$ has size $a$. Ranging over $1\le a\le s$ gives the statement.
	\end{proof}
	
	\begin{prop}[Zhai--Lin~\cite{ZL}]\label{prop:tplus1}
		Let $|F|=t+1$ and $F\in\mathcal A_{s,t}$. Then
		\[
		e(F)\ \le\ \binom t2+\Bigl\lfloor\frac{t+1}{s+1}\Bigr\rfloor-1,
		\]
		with equality when $\overline F$ is a disjoint union of $\lfloor\frac{t+1}{s+1}\rfloor$ stars, each on at least $s+1$ vertices. Consequently $e(F)>e(K_t\cup K_1)=\binom t2$ is possible if and only if $t\ge2s+1$.
	\end{prop}
	
	\begin{proof}
		By Lemma~\ref{lem:tplus1} every component of $\overline F$ has at least $s+1$ vertices, so $\overline F$ has at most $q=\lfloor\frac{t+1}{s+1}\rfloor$ components and hence at least $(t+1)-q$ edges. Thus $e(F)=\binom{t+1}2-e(\overline F)\le\binom{t+1}2-(t+1)+q=\binom t2+q-1$, and a disjoint union of $q$ stars with parts of size at least $s+1$ attains it. Finally $q\ge2$ if and only if $t+1\ge2(s+1)$.
	\end{proof}
	
	We also use two results on connected graphs with no $K_{1,t}$ minor. The first is an edge bound for connected graphs with no $K_{1,t}$ minor, due to Ding, Johnson and Seymour~\cite{DJS}; recall from Lemma~\ref{lem:join} that admissibility implies $K_{1,t}$-minor-freeness.

	\begin{lem}[Ding--Johnson--Seymour]\label{lem:djs}
		If $J$ is connected, has no $K_{1,t}$ minor and $|V(J)|\ge t+2$, then $e(J)\le E+|V(J)|-t$, where $E=\binom t2$.
	\end{lem}

	The second is the equality case of Lemma~\ref{lem:djs} at order $t+2$, which extends~\cite[Lemma~3.4]{ZL}; there it is proved under the assumption $q\le2$, which we do not impose. For nonnegative integers $e,f$ with $e+f=t-1$, let $J_{0,e,f}$ be obtained from a $(t-1)$-clique partitioned into sets of orders $e$ and $f$ by adding a path $uwv$ and joining $u$ to the first set and $v$ to the second, with no other edges between the path and the clique. For nonnegative integers $a_1,a_2,a_3$ with $a_1+a_2+a_3=t-2$, let $J'_{a_1,a_2,a_3}$ be obtained from a $(t-2)$-clique partitioned into sets $B_1,B_2,B_3$ of these orders by adding a claw with center $z$ and leaves $z_1,z_2,z_3$, and joining every vertex of $B_i$ to the two leaves other than $z_i$. Let $\overline P$ denote the complement of the Petersen graph $P$, that is, the line graph of $K_5$; it is $6$-regular with ten vertices, thirty edges and thirty triangles.

	\begin{lem}\label{lem:tplus2}
		Let $F$ be connected and admissible with $|V(F)|=t+2$ and $e(F)=E+2$. Then $F$ is isomorphic to $J_{0,e,f}$ with $e,f\ge s$, to $J'_{a_1,a_2,a_3}$ with $a_1,a_2,a_3\ge s$, or to $\overline P$ with $t=8$.
	\end{lem}

	\begin{proof}
		Let $Y=\overline F$, so $|V(Y)|=t+2$ and $e(Y)=2t-1$. If two nonadjacent vertices of $Y$ had no common neighbor in $Y$, they would be adjacent in $F$ and together adjacent in $F$ to all other $t$ vertices, and contracting their edge would give a $K_{1,t}$ minor. Hence $Y$ is connected with diameter at most $2$. A cut vertex of a graph of diameter two is adjacent to all other vertices, and would be isolated in $F$; so a minimum vertex cut $S$ of $Y$ has $|S|\ge2$, and $|S|\le\delta(Y)\le2e(Y)/|V(Y)|<4$. Put $W=V(Y)\setminus S$ and let $T_1,\dots,T_c$ be the components of $Y-S$. By the diameter condition every vertex of $W$ has a neighbor in $S$, and by minimality every $T_i$ has a neighbor at every vertex of $S$.

		Let $W_1$ be the set of vertices of $W$ with exactly one neighbor in $S$. These lie in a single component: if two of them lay in different components, their common neighbor $u\in S$ would be adjacent to all of $W$ by the diameter condition, and deleting $u$ (when $|S|=2$), or $u$ and another vertex of $S$ of degree at least three (when $|S|=3$), would leave too few edges for the remaining graph to be connected, contradicting the minimality of $S$. Index the components so that $W_1\subseteq V(T_1)$.

		Suppose $|S|=2$, $S=\{u_1,u_2\}$. Then $|W|=t$ and counting edges inside the components and from $W$ to $S$ gives $e(Y)\ge(|T_1|-1)+|T_1|+2(t-|T_1|)=2t-1$. Equality forces $u_1u_2\notin E(Y)$, makes $T_1$ a tree all of whose vertices have exactly one neighbor in $S$, and makes every other component a single vertex adjacent to both $u_1$ and $u_2$. Label each vertex of $T_1$ by its neighbor in $S$; both labels occur. Vertices of $T_1$ at distance at least three in $T_1$ have the same label, while every vertex has a neighbor in $T_1$ with the other label, by the diameter condition applied to it and the other vertex of $S$. A path of length at least four in $T_1$ would therefore start with two vertices of the same label, the first being a leaf, a contradiction. So $T_1$ is a star or a double star. If $T_1$ is a star, its center and leaves have opposite labels, and $F\cong J_{0,e,f}$, where the two parts of the clique are the leaves of $T_1$ and the other components of $Y-S$. If $T_1$ is a double star, the leaves at its two centers have the same label and the centers have the other one, and $F\cong J'_{a_1,a_2,a_3}$.

		Suppose $|S|=3$. Then $|W|=t-1$ and $e(Y)=2|W|+1$. The edges inside $T_1$ and from $T_1$ to $S$ number at least $2|T_1|$, with equality only if $T_1$ is a cycle all of whose vertices lie in $W_1$: a tree of order at least two has leaves outside $W_1$, and a single vertex has degree at least three. Every other component of order $m$ contributes at least $3m-1$ such edges if $m\ge2$ and $3$ if $m=1$, that is, at least one more than $2m$, and at least two more if $m\ge3$. The budget $2|W|+1$ therefore forces $T_1$ to be a cycle with all its vertices in $W_1$ and exactly one other component, of order one or two. The second possibility is excluded: its two vertices would have distinct pairs of neighbors in $S$ covering $S$, and by the diameter condition every vertex of $T_1$ would have its unique neighbor in their intersection, so $T_1$ would miss a vertex of $S$. Thus the other component is a single vertex adjacent to all of $S$, and $S$ is independent. Each vertex of $S$ has at least two neighbors on the cycle, so its length is at least six, and vertices at cyclic distance at least three must have the same neighbor in $S$. For length at least seven the graph joining such pairs is connected (use steps of length three, and of length four when the length exceeds seven), which would give all cycle vertices the same neighbor. So the cycle has length six, opposite vertices have the same neighbor in $S$, and the three pairs have distinct neighbors. This graph is $P$, so $t=8$ and $F\cong\overline P$.

		Finally, contracting an edge of the path $uwv$ in $J_{0,e,f}$ gives a graph on $t+1$ vertices whose complement consists of two stars with $e+1$ and $f+1$ vertices, and contracting an edge $zz_i$ in $J'_{a_1,a_2,a_3}$ gives a graph whose complement has a component with $a_i+1$ vertices. As $F$ is admissible, so are these minors, and Lemma~\ref{lem:tplus1} gives $e,f\ge s$ and $a_i\ge s$.
	\end{proof}

	\begin{lem}\label{lem:leaf}
		Let $F$ be $K_t$, or a graph on $t+1$ vertices with $E+1$ edges, or a connected admissible graph on $t+2$ vertices with $E+2$ edges. Then every vertex of $F$ lies in a connected set with at least $t-1$ neighbors outside it. Consequently, adding a new vertex of degree one to $F$ creates a $K_{1,t}$ minor.
	\end{lem}

	\begin{proof}
		For $K_t$ take a single vertex. If $|V(F)|=t+1$ and $e(F)=E+1$, then $\overline F$ has $t-1$ edges and is disconnected; for a prescribed vertex $x$ choose $w$ in another component of $\overline F$, so that $xw\in E(F)$ and every other vertex is adjacent in $F$ to $x$ or $w$. In the last case we use Lemma~\ref{lem:tplus2}. In $J_{0,e,f}$ the path $uwv$ has all $t-1$ clique vertices as neighbors, and a clique vertex, together with its neighbor on the path and $w$, has the other $t-2$ clique vertices and the other end of the path as neighbors. In $J'_{a_1,a_2,a_3}$ a set consisting of $z$ and two of its leaves has all $t-2$ clique vertices and the third leaf as neighbors, and such a set can be chosen to contain any prescribed vertex of the claw; and a vertex of $B_i$, together with a leaf $z_{i'}$ adjacent to it and $z$, has the other $t-3$ clique vertices and the other two leaves as neighbors. In $\overline P$, the two ends of any edge have exactly seven neighbors outside, since the corresponding nonadjacent pair of $P$ has exactly one common neighbor. In each case the new vertex of degree one adds a $t$-th outside neighbor to a connected set containing its neighbor.
	\end{proof}

	For $q\ge2$ let
	\[
	F_{s,t}=\overline{(q-1)K_{1,s}\cup K_{1,\,t-(q-1)(s+1)}},
	\]
	which by Lemma~\ref{lem:tplus1} and Proposition~\ref{prop:tplus1} is admissible on $t+1$ vertices and has the maximum number $E+q-1$ of edges. When $q=2$ the two centers of the complementary stars are adjacent in $F_{s,t}$; let $D_{s,t}$ be obtained by subdividing that edge once.

	\begin{lem}\label{lem:exadm}
		If $q=2$, then $D_{s,t}$ is admissible and has $t+2$ vertices and $E+2$ edges. If $t=8$, then $\overline P$ is admissible for every $s\le8$.
	\end{lem}

	\begin{proof}
		Let $J$ be a connected graph on $t+2$ vertices. Each forbidden minor $K_{a,t+1-a}$ has $t+1$ vertices, so a model of it in $J$ either consists of singletons on $t+1$ vertices, and is then a spanning subgraph of $J-x$ for the unused vertex $x$, or has exactly one branch set of order $2$, which is an edge $xy$, and is then a spanning subgraph of $J/xy$. Since $J$ is connected, $J-x$ is a spanning subgraph of $J/xy$ for any neighbor $y$ of $x$. Hence $J$ is admissible as soon as $J/e$ is admissible for every edge $e$, and by Lemma~\ref{lem:tplus1} this holds when every $\overline{J/e}$ is connected, since it has $t+1\ge s+1$ vertices.

		Let $q=2$ and $\ell=t-s-1$, so that $\ell\ge s\ge3$. The graph $D=D_{s,t}$ is $J_{0,s,\ell}$: a clique partitioned into sets $B_0,C_0$ of orders $s,\ell$ and a path $uwv$, with $u$ joined to $B_0$, $v$ joined to $C_0$, and no other edges between the path and the clique. It has $(t-1)+2=t+2$ vertices and $\binom{t-1}2+s+\ell+2=E+2$ edges. Its complement $\overline D$ has the edge $uv$, the edges from $u$ to $C_0$, from $v$ to $B_0$, and from $w$ to $B_0\cup C_0$. Contracting $uw$ or $wv$ gives $F_{s,t}$, whose complement is connected on each of its two stars and which is admissible by Lemma~\ref{lem:tplus1} because both stars have at least $s+1$ vertices. If $e$ joins two clique vertices, then in $\overline{D/e}$ every vertex coming from the clique is adjacent to $w$, the vertex $v$ is adjacent to the nonempty set $B_0\setminus e$, and $u$ is adjacent to $v$; so $\overline{D/e}$ is connected. If $e=ub$ with $b\in B_0$, the contracted vertex is adjacent in $D/e$ to every vertex except $v$, so in the complement it is adjacent to $v$; moreover $v$ is adjacent to $B_0\setminus\{b\}\ne\emptyset$, which is adjacent to $w$, which is adjacent to $C_0$. So $\overline{D/e}$ is connected. The case $e=vc$ with $c\in C_0$ is symmetric, using $|C_0\setminus\{c\}|=\ell-1\ge2$. These are all the edges of $D$.

		Let $t=8$. An edge $xy$ of $\overline P$ joins two nonadjacent vertices of $P$, which have exactly one common neighbor $z$ in $P$. In $\overline{\overline P/xy}$ the contracted vertex is therefore adjacent only to $z$, and the rest is $P-\{x,y\}$, which is connected since $P$ is $3$-connected. So $\overline{\overline P/xy}$ is connected, and $\overline P$ is admissible for every $s\le t=8$.
	\end{proof}

	By Lemma~\ref{lem:Aclique} and Proposition~\ref{prop:Rempty2} an extremal graph is $G=K_{s-1}\vee H$ with $H$ admissible, and $\CC_k(G)$ is connected because $s-1\ge k-1$. Its Perron vector is constant on the core by symmetry; let $u$ be that value and $y=x/u$, so $y\equiv1$ on the core. Write $\lambda=\rho(\CC_k(G))$, $\tau=s-k+1\ge1$, $B_i=\binom{s-1}{k-i}$, $B=B_1$ and $\xi=(B/\lambda)^{1/(k-1)}\to0$.

	\begin{lem}\label{lem:yexp}
		Uniformly over the light vertices, $y_v=\xi\bigl(1+d_H(v)\xi/\tau+O(\xi^{2})\bigr)$.
	\end{lem}

	\begin{proof}
		Since $\Delta(H)\le t-1$ and $\max_v y_v\to0$, the eigenequation at a light vertex reads $\lambda y_v^{k-1}=B+B_2\sum_{w\in N_H(v)}y_w+O\bigl((\max_w y_w)^2\bigr)$. The right side is $B+O(\max_wy_w)$, which first gives $\max_wy_w=O(\xi)$ and then $y_v=\xi(1+O(\xi))$; expanding once more and using $B_2/B=(k-1)/\tau$ gives the statement.
	\end{proof}

	For a graph $J$ put $D_2(J)=\sum_vd_J(v)^2$, let $c_3(J)$ be its number of triangles, and set
	\[
	\Omega(J)=D_2(J)+\theta\,c_3(J),\qquad \theta=\frac{2(k-2)\tau}{\tau+1}>0 .
	\]

	\begin{lem}\label{lem:replace}
		Fix $M$. For all sufficiently large $n$, let $J$ be a union of components of $H$ with $|V(J)|\le M$, and let $J'$ be admissible of the same order. Then neither $e(J')>e(J)$, nor $e(J')=e(J)$ together with $\Omega(J')>\Omega(J)$, is possible.
	\end{lem}

	\begin{proof}
		Write $S_i(J,y)=\sum_{Q\in\CC_i(J)}\prod_{v\in Q}y_v$ and let $L_J(y)=\sum_{i\ge1}B_iS_i(J,y)-\frac\lambda k\sum_{v\in J}y_v^{k}$ be the contribution of $J$ to $P_G(y)-\frac\lambda k\|y\|_k^{k}$. Substituting $y_v=\xi(1+w_v\xi+O(\xi^{2}))$ with bounded $w_v$ and expanding,
		\[
		\begin{aligned}
		L_J(y)={}&B\frac{k-1}{k}|V(J)|\xi+B_2e(J)\xi^{2}\\
		&+\xi^{3}\Bigl(B_3c_3(J)+B_2\sum_vd_J(v)w_v-\frac{B(k-1)}{2}\sum_vw_v^{2}\Bigr)+O(\xi^{4}),
		\end{aligned}
		\]
		the terms in the unspecified second-order corrections cancelling because the linear and norm parts have opposite first derivatives at $y_v=\xi$. For the coordinates of Lemma~\ref{lem:yexp}, and for the trial coordinates $y'_v=\xi(1+d_{J'}(v)\xi/\tau)$ on a replacement, this becomes
		\begin{equation}\label{eq:LJ}
			L_J=B\frac{k-1}{k}|V(J)|\xi+B_2e(J)\xi^{2}+\frac{B_2^{2}}{2B(k-1)}\Omega(J)\xi^{3}+O(\xi^{4}),
		\end{equation}
		using $1/\tau=B_2/[B(k-1)]$ and $2B(k-1)B_3/B_2^{2}=\theta$. A bounded admissible replacement of the same order therefore strictly increases the Rayleigh quotient if it increases $e$, or preserves $e$ and increases $\Omega$: all other contributions are unchanged, the old value of $P_G(y)-\frac\lambda k\|y\|_k^{k}$ is zero, and a positive new value gives a quotient larger than $\lambda$. Admissibility of the new light part follows from Lemma~\ref{lem:join}. For fixed $M$ there are finitely many graph types, so the threshold in $n$ is uniform.
	\end{proof}

	\begin{lem}\label{lem:omega}
		For every graph $J$ one has $6c_3(J)\le D_2(J)-2e(J)$, with equality exactly when every component of $J$ is complete. Consequently, if $J'$ is a disjoint union of cliques with the same order and edge count as $J$ and $D_2(J')\ge D_2(J)$, then $\Omega(J')>\Omega(J)$ whenever $J$ has a noncomplete component. If moreover all degrees of $J$ lie in $[\ell,u]$, then $D_2(J)\le2(\ell+u)e(J)-\ell u|V(J)|$.
	\end{lem}

	\begin{proof}
		Counting adjacent pairs inside neighborhoods gives $3c_3(J)\le\sum_v\binom{d_J(v)}{2}$, with equality exactly when every neighborhood is a clique, that is, when every component is complete; this is the first inequality. The comparison of $\Omega$ follows since $\theta>0$. Finally $(d-\ell)(d-u)\le0$ gives $d^{2}\le(\ell+u)d-\ell u$, and summing over the vertices gives the last bound.
	\end{proof}

	The four lemmas below follow Lemmas~3.5 and~3.6 and Theorems~3.1--3.3 of~\cite{ZL}, with Lemma~\ref{lem:replace} in place of their local edge maximality and degree-sequence majorization. In these four lemmas we assume $t\ge4$; the proofs use $\binom t2-t\ge2$ and $t<\binom t2$, which fail for $t=3$. The case $t=3$ is treated separately in the proof of Theorem~\ref{thm:sgek}.

	\begin{lem}\label{lem:comporder}
		Let $t\ge4$. Every component of $H$ has order at most $t+3$; every component of order at most $t$ is complete; and for each $i\ne t$ at most $t-1$ components have order $i$. Consequently all but $O_t(1)$ light vertices lie in components isomorphic to $K_t$.
	\end{lem}

	\begin{proof}
		Let $J$ be a component of order $m=b_0t+p_0>t+3$ with $1\le p_0\le t$. By Lemma~\ref{lem:djs}, $e(J)\le E+m-t$, while the admissible replacement $J'=b_0K_t\cup K_{p_0}$ satisfies
		\[
		e(J')-(E+m-t)=(b_0-1)(E-t)+\tfrac12p_0(p_0-3)>0,
		\]
		using $p_0\ge4$ when $b_0=1$, and $E-t\ge2$ together with $p_0(p_0-3)/2\ge-1$ when $b_0\ge2$. If $b_0\le7$ the replacement has bounded order and Lemma~\ref{lem:replace} applies. If $b_0\ge8$, then $e(J)\le E+b_0t$ and $e(J')\ge b_0E$ give $e(J)/e(J')\le\frac1{b_0}+\frac{2}{t-1}\le\frac18+\frac23=\frac{19}{24}$ and $m/e(J')\le\frac{9}{4(t-1)}$; evaluating the replacement at the old coordinates changes $P_G$ by at least $B_2\bigl(e(J')y_-^{2}-e(J)y_+^{2}\bigr)-O(my_+^{3})>0$, where $y_\pm$ are the extreme light coordinates, since $y_-/y_+\to1$ and $y_+\to0$ by Lemma~\ref{lem:yexp}. Both cases are impossible.

		A component of order at most $t$ may be completed, which is admissible and adds edges unless it was already complete. If $t$ components had the same order $i\ne t$, replacing them by $iK_t$ increases the edges: $t\binom i2<iE$ for $i<t$; $t(E+q-1)<(t+1)E$ for $i=t+1$, because $t(q-1)\le t(t-3)/4<E$; and $t(E+r)<(t+r)E$ for $i=t+r$ with $r=2,3$. All these replacements have bounded order.
	\end{proof}

	\begin{lem}\label{lem:comp1}
		Let $t\ge4$. If $H$ has a component of order $t+1$, then $q\ge2$ and that component is isomorphic to $F_{s,t}$.
	\end{lem}

	\begin{proof}
		Let $J$ be such a component. Since an admissible graph attaining the bound of Proposition~\ref{prop:tplus1} is available, $e(J)=E+q-1$, so the complement $Y=\overline J$ is a forest with $q$ components. With $N=t+1$, since $Y$ has no triangles, $c_3(J)=\binom N3-(N-2)e(Y)+\sum_v\binom{d_Y(v)}{2}$ and $D_2(J)=\sum_v(N-1-d_Y(v))^{2}$, so with $N$ and $e(Y)$ fixed, maximizing $\Omega(J)$ amounts to maximizing $\sum_vd_Y(v)^{2}$.

		Every component of $Y$ is then a star: otherwise pick a vertex $v$ of maximum degree and a leaf $w$ not adjacent to $v$, with neighbor $u$, and replace $uw$ by $vw$; this keeps the component a tree of the same order, hence keeps $J$ admissible by Lemma~\ref{lem:tplus1}, and increases $\sum_vd_Y(v)^{2}$ by $2(d_Y(v)-d_Y(u))+2>0$, contradicting Lemma~\ref{lem:replace}. If $q=1$ this makes $J=K_t\cup K_1$, contradicting connectedness, so $q\ge2$. Finally, if two stars have orders $b_1\ge b_2\ge s+2$, transferring a leaf from the smaller to the larger keeps both orders at least $s+1$ and increases the degree-square sum by $2(b_1-b_2)+2$. Hence all but at most one star has order $s+1$, which gives $J\cong F_{s,t}$.
	\end{proof}

	\begin{lem}\label{lem:comp2}
		Let $t\ge4$. If $H$ has a component of order $t+2$, then either $q=1$, $t=8$ and the component is $\overline P$, or $q=2$ and the component is $D_{s,t}$.
	\end{lem}

	\begin{proof}
		Let $J$ be such a component. Comparison with $K_t\cup K_2$ and Lemma~\ref{lem:djs} give $E+1\le e(J)\le E+2$. If $e(J)=E+1$ then all degrees lie in $[1,t-1]$, so Lemma~\ref{lem:omega} gives $D_2(J)\le t(t-1)^{2}+2=D_2(K_t\cup K_2)$, and since $J$ is connected and not complete the comparison of $\Omega$ is strict, contradicting Lemma~\ref{lem:replace}. Hence $e(J)=E+2$. If $q\ge3$, replacing $J\cup K_t$ by $2F_{s,t}$ raises the edge count from $2E+2$ to $2E+2q-2$; so $q\le2$. By Lemma~\ref{lem:tplus2}, $J$ is $\overline P$ with $t=8$, or $J_{0,e,f}$ with $e,f\ge s$, or $J'_{a_1,a_2,a_3}$ with $a_i\ge s$; the last requires $t-2\ge3s$, hence $q\ge3$, and $J_{0,e,f}$ requires $t-1\ge2s$, hence $q\ge2$. If $q=1$, only $\overline P$ remains, with $t=8$. Suppose $q=2$. Then
		\[
		D_2(J_{0,e,f})=(t-1)^{3}+4+(e+1)^{2}+(f+1)^{2},\qquad
		c_3(J_{0,e,f})=\binom{t-1}{3}+\binom e2+\binom f2,
		\]
		and with $f=t-1-e$ both are strictly convex and symmetric functions of $e$ on $[s,t-1-s]$. Hence $\Omega$ is maximized exactly at the endpoints, both of which give $J_{0,s,t-1-s}\cong D_{s,t}$. If $\overline P$ were also available, then $t=8$ and $q=2$ force $s=k=3$, and $(D_2,c_3)(D_{3,8})=(388,44)$ against $(D_2,c_3)(\overline P)=(360,30)$, so $D_{3,8}$ wins strictly.
	\end{proof}

	\begin{lem}\label{lem:comp3}
		Let $t\ge4$. No component of $H$ has order $t+3$.
	\end{lem}

	\begin{proof}
		Let $J$ be such a component; comparison with $K_t\cup K_3$ and Lemma~\ref{lem:djs} force $e(J)=E+3$. We have $\delta(J)\ge2$: otherwise deleting a vertex $v$ of degree one leaves a connected admissible graph of order $t+2$ with $E+2$ edges, and by Lemma~\ref{lem:leaf} restoring $v$ creates a $K_{1,t}$ minor. Hence all degrees lie in $[2,t-1]$, so Lemma~\ref{lem:omega} gives $D_2(J)\le t(t-1)^{2}+12=D_2(K_t\cup K_3)$; the edge counts agree and $J$ is connected and not complete, so replacing $J$ by $K_t\cup K_3$ strictly increases $\Omega$.
	\end{proof}

	\begin{thm}\label{thm:sgek}
		Let $3\le k\le s\le t$ and write $h=n-s+1=bt+p$ with $0\le p<t$, and put $q=\lfloor\frac{t+1}{s+1}\rfloor$. For all sufficiently large $n$ the unique hypergraph of maximum spectral radius in $\HH^{(k)}_{K_{s,t}}(n)$ is $\CC_k(K_{s-1}\vee H^{*})$, and $K_{s-1}\vee H^{*}$ is the unique extremal graph, where
		\[
		H^{*}=\begin{cases}
			(b-1)K_8\cup\overline P, & q=1,\ t=8,\ p=2,\\[2pt]
			(b-1)K_t\cup D_{s,t}, & q=2,\ p=2,\\[2pt]
			(b-p)K_t\cup pF_{s,t}, & 1\le p\le2q-2,\ (q,p)\ne(2,2),\\[2pt]
			bK_t\cup K_p, & \text{otherwise.}
		\end{cases}
		\]
	\end{thm}

	\begin{proof}
		Suppose first $t\ge4$. By Lemmas~\ref{lem:comporder}--\ref{lem:comp3} the components of $H$ are cliques $K_i$ with $i\le t$, copies of $F_{s,t}$ when $q\ge2$, and at most one component $X$ of order $t+2$, which is $\overline P$ when $q=1$ and $t=8$ and is $D_{s,t}$ when $q=2$; every such $X$ has $E+2$ edges.

		At most one component lies in $\{K_i:i<t\}\cup\{X\}$. Merging two small cliques $K_i,K_\ell$ into full blocks and a residual clique gains $i\ell$ edges if $i+\ell\le t$ and $(t-i)(t-\ell)$ otherwise; the same replacement applied to $X\cup K_i$ gains $2i-1$ edges for $i\le t-2$ and $t-3$ for $i=t-1$; and $2X$ may be replaced by $2K_t\cup K_4$, gaining two edges, an $X$ occurring only for $t\ge7$.

		If $q=1$ no $F_{s,t}$ occurs, so the residue determines $H$ as $bK_t\cup K_p$ unless $t=8$ and $p=2$, where $\overline P$ has thirty edges against twenty-nine for $K_8\cup K_2$.

		Let $q\ge2$ and write $F=F_{s,t}$. There are at most $2q-2$ copies of $F$: replacing $2qF$ by $2qK_t\cup K_{2q}$ gains $q$ edges, and $2q<t$ because $s\ge3$. With $q_0=2q-1$ copies the orders and edge counts of $q_0F$ and $q_0K_t\cup K_{q_0}$ agree, and all degrees of $F$ lie between $\ell_0=(q-1)(s+1)$ and $u_0=t-1$, so Lemma~\ref{lem:omega} gives $D_2(q_0F)\le2(\ell_0+u_0)e-\ell_0u_0N$, where $e$ and $N$ are the common edge count and order. The degrees of $q_0K_t\cup K_{q_0}$ take only the values $u_0$ and $\ell_1=q_0-1$, so $D_2(q_0K_t\cup K_{q_0})=2(\ell_1+u_0)e-\ell_1u_0N$, and the difference of the two right sides is $(\ell_0-\ell_1)(u_0N-2e)\ge0$, because $\ell_0=(q-1)(s+1)>2q-2=\ell_1$ and $2e\le u_0N$. Hence $D_2(q_0F)\le D_2(q_0K_t\cup K_{q_0})$, and Lemma~\ref{lem:omega} gives a strict gain in $\Omega$. A copy of $F$ cannot coexist with a small clique $K_i$: for $i\ge q$ replacing $F\cup K_i$ by $K_t\cup K_{i+1}$ gains $i-q+1$ edges, while for $i\le q-1$ replacing $iK_t\cup K_i$ by $iF$ gains $i(q-1)-\binom i2>0$, the blocks being available by Lemma~\ref{lem:comporder}. Nor can $F$ coexist with $D_{s,t}$ when $q=2$: the pairs $F\cup D_{s,t}$ and $2K_t\cup K_3$ have the same order and $2E+3$ edges, all degrees of the former lie in $[2,t-1]$, and Lemma~\ref{lem:omega} again favors the union of cliques.

		Hence $H$ consists of full blocks together with one small clique, or with $q_0\le2q-2$ copies of $F$, or, when $q=2$, with one $D_{s,t}$. In the second case the residue forces $q_0=p$, and for $1\le p\le2q-2$ the small-clique alternative is excluded by
		\[
		e(pF)-e(pK_t\cup K_p)=p(q-1)-\binom p2=\tfrac12p(2q-p-1)>0 ,
		\]
		while for $p=0$ or $p>2q-2$ only the small-clique alternative survives. It remains to separate $2F$ from $K_t\cup D_{s,t}$ when $q=p=2$. With $\ell=t-s-1$, the graph $F$ has $t-1$ vertices of degree $t-1$ and two of degrees $\ell+1$ and $s+1$, and the subdivided edge has no common neighbor, so $D_2(D_{s,t})=D_2(F)+4$ and $c_3(D_{s,t})=c_3(F)$. Subtracting,
		\[
		D_2(K_t\cup D_{s,t})-D_2(2F)=2(s-1)(\ell-1)>0,\qquad c_3(K_t\cup D_{s,t})-c_3(2F)=s\ell>0,
		\]
		with equal edge counts, so $K_t\cup D_{s,t}$ wins by~\eqref{eq:LJ}.

		For $t=3$ we have $k=s=t=3$, and admissibility excludes $K_{1,3}$ and $K_{2,2}$ minors, so every component of $H$ is a path or a triangle. No path has order at least four: for a path $v_1\cdots v_m$ with light coordinates $y_1,\dots,y_m$, delete $v_3v_4$ and add $v_1v_3$, which keeps the light part admissible, and interchange $y_1$ and $y_4$ in the trial vector, preserving its norm. For $m\ge5$ the exact change in the clique polynomial is $2(y_4-y_1)(y_2-y_5)+y_4y_2y_3$, whose first term is $O(\xi^{4})$ and whose last is $(1+O(\xi))\xi^{3}>0$ by Lemma~\ref{lem:yexp}; for $m=4$ symmetry gives $y_1=y_4$ and only the positive triangle term remains. Completing every remaining three-vertex path to a triangle and merging the leftover $K_1$ and $K_2$ components as above leaves $H=bK_3\cup K_p$.

		Every graph listed in $H^{*}$ is admissible, by Lemmas~\ref{lem:tplus1} and~\ref{lem:exadm}, and has $h$ vertices, and for large $n$ the stated numbers of blocks are nonnegative. All the replacements used above involve finitely many graph types, so their positive gaps in $e$ or $\Omega$ have a uniform positive lower bound, and the one unbounded replacement in Lemma~\ref{lem:comporder} has the fixed ratio bound $19/24$; hence a single threshold $n_0(k,s,t)$ suffices. Finally, if $\HH$ is any maximizing hypergraph and $G=\shad\HH$, then $\CC_k(G)$ is also maximizing by Lemma~\ref{lem:reduce}, so $G\cong K_{s-1}\vee H^{*}$; as $\CC_k(G)$ is connected, strict monotonicity forces $\HH=\CC_k(G)$.
	\end{proof}

\section{Weighted clique inequalities}\label{sec:weighted}

In this section $k\ge3$ is fixed. For a graph $F$, an integer $1\le r<k$ and a nonnegative vector $z$ on $V(F)$, put
\[
S_r(F,z)=\sum_{Q\in E(\CC_r(F))}\prod_{v\in Q}z_v,\qquad
L_r(F)=\max\Bigl\{S_r(F,z):\ z\ge0,\ \sum_{v}z_v^{k}=1\Bigr\},
\]
with $L_r(F)=0$ when $V(F)=\emptyset$. The exponent $k$ in the normalization stays fixed when $r$ varies. Put
\[
\sigma_r=\frac{k}{k-r},\qquad C_r=\frac1t\binom tr=\frac1r\binom{t-1}{r-1},\qquad
\mathcal E_r(F)=\Bigl(\frac{L_r(F)}{C_r}\Bigr)^{\sigma_r},
\]
and call $\mathcal E_r(F)$ the \emph{capacity} of $F$. It depends on $k$ and $t$, which are fixed throughout. By Lemma~\ref{lem:sigma}, $\mathcal E_r(F)\le|V(F)|$ whenever $\Delta(F)\le t-1$.

\begin{lem}\label{lem:capadd}
	For vertex-disjoint graphs $F_1,\dots,F_q$ one has $\mathcal E_r(F_1\cup\dots\cup F_q)=\sum_i\mathcal E_r(F_i)$. Moreover, for $m\ge r$,
	\[
	L_r(K_m)=\binom mr m^{-r/k},\qquad
	\mathcal E_r(K_m)=m\Bigl(\binom{m-1}{r-1}\Big/\binom{t-1}{r-1}\Bigr)^{\sigma_r},
	\]
	and both vanish when $m<r$. In particular $\mathcal E_r(K_t)=t$.
\end{lem}

\begin{proof}
	Write $W_i=\sum_{v\in V(F_i)}z_v^{k}$, so that $\sum_iW_i=1$. By homogeneity and H\"older's inequality with exponents $\sigma_r$ and $k/r$,
	\[
	S_r(F_1\cup\dots\cup F_q,z)\le\sum_iL_r(F_i)W_i^{r/k}\le\Bigl(\sum_iL_r(F_i)^{\sigma_r}\Bigr)^{1/\sigma_r}.
	\]
	If some $L_r(F_i)$ is positive, equality is attained by taking $W_i$ proportional to $L_r(F_i)^{\sigma_r}$ and a maximizing vector inside each $F_i$ of positive mass. For $K_m$, Maclaurin's inequality and the power-mean inequality give $e_r(z)\le\binom mr\bigl(\frac1m\sum_vz_v\bigr)^{r}\le\binom mr m^{-r/k}$, with equality for the constant vector. Since $\binom mr=\frac mr\binom{m-1}{r-1}$ and $r\sigma_r/k=\sigma_r-1$, raising to the power $\sigma_r$ gives the formula for $\mathcal E_r(K_m)$.
\end{proof}

\begin{lem}\label{lem:pack}
	Let $2\le r\le t$ and put $f(m)=m\binom{m-1}{r-1}^{\sigma_r}$ for $m\ge r$ and $f(m)=0$ for $0\le m<r$. If $m_1,\dots,m_c\le t$ are nonnegative integers with $\sum_im_i=bt+p$ and $0\le p<t$, then
	\[
	\sum_{i=1}^{c}f(m_i)\le bf(t)+f(p).
	\]
	If $bt+p\ge r$, equality requires $b$ parts equal to $t$ and, when $p\ge r$, one further part equal to $p$; when $p<r$, the remaining parts form an arbitrary partition of $p$.
\end{lem}

\begin{proof}
	Put $g(m)=\binom{m-1}{r-1}^{\sigma_r}$ for $m\ge r$ and $g(m)=0$ otherwise. The sequence $\binom{m-1}{r-1}$ (extended by zero) is nonnegative, nondecreasing and convex, since its successive differences are $0,\dots,0,1,\binom{r-1}{r-2},\binom{r}{r-2},\dots$; as $x\mapsto x^{\sigma_r}$ is convex and increasing, $g$ has the same three properties. With $\Delta g(m)=g(m+1)-g(m)$,
	\[
	\Delta^2f(m)=m\,\Delta^2g(m)+2\,\Delta g(m+1)\ge0,
	\]
	with strict inequality for $m\ge r-2$. If two positive parts are both smaller than $t$, moving a vertex from the smaller to the larger does not decrease the sum, and increases it strictly if one of the two parts has size at least $r$. Repeating leaves parts equal to $t$ and at most one smaller nonzero part, which proves the bound. In an equality partition at most one part of size at least $r$ is smaller than $t$, and such a part coexists with no other smaller positive part. If $bt+p\ge r$ the maximum is positive, so some part has size at least $r$. Parts smaller than $r$ have total smaller than $r$, since otherwise concentrating them creates a part of size $r$ and a strict gain. The equality conditions follow.
\end{proof}

\subsection*{Cliques of size at least three}

The next lemma is a weighted form of the closed-neighborhood counting of Chao and Dong~\cite{ChD}; the weights must be kept, and we give the proof.

\begin{lem}\label{lem:closednbhd}
	Let $F$ be a graph with nonnegative vertex weights $z$ and let $r\ge3$. For $v\in V(F)$ let $\mathcal T_v$ be the set of $r$-cliques of $F$ meeting $N_F[v]=N_F(v)\cup\{v\}$. Then
	\[
	\sum_{v\in V(F)}z_v\sum_{Q\in\mathcal T_v}\prod_{u\in Q}z_u\ \le\ \sum_{v\in V(F)}z_v\,e_r\bigl(z_{N_F[v]}\bigr).
	\]
\end{lem}

\begin{proof}
	We compare coefficients of monomials of degree $r+1$. A monomial with a repeated variable has the form $z_v^2\prod_{u\in B}z_u$ with $|B|=r-1$ and $v\notin B$. Its coefficient on the left is $1$ exactly when $B\cup\{v\}$ is an $r$-clique, and on the right it is $1$ whenever $B\subseteq N_F(v)$; so the right coefficient is at least the left one.
	
	Now fix an $(r+1)$-set $D$. The coefficient of $\prod_{u\in D}z_u$ on the right is the number of vertices of $D$ adjacent to all other vertices of $D$. On the left it is the number of $r$-cliques $Q\subset D$ for which the remaining vertex of $D$ has a neighbor in $Q$. If $F[D]$ has no $r$-clique, the latter is zero. Otherwise fix an $r$-clique $Q\subset D$, and let $\ell$ be the number of neighbors in $Q$ of the remaining vertex. The two coefficients are
	\[
	\begin{array}{c|cccc}
		\ell & 0 & 1\le\ell\le r-2 & r-1 & r\\\hline
		\text{right} & 0 & \ell & r-1 & r+1\\
		\text{left} & 0 & 1 & 2 & r+1
	\end{array}
	\]
	and since $r\ge3$ the right coefficient is always at least the left one.
\end{proof}

For $r=2$ the comparison fails: for a path on three vertices the two coefficients of the product of its three weights are $1$ and $2$. This is why the case $j=2$ below needs the minor conditions.

\begin{lem}\label{lem:completion}
	Let $r\ge3$, $\Delta(F)\le t-1$ and $z\ge0$. There is a graph $F'$ on $V(F)$, each component of which is complete of order at most $t$, with $S_r(F,z)\le S_r(F',z)$.
\end{lem}

\begin{proof}
	If $z=0$, take $F'$ edgeless. Otherwise Lemma~\ref{lem:closednbhd} provides a vertex $v$ with $z_v>0$ and $\sum_{Q\in\mathcal T_v}\prod_{u\in Q}z_u\le e_r(z_{N_F[v]})$. Delete all edges between $N_F[v]$ and its complement and make $N_F[v]$ a clique. The cliques destroyed are exactly those in $\mathcal T_v$, and the new cliques inside $N_F[v]$ have total weight $e_r(z_{N_F[v]})$, so the weighted clique sum does not decrease. The new clique has order at most $t$, and no other degree increases. Repeat the argument in the graph induced on $V(F)\setminus N_F[v]$, keeping the new clique fixed.
\end{proof}

\begin{thm}\label{thm:weighted3}
	Let $3\le r\le t$ and $k>r$, and write $h=bt+p$ with $0\le p<t$. Put $\gamma_p=p\bigl(\binom{p-1}{r-1}/\binom{t-1}{r-1}\bigr)^{\sigma_r}$ if $p\ge r$ and $\gamma_p=0$ otherwise. Every graph $F$ of order $h$ with $\Delta(F)\le t-1$ satisfies
	\[
	\mathcal E_r(F)\le bt+\gamma_p .
	\]
	If $p\ge r$, equality holds if and only if $F\cong bK_t\cup K_p$; if $p<r$, it holds if and only if $F\cong bK_t\cup F_0$ for some graph $F_0$ of order $p$.
\end{thm}

\begin{proof}
	Lemmas~\ref{lem:completion}, \ref{lem:capadd} and~\ref{lem:pack} give the bound, and the stated graphs attain it by Lemma~\ref{lem:capadd}. Suppose $h\ge r$ and equality holds, and let $z$ be a maximizing vector for $S_r(F,\cdot)$. Every vertex of an $r$-clique has positive coordinate: if such a clique had $q\ge1$ zero coordinates, replacing them by $\varepsilon$ would add a term of order $\varepsilon^{q}$ at a normalization cost $O(\varepsilon^k)$, and $q\le r<k$. A vertex in no $r$-clique has zero coordinate, since otherwise its mass could be moved to the clique terms. Let $U=\{v:z_v>0\}$.
	
	Suppose some $v\in U$ has two nonadjacent neighbors in $U$. As $v$ lies in an $r$-clique, $|N_F(v)\cap U|\ge r-1$, so these two neighbors extend to an $(r-1)$-set $B\subseteq N_F(v)\cap U$. The monomial $z_v^2\prod_{u\in B}z_u$ is positive and has coefficient $1$ on the right of Lemma~\ref{lem:closednbhd} and $0$ on the left, so that inequality is strict, and some vertex of positive weight gives a strict gain in the step of Lemma~\ref{lem:completion}. The resulting graph has order $h$ and maximum degree at most $t-1$ and exceeds the bound, a contradiction. Hence every neighborhood in $F[U]$ is complete, so each component of $F[U]$ is complete, of order between $r$ and $t$. Applying Lemma~\ref{lem:pack} to these orders, together with a part of size one for each vertex outside $U$, gives $U=V(F)$ and $F=bK_t\cup K_p$ if $p\ge r$, and $F[U]=bK_t$ if $p<r$. In the latter case every vertex of $U$ already has degree $t-1$, so no edge joins $U$ to its complement, which has order $p$.
\end{proof}

\begin{lem}\label{lem:defect3}
	Let $3\le r\le t$, $\Delta(F)\le t-1$, and suppose that no component of $F$ is isomorphic to $K_t$. Then
	\[
	\mathcal E_r(F)\le\zeta\,|V(F)|,\qquad \zeta=\Bigl(1-\frac{(r-1)(r-2)}{(t-1)(t-2)}\Bigr)^{\sigma_r}<1 .
	\]
\end{lem}

\begin{proof}
	Let $m_v$ be the number of $r$-cliques through $v$ and $M=\binom{t-1}{r-1}$. If $d_F(v)\le t-2$ then $m_v\le\binom{t-2}{r-1}=M-\binom{t-2}{r-2}$. If $d_F(v)=t-1$, the neighborhood of $v$ is not complete, for otherwise $N_F[v]$ would be a $K_t$ whose vertices have full degree and hence a component; a missing edge in $N_F(v)$ excludes at least $\binom{t-3}{r-3}$ of its $(r-1)$-subsets. In both cases $m_v\le M-\binom{t-3}{r-3}=M\bigl(1-\frac{(r-1)(r-2)}{(t-1)(t-2)}\bigr)$. By Lemma~\ref{lem:sigma}, $\mathcal E_r(F)\le\sum_v(m_v/M)^{\sigma_r}$, and the claim follows.
\end{proof}

\begin{cor}\label{cor:gap3}
	Let $3\le r\le t$, $k>r$ and $0\le p<t$. There is $\delta>0$, depending only on $k,r,t$, such that for every $b\ge0$ every graph $F$ of order $bt+p$ with $\Delta(F)\le t-1$ is either an equality graph of Theorem~\ref{thm:weighted3} or satisfies $\mathcal E_r(F)\le bt+\gamma_p-\delta$.
\end{cor}

\begin{proof}
	Write $F=qK_t\cup J$, where $J$ has no component $K_t$, and put $d=b-q$, so that $|V(J)|=dt+p$. By Lemmas~\ref{lem:capadd} and~\ref{lem:defect3},
	\[
	bt+\gamma_p-\mathcal E_r(F)=dt+\gamma_p-\mathcal E_r(J)\ge(1-\zeta)dt+\gamma_p-\zeta p ,
	\]
	which is at least $1$ once $d\ge d_0$. For $d<d_0$ there are finitely many graphs $J$, and by Theorem~\ref{thm:weighted3} each of them gives an equality graph or a positive difference. Let $\delta$ be the minimum of $1$ and these differences.
\end{proof}

\subsection*{Edges}

Throughout the rest of this section $2\le s\le t$, $k=s+1$, $\mathcal E=\mathcal E_2$ and
\[
\Delta_t=t-1,\qquad \sigma=\sigma_2=\frac{s+1}{s-1},\qquad E=\binom t2 .
\]
Thus $\mathcal E(F)=(2L_2(F)/\Delta_t)^{\sigma}$, and $\mathcal E(K_m)=m\bigl(\frac{m-1}{\Delta_t}\bigr)^{\sigma}$ for $m\ge1$. The arithmetic--geometric mean inequality followed by H\"older's inequality, as in Lemma~\ref{lem:sigma}, gives
\begin{equation}\label{eq:degest}
	\mathcal E(F)\le \Delta_t^{-\sigma}\sum_{v\in V(F)}d_F(v)^{\sigma}.
\end{equation}
If $F$ is connected and has an edge, a maximizing vector is positive, since a small positive coordinate at a vertex adjacent to the support gives a linear gain at a normalization cost of order $k$. Hence equality in~\eqref{eq:degest} for such $F$ forces equal coordinates along every edge and then equal degrees: it requires $F$ to be regular.

\begin{lem}\label{lem:capt1}
	Every admissible graph $F$ on $t+1$ vertices satisfies $\mathcal E(F)\le t$, with equality if and only if $F\cong K_t\cup K_1$.
\end{lem}

\begin{proof}
	Fix $x\ge0$. In each component of $\overline F$ take a spanning tree rooted at a vertex of minimum weight, and replace the tree by the star at that root, replacing $x_vx_{\mathrm{parent}(v)}$ by $x_vx_{\mathrm{root}}$ for every other vertex $v$; then delete the remaining edges of $\overline F$. The weighted edge sum of the complement does not increase, so that of $F$ does not decrease, and by Lemma~\ref{lem:tplus1} the new graph $F^*$ is still admissible, since the component orders of the complement are unchanged. Thus $\mathcal E(F)\le\mathcal E(F^*)$, where $\overline{F^*}$ is a forest of $c$ stars with $a_1,\dots,a_c\ge s$ leaves and $\sum_ia_i=t+1-c$.
	
	If $c=1$ then $F^*=K_t\cup K_1$ and $\mathcal E(F^*)=t$. Let $c\ge2$. Then $F^*$ has $t+1-c$ vertices of degree $\Delta_t$ and $c$ vertices of degrees $t-a_i$. With $u_i=(t-a_i)/\Delta_t$ we have $u_i\le1-\frac{s-1}{\Delta_t}$ and $\sum_iu_i=(c-1)(1+\frac2\Delta_t)$, so by~\eqref{eq:degest}
	\[
	\mathcal E(F^*)\le t+1-c+\sum_iu_i^{\sigma}\le t+1-c+\Bigl(1-\frac{s-1}{\Delta_t}\Bigr)^{\frac{2}{s-1}}\Bigl(1+\frac2\Delta_t\Bigr)(c-1)<t,
	\]
	because $\frac{2}{s-1}\log\bigl(1-\frac{s-1}{\Delta_t}\bigr)+\log\bigl(1+\frac2\Delta_t\bigr)<-\frac2\Delta_t+\frac2\Delta_t=0$; here $c\ge2$ gives $s-1<\Delta_t$.
	
	If $\mathcal E(F)=t$, then $c=1$ and a maximizing vector of $F$ also maximizes $K_t\cup K_1$: it vanishes at the star center and is constant and positive elsewhere. Equality in the complement comparison makes these $t$ vertices a clique of $F$; their degrees are then $\Delta_t$, so the remaining vertex is isolated.
\end{proof}

\begin{lem}\label{lem:capP}
	$\mathcal E(\overline P)=10\bigl(\frac67\bigr)^{\sigma}$ when $t=8$. This is smaller than $\mathcal E(K_8\cup K_2)$ for $s\le7$ and larger for $s=8$.
\end{lem}

\begin{proof}
	The graph $\overline P$ is $6$-regular, so the constant vector attains~\eqref{eq:degest}. Put $f(\sigma)=10(\frac67)^{\sigma}-8-2(\frac17)^{\sigma}$, the difference of the two capacities. Then $f'(\sigma)=2\cdot7^{-\sigma}\bigl(\log7-5\cdot6^{\sigma}\log\frac76\bigr)<0$ for $\sigma\ge1$, since $30\log\frac76>\frac{30}{7}>\log7$. For $s\le7$ we have $\sigma\ge\frac43$, and $6^{1/3}<\frac{20}{11}$, $7^{1/3}>\frac{153}{80}$ give
	\[
	10\cdot6^{4/3}<\tfrac{1200}{11}<\tfrac{1091}{10}<8\cdot7^{4/3}+2 ,
	\]
	that is, $f(\frac43)<0$. For $s=8$, $\sigma=\frac97$, and $6^9>10^7$, $4^7<7^5$ give $10\cdot6^{9/7}>100>8\cdot7^{9/7}+2$, that is, $f(\frac97)>0$.
\end{proof}

\begin{lem}\label{lem:capt2}
	Let $t\ge3$ and let $F$ be connected and admissible on $t+2$ vertices. Then $\mathcal E(F)<\mathcal E(K_t\cup K_2)$, unless $t=8$ and $F\cong\overline P$.
\end{lem}

\begin{proof}
	Note that $\mathcal E(K_t\cup K_2)=t+2\Delta_t^{-\sigma}$. By Lemma~\ref{lem:djs}, $e(F)\le E+2$. If $e(F)\le E+1$, the degrees lie in $[1,\Delta_t]$ with sum at most $tD+2$, and by convexity $\sum_vd_F(v)^{\sigma}\le tD^{\sigma}+2$; equality would force the degree pattern $(\Delta_t^t,1^2)$, which is not regular, so~\eqref{eq:degest} is strict.
	
	Let $e(F)=E+2$, and apply Lemma~\ref{lem:tplus2}. For $J_{0,e,f}$ the degrees are $\Delta_t$ (on the clique), $2$, $e+1$ and $f+1$, with $e,f\ge s$ and $e+f=t-1$; by convexity $2^{\sigma}+(e+1)^{\sigma}+(f+1)^{\sigma}\le2^{\sigma}+(s+1)^{\sigma}+(t-s)^{\sigma}$, and it suffices to show that this is less than $\Delta_t^{\sigma}+2$. The difference $\Delta_t^{\sigma}+2-2^{\sigma}-(s+1)^{\sigma}-(t-s)^{\sigma}$ increases with $t$, and $t\ge2s+1$, so it suffices to prove $(2s)^{\sigma}-2(s+1)^{\sigma}>2^{\sigma}-2$. This is $10>6$ for $s=2$ and $4>2$ for $s=3$. For $s\ge4$ put $a=s-1$ and $\omega=2/a$, so that $\sigma=1+\omega$ and $(1+\omega)a=a+2$. Then
	\[
	\begin{aligned}
	(2a+2)^{1+\omega}-2(a+2)^{1+\omega}&=(1+\omega)\int_0^a\bigl((a+2+x)^\omega-(a+2)^\omega\bigr)\,dx\\
	&\ge(a+2)(2a+2)^{\omega-1}>\tfrac12\,8^{\omega}>2^{1+\omega}-2,
	\end{aligned}
	\]
	using $(a+2+x)^\omega-(a+2)^\omega\ge \omega(2a+2)^{\omega-1}x$ for $0\le x\le a$, then $a\ge3$, and finally $y^3-4y+4>0$ for $y=2^\omega>0$.
	
	For $J'_{a_1,a_2,a_3}$ the degrees are $\Delta_t$ (on the $t-2$ clique vertices), $3$ and $\Delta_t-a_i$, with $a_i\ge s$ and $\sum_ia_i=t-2$, so $\Delta_t\ge3s+1$. By convexity $3^{\sigma}+\sum_i(\Delta_t-a_i)^{\sigma}\le3^{\sigma}+2(\Delta_t-s)^{\sigma}+(2s+1)^{\sigma}$, and it suffices to show that this is less than $2\Delta_t^{\sigma}+2$. Since $\Delta_t^{\sigma}-(\Delta_t-s)^{\sigma}$ increases with $\Delta_t$, it suffices to take $\Delta_t=3s+1$, that is, to show $2(3s+1)^{\sigma}-3(2s+1)^{\sigma}>3^{\sigma}$. With $a=s-1$ and $B=2a+3$, convexity gives
	\[
	2(B+a+1)^{\sigma}-3B^{\sigma}\ge B^{\sigma-1}\bigl(2\sigma(a+1)-B\bigr)=\Bigl(3+\frac4a\Bigr)B^{2/a}>3^{1+2/a}.
	\]
	In both families, dividing by $\Delta_t^{\sigma}$ and using~\eqref{eq:degest} gives $\mathcal E(F)<t+2\Delta_t^{-\sigma}$. The remaining family is $\overline P$ with $t=8$.
\end{proof}

\begin{lem}\label{lem:pendant}
	Let $t\ge4$ and let $F$ be a connected $K_{1,t}$-minor-free graph on $m\ge t+2$ vertices with $\ell$ vertices of degree one. Then
	\[
	e(F)\le E+m-t-\lfloor\ell/2\rfloor .
	\]
\end{lem}

\begin{proof}
	A connected graph has a $K_{1,t}$ minor if and only if it has a spanning tree with at least $t$ leaves. We show that adding an edge $uv$ between two vertices of degree one, with neighbors $a$ and $b$, does not increase the largest number of leaves of a spanning tree. Let $T$ be a spanning tree of the new graph containing $uv$. If $ua,vb\in E(T)$, delete $uv$ and reconnect the two components by an edge of $F$, which avoids $u$ and $v$; the vertices $u,v$ become leaves and at most two leaves are lost. If $ua\notin E(T)$, then $u$ is a leaf of $T$ and $vb\in E(T)$; replace $uv$ by $ua$, so that $v$ becomes a leaf and at most one leaf is lost. Adding a matching of $\lfloor\ell/2\rfloor$ such edges therefore keeps the graph $K_{1,t}$-minor-free, and Lemma~\ref{lem:djs} gives the bound.
\end{proof}

\begin{lem}\label{lem:pendcount}
	Let $t\ge4$, $3\le r<t$, and let $F$ be connected and admissible with $|V(F)|=t+r$ and $e(F)=E+r-q_F$. If $F$ has $\ell$ vertices of degree one, then $\ell-2q_F\le r-3$.
\end{lem}

\begin{proof}
	By Lemma~\ref{lem:djs}, $q_F\ge0$. Suppose $\ell-2q_F\ge r-2$, and delete the $\ell$ vertices of degree one to obtain a nonempty connected graph $F_0$.
	
	If $q_F\ge1$, then $\ell\ge r$; write $\ell=r+d$, so that $|V(F_0)|=t-d$ and $e(F_0)=E-d-q_F$. The bound $e(F_0)\le\binom{t-d}2$ gives $d(2t-d-3)\le2q_F\le d+2$, hence $d(2t-d-4)\le2$. For $1\le d\le t-1$ the left side is at least $\min\{2t-5,(t-1)(t-3)\}>2$, so $d=0$ and $q_F=1$. Then $F_0=K_t-e$, and the vertices of degree one can only be attached to the two ends of the missing edge, at most one each, because $\Delta(F)\le t-1$. So $\ell\le2<r$, a contradiction.
	
	If $q_F=0$, put $u=|V(F_0)|\le t+2$, so that $e(F_0)=E+u-t$; the bound $e(F_0)\le\binom u2$ forces $u\ge t$. If $u=t$ then $F_0=K_t$, which cannot receive a pendant neighbor. If $u=t+1$ then $\overline{F_0}$ has $t-1$ edges and is disconnected, and by Lemma~\ref{lem:leaf} a pendant neighbor creates a $K_{1,t}$ minor. If $u=t+2$ then $F_0$ is admissible with $E+2$ edges, and Lemma~\ref{lem:leaf} again forbids a pendant neighbor. Since $\ell\ge r-2\ge1$, each case is impossible.
\end{proof}

\begin{lem}\label{lem:large}
	Let $t\ge4$ and let $F$ be connected and admissible on $m\ge t+3$ vertices, where $m=bt+p$ with $0\le p<t$. Then $\mathcal E(F)<\mathcal E(bK_t\cup K_p)$, and the difference is at least a positive constant depending only on $s$ and $t$.
\end{lem}

\begin{proof}
	First let $m=t+r$ with $3\le r<t$, write $e(F)=E+r-q_F$, and let $\ell$ be the number of vertices of degree one. We claim that the nonincreasing degree sequence of $F$ is weakly majorized by $(\Delta_t^{t},(r-1)^{r})$. Prefixes of length at most $t$ are handled by $\Delta(F)\le \Delta_t$. If $\ell\ge r$, at most $t$ vertices have degree at least two and the comparison holds termwise. Otherwise, for a prefix of length $t+a$ with $1\le a\le r$, the remaining $r-a$ vertices have degree sum at least $2(r-a)-\min\{\ell,r-a\}$, so the prefix sum is at most
	\[
	tD+2a-2q_F+\ell\le tD+2a+r-3\le tD+a(r-1),
	\]
	by Lemma~\ref{lem:pendcount}. Since $x^{\sigma}$ is increasing and strictly convex, $\sum_vd_F(v)^{\sigma}\le tD^{\sigma}+r(r-1)^{\sigma}$, which is $\Delta_t^{\sigma}\mathcal E(K_t\cup K_r)$. For $r\ge4$ the degree sum $tD+2r-2q_F$ is strictly smaller than $tD+r(r-1)$, so the inequality is strict; for $r=3$ equality would force the degree sequence $(\Delta_t^t,2^3)$, which is not regular, so~\eqref{eq:degest} is strict. As there are finitely many such $F$, the gaps have a positive minimum.
	
	Now let $m\ge2t$, and put $A_\sigma=\frac{\Delta_t^{\sigma}-2^{\sigma}}{\Delta_t-2}$ and $B_\sigma=A_\sigma-(2^{\sigma}-1)$, which is positive by convexity. With $\varepsilon=\ell\bmod2$, Lemma~\ref{lem:pendant} gives
	\[
	\sum_{d_F(v)\ge2}(d_F(v)-2)=2e(F)-2m+\ell\le t(\Delta_t-2)+\varepsilon .
	\]
	Bounding $d^{\sigma}$ by the chord of $x^{\sigma}$ on $[2,\Delta_t]$ at the vertices of degree at least two, and by $1$ at the others,
	\[
	\sum_vd_F(v)^{\sigma}\le tD^{\sigma}+(m-t)2^{\sigma}+\varepsilon A_\sigma-\ell(2^{\sigma}-1)\le tD^{\sigma}+(m-t)2^{\sigma}+B_\sigma .
	\]
	Since $p\bigl((p-1)^{\sigma}-2^{\sigma}\bigr)\ge-2(2^{\sigma}-1)$ for $0\le p<t$, and $b\ge2$, the sum $btD^{\sigma}+p(p-1)^{\sigma}=\Delta_t^{\sigma}\mathcal E(bK_t\cup K_p)$ exceeds this bound by at least
	\[
	t(\Delta_t^{\sigma}-2^{\sigma})-A_\sigma-(2^{\sigma}-1)=\Bigl(t-\frac1{t-3}\Bigr)(\Delta_t^{\sigma}-2^{\sigma})-2^{\sigma}+1\ge3\cdot3^{\sigma}-4\cdot2^{\sigma}+1>0 ,
	\]
	where the last expression equals $2$ at $\sigma=1$ and increases. By~\eqref{eq:degest} the gap is at least this constant divided by $\Delta_t^{\sigma}$.
\end{proof}

\begin{lem}\label{lem:packK}
	Among disjoint unions of cliques of order at most $t$ on $bt+p$ vertices, $0\le p<t$, the unique maximizer of $\mathcal E$ is $bK_t\cup K_p$, and every other such union has capacity smaller by a positive constant depending only on $k$ and $t$.
\end{lem}

\begin{proof}
	The numbers $c_m=\mathcal E(K_m)$, with $c_0=0$, have strictly increasing differences, since $x(x-1)^{\sigma}$ is strictly convex for $x>1$ and $c_1-c_0=0$. Hence moving a vertex from a component of order $a$ to one of order $b'$ with $1\le a\le b'<t$ strictly increases $c_a+c_{b'}$, and the gains form a finite set.
\end{proof}

\begin{lem}\label{lem:packP}
	Let $(s,t)=(8,8)$. Among disjoint unions of cliques of order at most $8$ and copies of $\overline P$ on $8b+p$ vertices, the unique maximizer of $\mathcal E$ is $(b-1)K_8\cup\overline P$ if $p=2$ and $b\ge1$, and $bK_8\cup K_p$ otherwise; every other union has capacity smaller by a positive constant.
\end{lem}

\begin{proof}
	Here $\sigma=\frac97$. Put $c_m=\mathcal E(K_m)$ and $\beta=\mathcal E(\overline P)$. Then
	\[
	8+c_2<\beta<\tfrac{33}{4},\qquad c_3>\tfrac12 .
	\]
	The first inequality is Lemma~\ref{lem:capP}; the second is equivalent to $40^7\cdot6^9<33^7\cdot7^9$, and the third to $6^7\cdot2^9>7^9$. Two copies of $\overline P$ lose to $2K_8\cup K_4$, since $2\beta<\frac{33}2<16+c_4$. If a copy of $\overline P$ occurs together with a clique $K_m$, $1\le m\le5$, then replacing $\overline P\cup K_m$ by $K_8\cup K_{m+2}$ gains $8+c_{m+2}-c_m-\beta\ge8+c_3-\beta>0$ by convexity; for $m=6,7$, the bound $\beta+c_m<\frac{33}4+7<16$ shows that the complete packing on $10+m$ vertices is better. Combined with Lemma~\ref{lem:packK}, a maximizer contains $\overline P$ only if $p=2$ and all other components are $K_8$, and then $\overline P$ wins by the first inequality.
\end{proof}

\begin{lem}\label{lem:t23}
	Let $2\le s\le t\le3$, and let $H$ be admissible on $bt+p$ vertices, $0\le p<t$. Then $\mathcal E(H)\le\mathcal E(bK_t\cup K_p)$, with equality if and only if $H\cong bK_t\cup K_p$, and every other admissible $H$ has capacity smaller by a positive constant $\delta_{k,t}$.
\end{lem}

\begin{proof}
	For $t=2$, admissible graphs have maximum degree one, so their components are $K_1$ and $K_2$, of capacities $0$ and $2$, and the claim is immediate with $\delta=2$.
	
	Let $t=3$, so $k\in\{3,4\}$ and $\sigma\in\{3,2\}$. Admissibility excludes $K_{1,3}$ and $K_{2,2}$ minors, so every component of $H$ is a path or a triangle. Since $\Delta_t=2$ and $2\sigma=k(\sigma-1)$, optimizing the scalar multiple of a unit vector gives
	\[
	\mathcal E(F)=\max_{z\ge0}\Bigl\{\sigma\sum_{uv\in E(F)}z_uz_v-(\sigma-1)\sum_{v}z_v^{k}\Bigr\}.
	\]
	For $d\in\{1,2\}$ put $\phi_d(z)=(\frac d2)^{\sigma}-\frac{\sigma d}2z^2+(\sigma-1)z^k$. Then $\phi_d\ge0$, with equality only at $z=(d/2)^{1/(k-2)}$, so
	\[
	\varepsilon_k=\min_{u,v\ge0}\Bigl\{\phi_1(u)+\phi_2(v)+\tfrac\sigma2(u-v)^2\Bigr\}>0 .
	\]
	For the path $P_m$, $m\ge3$, the objective above equals $\sum_v(\frac{d(v)}2)^{\sigma}-\sum_v\phi_{d(v)}(z_v)-\frac\sigma2\sum_{uv}(z_u-z_v)^2$, and keeping only the terms of one end vertex and its neighbor gives $\mathcal E(P_m)\le m-2+2^{1-\sigma}-\varepsilon_k$. Put $c=2^{1-\sigma}\in\{\frac14,\frac12\}$. The packed union of triangles and a residual clique on $m$ vertices has capacity $m$, $m-1$ or $m-2+c$ according as $m\equiv0,1,2\pmod3$, so replacing a path of order at least $3$ gains at least $\varepsilon_k$. It remains to compare unions of $K_1$, $K_2$ and $K_3$: with $a_1$ copies of $K_1$ and $a_2$ of $K_2$ the capacity is $bt+p-a_1-(2-c)a_2$, and the largest value is attained only by $(a_1,a_2)=(0,0),(1,0),(0,1)$ for $p=0,1,2$, respectively. Every other value is smaller by a positive multiple of $c$. Take $\delta=\min\{\varepsilon_k,c\}$.
\end{proof}

\begin{thm}\label{thm:weighted2}
	Let $2\le s\le t$ and $k=s+1$, write $h=bt+p$ with $0\le p<t$, and let $H^*_h=(b-1)K_8\cup\overline P$ if $(s,t)=(8,8)$, $p=2$ and $b\ge1$, and $H^*_h=bK_t\cup K_p$ otherwise. Every admissible graph $H$ on $h$ vertices satisfies $\mathcal E(H)\le\mathcal E(H^*_h)$, with equality if and only if $H\cong H^*_h$. Moreover there is $\delta=\delta(s,t)>0$, independent of $h$, such that $\mathcal E(H)\le\mathcal E(H^*_h)-\delta$ for every other admissible $H$.
\end{thm}

\begin{proof}
	For $t\le3$ this is Lemma~\ref{lem:t23}; let $t\ge4$. A connected component of order at most $t$ can be completed without leaving the admissible class, and this strictly increases the capacity unless it is already complete. Lemmas~\ref{lem:capt1} and~\ref{lem:capt2} replace every component of order $t+1$ or $t+2$ by $K_t\cup K_1$ or $K_t\cup K_2$ with a strict gain, except $\overline P$ when $(s,t)=(8,8)$; by Lemma~\ref{lem:capP} and Lemma~\ref{lem:exadm}, $\overline P$ is admissible and loses to $K_8\cup K_2$ when $t=8$ and $s\le7$. Lemma~\ref{lem:large} replaces every larger component. All replacements are admissible, and capacity is additive by Lemma~\ref{lem:capadd}, so the problem reduces to Lemmas~\ref{lem:packK} and~\ref{lem:packP}. The strict gains for components of order at most $2t-1$ form a finite set, the larger ones are bounded below by Lemma~\ref{lem:large}, and the packing lemmas supply a uniform gain for every other union of retained components; the minimum of these constants is $\delta$.
\end{proof}

\section{All residues when $k>s$}\label{sec:kgs}

Throughout this section $2\le s\le\min(k-1,t)$, so $j=k-s+1\ge2$, and $t\ge j$. Put
\[
\alpha=\frac{s-1}{k},\qquad K=\frac1t\binom tj=C_j,\qquad \mathcal E=\mathcal E_j ,
\]
so that $1-\alpha=j/k$ and $\mathcal E(K_t)=t$. By homogeneity, if $\sum_vz_v^k=W$ then
\begin{equation}\label{eq:capscale}
	S_j(F,z)\le K\,\mathcal E(F)^{\alpha}\,W^{j/k}.
\end{equation}
Write $h=n-s+1=bt+p$ with $0\le p<t$, and let $H^*=bK_t\cup K_p$ if $j\ge3$, and $H^*=H^*_h$ as in Theorem~\ref{thm:weighted2} if $j=2$. In both cases $H^*=BK_t\cup Q$ with $Bt+|V(Q)|=h$, where $Q$ belongs to a finite list depending only on $k,s,t,p$, and $H^*$ is admissible. Put $\mathcal E_*=\mathcal E(H^*)=Bt+\mathcal E(Q)$. Theorem~\ref{thm:weighted3}, Corollary~\ref{cor:gap3} and Theorem~\ref{thm:weighted2} give $\delta>0$, independent of $b$, such that every admissible graph $F$ on $h$ vertices satisfies
\begin{equation}\label{eq:unifgap}
	\mathcal E(F)\le\mathcal E_*,\qquad \mathcal E(F)\le\mathcal E_*-\delta\ \text{ unless }\ \CC_j(F)\cong\CC_j(H^*).
\end{equation}

\begin{lem}\label{lem:struct}
	Let $G$ be an extremal graph with normalized Perron vector $x$, let $A$ be the set of its $s-1$ largest coordinates, $T=N(A)$ and $R=V(G)\setminus(A\cup T)$. For all sufficiently large $n$, $A$ induces a clique, $|R|=O(1)$, $x_a\to k^{-1/k}$ for $a\in A$, $\max_{v\notin A}x_v=O(n^{-1/k})$, $\Delta(G[T])\le t-1$, $\Delta(G-A)=O(1)$ and $\rho(\CC_k(G))=\Theta(n^{\alpha})$.
\end{lem}

\begin{proof}
	Lemma~\ref{lem:Acliquelow}, Proposition~\ref{prop:Rbdd}, Theorems~\ref{thm:exact} and~\ref{thm:stab} and Corollary~\ref{cor:deg} give everything except the last two bounds on degrees and coordinates. Every $v\notin A$ has at most $t-1$ neighbors in $T$ by Lemma~\ref{lem:common} applied to $A\cup\{v\}$, so $\Delta(G-A)\le t-1+|R|=O(1)$. Let $\eta=\max_{v\notin A}x_v$ be attained at $v$. Only boundedly many $k$-cliques contain $v$, and each has at least $j-1$ further vertices outside $A$, so the eigenequation gives $\rho\,\eta^{k-1}\le C\eta^{j-1}$, that is, $\eta^{s-1}\le C/\rho=O(n^{-(s-1)/k})$.
\end{proof}

\begin{lem}\label{lem:refined}
	With the notation of Lemma~\ref{lem:struct}, put $H=G[T]$, $W_0=\sum_{v\notin A}x_v^k$ and $\pi_A=\prod_{a\in A}x_a$. Then
	\[
	\rho\bigl(\CC_k(K_{s-1}\vee H^*)\bigr)-\rho(\CC_k(G))\ \ge\ k\,\pi_AK\,W_0^{j/k}\bigl(\mathcal E_*^{\alpha}-\mathcal E(H)^{\alpha}\bigr)-O\bigl(n^{-(j+1)/k}\bigr).
	\]
\end{lem}

\begin{proof}
	A $k$-clique meeting $R$ contains at most $s-2$ vertices of $A$, hence at least $j+1$ vertices outside $A$, and only boundedly many $k$-cliques meet $R$; by Lemma~\ref{lem:struct} their total weight is $O(n^{-(j+1)/k})$. Every other $k$-clique consists of $k-\ell$ vertices of $A$ and an $\ell$-clique of $H$, with $j\le\ell\le\min(k,t)$, and for fixed $\ell$ these contribute $e_{k-\ell}(x_A)S_\ell(H,x)$. For $\ell=j$, the coefficient is $\pi_A$ and~\eqref{eq:capscale} bounds $S_j(H,x)$ by $K\mathcal E(H)^{\alpha}W_0^{j/k}$. For $\ell>j$, Lemma~\ref{lem:sigma} (and, for $\ell=k$, the arithmetic--geometric mean inequality alone) gives $S_\ell(H,x)\le\frac1t\binom t\ell W_0^{\ell/k}h^{1-\ell/k}$.
	
	On $K_{s-1}\vee H^*$ keep the coordinates $x_A$. If $\mathcal E(Q)>0$, let $z$ maximize $S_j(Q,\cdot)$ under $\sum_qz_q^k=\mathcal E(Q)$, so that $S_j(Q,z)=K\mathcal E(Q)$; otherwise let $z=0$. Put $v=(W_0/\mathcal E_*)^{1/k}$, and give coordinate $v$ to every vertex of the $B$ copies of $K_t$ and $vz_q$ to $q\in V(Q)$. The light mass is $v^k(Bt+\mathcal E(Q))=W_0$, so the vector is normalized, and its light $j$-clique sum is $Kv^j\mathcal E_*=KW_0^{j/k}\mathcal E_*^{\alpha}$. For $\ell>j$ its light $\ell$-clique sum is $v^\ell\bigl(B\binom t\ell+O(1)\bigr)$; since $Bt=h-O(1)$ and $\mathcal E_*=h-O(1)$, this is $\frac1t\binom t\ell W_0^{\ell/k}h^{1-\ell/k}+O(n^{-\ell/k})$. Multiplying by the bounded coefficients $e_{k-\ell}(x_A)$, summing, and comparing with the upper bounds for $G$ gives the claim, by~\eqref{eq:var}.
\end{proof}

\begin{thm}\label{thm:kgs}
	Let $2\le s\le\min(k-1,t)$ and $t\ge j$, and write $h=n-s+1=bt+p$ with $0\le p<t$. For all sufficiently large $n$ the unique hypergraph of maximum spectral radius in $\HH^{(k)}_{K_{s,t}}(n)$ is $\CC_k(K_{s-1}\vee H^*)$, where $H^*=(b-1)K_8\cup\overline P$ if $(k,s,t,p)=(9,8,8,2)$ and $H^*=bK_t\cup K_p$ otherwise. If $p=0$ or $p\ge j$, then $K_{s-1}\vee H^*$ is the unique extremal graph. If $0<p<j$, then in every extremal graph the $p$ vertices that lie neither in the dominating clique nor in a copy of $K_t$ belong to no $k$-clique, and the extremal graph is not unique.
\end{thm}

\begin{proof}
	Let $G$ be extremal, with the notation of Lemma~\ref{lem:struct}. The graph $H=G[T]$ is admissible because $G[A\cup T]=K_{s-1}\vee H$, and adding $|R|$ isolated vertices gives an admissible graph $\widetilde H$ on $h$ vertices with $\mathcal E(\widetilde H)=\mathcal E(H)$. If $\CC_j(\widetilde H)\not\cong\CC_j(H^*)$, then $\mathcal E(H)\le\mathcal E_*-\delta$ by~\eqref{eq:unifgap}, and by concavity
	\[
	\mathcal E_*^{\alpha}-\mathcal E(H)^{\alpha}\ge\alpha\mathcal E_*^{\alpha-1}\delta\ge\alpha\delta\,h^{-j/k}.
	\]
	Since $\pi_A\to k^{-(s-1)/k}$ and $W_0\to j/k$, Lemma~\ref{lem:refined} gives $\rho(\CC_k(K_{s-1}\vee H^*))-\rho(\CC_k(G))\ge cn^{-j/k}-O(n^{-(j+1)/k})>0$, contradicting extremality. Hence $\CC_j(\widetilde H)\cong\CC_j(H^*)$.
	
	If $p=0$ or $p\ge j$, then $\CC_j(H^*)$ has no isolated vertex, so $R=\emptyset$. For $j=2$ this gives $H\cong H^*$ directly. For $j\ge3$, every pair of vertices in a copy of $K_t$ or in the residual $K_p$ lies in a $j$-clique, so these sets are cliques of $H$; their vertices in the copies of $K_t$ have degree $t-1$, so no other edges exist, and $H\cong H^*$. In both cases $G=K_{s-1}\vee H^*$.
	
	Let $0<p<j$. Then $H$ contains $b$ copies of $K_t$, which are components of $H$, and $p$ vertices of $G$ lie outside $A$ and these copies. A $k$-clique through one of these $p$ vertices cannot lie in $A$ together with them, since $s-1+p<k$, so it meets exactly one copy $Y$ of $K_t$. Then $G[A\cup V(Y)]=K_{s+t-1}$, the clique has at least $k-p>s-1$ vertices there, and its part outside $A\cup V(Y)$ is nonempty and connected, contradicting Lemma~\ref{lem:Kst}. So these $p$ vertices lie in no $k$-clique, and $\CC_k(G)\cong\CC_k(K_{s-1}\vee H^*)$. In particular, deleting the edges at these $p$ vertices in $K_{s-1}\vee H^*$ gives a different extremal graph.
	
	Finally, if $\HH$ is any maximizing hypergraph, then $\CC_k(\shad\HH)$ is a maximizing clique hypergraph by Lemma~\ref{lem:reduce}, hence isomorphic to $\CC_k(K_{s-1}\vee H^*)$. The vertices covered by its edges span a connected subhypergraph, so Observation~\ref{obs:strict} gives $\HH=\CC_k(\shad\HH)$.
\end{proof}

\section{Concluding remarks}\label{sec:remark}

Together, Theorems~\ref{thm:Kr}, \ref{thm:sgek} and~\ref{thm:kgs} determine, for all sufficiently large $n$, the $k$-uniform hypergraphs of maximum spectral radius whose shadow has no $K_t$ minor, or no $K_{s,t}$ minor with $2\le s\le t$, in every case in which the problem is nontrivial. In each case the extremal hypergraph is the clique hypergraph of a join of a clique with a light part, and the light part is decided by a single comparison that depends on $j=k-s+1$.

When $j\le1$, every $k$-clique of the join can use a single light vertex, the light coordinates are asymptotically equal, and the light parts are compared first by their numbers of edges. This is the comparison that governs the adjacency matrix, and the answer is the one of Zhai and Lin. When $j\ge2$, a regime with no counterpart for graphs, the light coordinates are no longer equal and the comparison is by the capacity $\mathcal E_j$. For $j\ge3$ this capacity is controlled by the maximum degree alone, through Lemma~\ref{lem:closednbhd}, and complete blocks are always optimal. For $j=2$ the minor conditions are needed, and they leave exactly one exceptional light part, the complement of the Petersen graph at $(k,s,t)=(9,8,8)$. Thus the answer for $j\ge2$ differs from the graph case: the exceptional components $F_{s,t}$ and $D_{s,t}$ of Zhai and Lin never occur.

	\section*{Declaration of generative AI and AI-assisted technologies in the manuscript preparation process}
	
	During the preparation of this work, the authors used Claude (Anthropic) for language refinement, technical editing, and computational verification of examples. The authors reviewed and edited the output as needed and take full responsibility for the content of the published article.

\end{document}